\documentclass[10pt]{article}

\usepackage{lipsum}
\usepackage{amsfonts}
\usepackage{graphicx}
\usepackage{graphics}
\usepackage{epstopdf}
\usepackage{algorithmic}
\usepackage{amsopn}
\usepackage{amsmath}
\usepackage{amssymb}
\usepackage{bbm} % For \mathbb{1}

\usepackage[section]{algorithm}

\usepackage{geometry}

\usepackage[hidelinks]{hyperref}
\usepackage{enumerate}
\usepackage[normalem]{ulem}

\usepackage{enumitem}  % the have lists numbered as (i), (ii), ...

\usepackage{color}

\newcommand{\bm}[1]{{\boldsymbol{#1}}}
\newcommand{\bmat}[1]{{\begin{bmatrix}#1\end{bmatrix}}}
\newcommand{\mat}[1]{\bm{#1}}

\usepackage{amsthm}

\theoremstyle{plain}

\newtheorem{thm}{Theorem}[section]

\newtheorem{lem}[thm]{Lemma}

\theoremstyle{definition}
\newtheorem{defn}[thm]{Definition}

\theoremstyle{remark}
\newtheorem{rmrk}[thm]{Remark}

\DefineNamedColor{named}{JungleGreen} {cmyk}{0.99,0,0.52,0}

\DeclareMathOperator{\supp}{supp}

\DeclareMathOperator{\Expe}{\mathbb{E}}

\usepackage[font=small]{caption}

\begin{document}

\title{Correcting for unknown errors in sparse high-dimensional function approximation}


%\author{Ben Adcock, Anyi Bao, Simone Brugiapaglia}

\author{Ben Adcock\footnote{Simon Fraser University, Burnaby, BC, Canada. e-mail: ben\_adcock@sfu.ca}, Anyi Bao\footnote{University of British Columbia, Vancouver, BC, Canada. e-mail: abao@math.ubc.ca},  and Simone Brugiapaglia\footnote{Simon Fraser University, Burnaby, BC, Canada. e-mail: simone\_brugiapaglia@sfu.ca }}


%\institute{Ben Adcock \at
%              Simon Fraser University, Burnaby, BC, Canada \\
%              \email{ben\_adcock@sfu.ca}    
%           Anyi Bao \at
%              University of British Columbia, Vancouver, BC, Canada\\ 
%              \email{abao@math.ubc.ca}
%           Simone Brugiapaglia \at
%              Simon Fraser University, Burnaby, BC, Canada\\ 
%              \email{simone\_brugiapaglia@sfu.ca}
%}


\maketitle

\begin{abstract}
We consider sparsity-based techniques for the approximation of high-dimensional functions from random pointwise evaluations. To date, almost all the works published in this field contain some \emph{a priori} assumptions about the error corrupting the samples that are hard to verify in practice. In this paper, we instead  focus on the scenario where the error is unknown. We study the performance of four sparsity-promoting optimization problems: weighted quadratically-constrained basis pursuit, weighted LASSO, weighted square-root LASSO, and  weighted LAD-LASSO. From the theoretical perspective, we prove uniform recovery guarantees for these decoders, deriving recipes for the optimal choice of the respective tuning parameters. On the numerical side, we compare them in the pure function approximation case and in applications to uncertainty quantification of ODEs and PDEs with random inputs.  Our main conclusion is that the lesser-known square-root LASSO is better suited for high-dimensional approximation than the other procedures in the case of bounded noise, since it avoids (both theoretically and numerically) the need for parameter tuning.
%The comparison turns out to be in favor of square-root LASSO, as opposed to the commonly-used quadratically-constrained basis pursuit.
\end{abstract}

%Colors: \PURP{Simone's edits}, \RED{Simone's comments}, \ORNG{Ben's edits}, \GR{Ben's comments}.\\

\section{Introduction}

\label{sec:intro}

Sparse regularization has been recently proved to be a useful tool for the approximation of functions defined over $d$-dimensional domains, due to its ability to lessen the curse of dimensionality when combined with compressed sensing principles. This approach is based on three main elements: sparse approximation of the function with respect to orthogonal polynomials, random pointwise sampling, and sparse recovery via a (weighted) $\ell^1$ minimization decoder. Combining these three ingredients, it is possible to construct approximations
%approximate a function defined over a $d$-dimensional domain 
 using a number of pointwise evaluations that depends only logarithmically on $d$ \cite{Adcock2017c,Adcock2017,Chkifa2017}. This feature is particularly appealing in uncertainty quantification applications, where the function to approximate is a quantity of interest of a parametric PDE   \cite{Peng2014,Yang2013}.

This paper concerns the effective treatment of errors in the data by sparse recovery procedures.  There are three primary sources of such error:
\begin{enumerate}
\item[(i)] \textit{Truncation error.} This occurs when infinite expansion in an orthonormal polynomial basis is replaced by a finite expansion. This is an intrinsic source of error in high-dimensional function approximation, which does not depend on the particular application considered. 

\item[(ii)] \textit{Discretization error.} In practice, in applications to parametric PDEs, function samples are computed using a numerical PDE solver.  This error arises due to the discretization in the numerical routine (e.g.\ introducing a finite element mesh or performing a Galerkin projection).

\item[(iii)] \textit{Numerical error.} Following (ii), this further source of error occurs when the discretized problem is solved numerically (e.g.\ via an iterative solver).
%This occurs when the function samples are computed approximately by a numerical PDE solver.
\end{enumerate} 

To date, almost all works in sparse regularization for high-dimensional approximation have assumed an \textit{a priori} bound for this error.  Specifically, if the vector of data is $\bm{y} = \bm{y}^{\mathrm{exact}} + \bm{e}$, where $\bm{y}^{\mathrm{exact}}$ is the error-free data and $\bm{e}$ is the vector of errors, then one assumes the bound
\begin{equation}
\label{knownbound}
\|\bm{e}\|_2 \leq \eta,
\end{equation}
for some known $\eta > 0$.  In this case, sparse regularization performed using the (weighted) quadratically-constrained basis pursuit decoder admits rigorous theoretical recovery guarantees \cite{Adcock2017c,Adcock2017,Chkifa2017,Rauhut2016,YanGuoXui_l1UQ}.  In practice, however, a bound of the form \eqref{knownbound} is usually \textit{unknown}, since the sources of error (i), (ii), and (iii) are function dependent.  Cross validation is a common empirical remedy to tune such a parameter \cite{Doostan2011,Jakeman2015,Yang2013}.  Yet this is computationally expensive, and lacks theoretical guarantees. We also observe that for certain model problems (in the context of parametric PDEs) theoretical estimates for $\eta$ are available. However, the presence of implicit constants can create numerical issues, as we will demonstrate in our numerical experiments.

In this paper, we study three different decoders for sparse recovery: weighted quadratically-constrained basis pursuit, weighted LASSO and weighted square-root LASSO.  In all cases, we consider their performance for \textit{unknown errors}: that is, in the absence of the bound \eqref{knownbound}.  These decoders are compared from both the theoretical and numerical viewpoints. 
%In particular, we discuss the case where the samples are corrupted by unknown error (e.g., truncation, discretization, and numerical error). 
On the theoretical side, we prove uniform and robust recovery error estimates for the proposed decoders. From the numerical viewpoint, we compare their recovery performances in the case of pure function approximation and in the case of parametric ODEs and PDEs.

One of our main messages is that the lesser-known weighted square-root LASSO is more suitable for sparse high-dimensional function approximation than the commonly-used quadratically-constrained basis pursuit decoder.  It exhibits recovery performances comparable (if not superior) to the others, but the optimal choice of its tuning parameter does not depend on the noise level in the data.  Therefore, it does not require any \emph{a priori} knowledge on the error corrupting the samples. This fact is well-known by statisticians; however, it does not seem to have been fully exploited by the high-dimensional function approximation and the compressed sensing communities.

To complete the paper, we also consider a fourth type of error.  In addition to (i), (ii), and (iii), in large-scale uncertainty quantification implementations it is increasingly common to encounter 
\begin{enumerate}
\item[(iv)] \textit{Corruption error.} This occurs when an unknown subset of the measurements are corrupted arbitrarily, for instance due to a node failure when computations are performed on a cluster. 
\end{enumerate}
This is a fundamentally different type of error to the others.  Corruption errors can be arbitrarily large, but only affect a (typically unknown) fraction of the measurements. They arise in modern complex computational frameworks consisting of many interconnected parallel components, which are subject to  faults. These faults may be hard to detect and heavily pollute the measurements \cite{Bridges2012}.  

In this setting, the vector of errors is of the form $ \bm{e} = \bm{e}^{\mathrm{bounded}} + \bm{e}^{\mathrm{sparse}}$, where $\bm{e}^{\mathrm{bounded}}$ is a combination of (i), (ii), or (iii) and has small norm, and $ \bm{e}^{\mathrm{sparse}}$ is of type (iv) and it is sparse with large norm. Methods that are robust to this type of errors are often called \emph{fault-tolerant} or \emph{resilient} and have already been considered in uncertainty quantification and in compressed sensing \cite{Adcock2017compressed,Laska2009,Shin2016,Wright2010}.  To correct for this type of error, we study a fourth decoder; the weighted LAD-LASSO.  Numerically, we show that this method can achieve quite striking performance; in our experiments, it can easily correct for over 10\% corrupted measurements.  Much like the square-root LASSO, we conclude that this should be the method of choice in the presence of errors (i), (ii), (iii), and (iv).

%The main focus of this paper is the sparse recovery phase. We  examine four weighted $\ell^1$ minimization decoders: weighted quadratically-constrained basis pursuit, weighted LASSO, weighted square-root LASSO, and weighted LAD-LASSO. These decoders are compared from the theoretical and the numerical viewpoint. In particular, we discuss the case where the samples are corrupted by unknown error (e.g., truncation, discretization, and numerical error). On the theoretical side, we prove uniform and robust recovery error estimates for the proposed decoders. From the numerical viewpoint, we compare their recovery performances in the case of ``pure'' function approximation and in the case of parametric PDEs.

%One of the main messages of our comparative study is that the weighted square-root LASSO deserves more attention. In fact, this decoder exhibits recovery performances comparable (if not superior) to the others, but the optimal choice of its tuning parameter does not depend on the noise level. Therefore, it does not require any \emph{a priori} knowledge on the error corrupting the samples. This fact is well-known by statisticians; however, it does not seems to have been fully exploited by the high-dimensional function approximation and the compressed sensing communities yet.

The paper is organized as follows. In \S\ref{sec:main_contr}, we describe the problem setting, defining the high-dimensional function approximation framework and the four decoders considered for weighted $\ell^1$ minimization; moreover, we present the main contributions of the paper, i.e.\ uniform recovery guarantees for the proposed decoders under unknown error corrupting the measurements. \S\ref{sec:lit} contains a review of the literature for each decoder. In \S\ref{sec:num}, we compare the four decoders from a numerical perspective in the case of pure function approximation and in applications to parametric PDEs and ODEs with random inputs. \S \ref{sec:theory} deals with the theoretical analysis; this is the mathematical core of the paper and it is the most technical part. Finally, in \S\ref{sec:concl} we give our conclusions and state some open problems.

\section{Main contributions}
\label{sec:main_contr}
We now present the main contributions of the paper. We define the sparse high-dimensional approximation methodology in \S\ref{sec:PB}. Then, in \S\ref{sec:lower} we introduce the concept of sparsity in lower sets and other related notions, which are necessary to understand the theoretical analysis. Finally, in \S\ref{sec:rec_err} we state the uniform robust recovery error estimates for the four decoders (proofs of these results are given in \S\ref{sec:theory}).

\subsection{Sparse high-dimensional approximation}
\label{sec:PB}

Let $d \in \mathbb{N} := \{1,2,3,\ldots\}$, with $d \gg 1$, and consider a high-dimensional complex-valued function 
%we aim at approximating. Namely,
$$
f:D  \to \mathbb{C}, \quad \text{with } D:= (-1,1)^d.
$$ 
We define $d$ probability measures $\nu^{(1)},\ldots,\nu^{(d)}$ on $(-1,1)$ and  the corresponding tensor product measure on $D$ as 
$$
\nu := \nu^{(1)} \otimes \cdots \otimes \nu^{(d)}.
$$
Moreover, we consider $d$ orthonormal bases $\{\phi_j^{(\ell)}\}_{j\in\mathbb{N}_0}$ of $L^2_{\nu^{(\ell)}}(-1,1)$ for $\ell \in [d]$, where $[d]:=\{1,\ldots,d\}$ and $\mathbb{N}_0:=\{0\}\cup \mathbb{N}$. Then, we build $\{\phi_{\bm{i}}\}_{\bm{i}\in\mathbb{N}^d_0}$ as the corresponding tensor product basis of $L^2_\nu(D)$:
\begin{equation}
\label{eq:defbasis}
\phi_{\bm{i}} := \phi_{i_1}^{(1)} \otimes \cdots \otimes \phi_{i_d}^{(d)}, \quad \forall \bm{i} \in \mathbb{N}_0^d.
\end{equation}
Assume $f \in L^2_\nu(D) \cap L^{\infty}(D)$ and consider its expansion  
\begin{equation}
\label{eq:fexpansion}
f =  \sum_{\bm{i} \in \mathbb{N}_0^d} x_{\bm{i}} \phi_{\bm{i}},
\end{equation}
where $x_{\bm{i}} = \langle f, \phi_{\bm{i}}\rangle_{L^2_\nu}$, for every $\bm{i}\in\mathbb{N}_0^d$. Throughout the paper, we define the sequence of coefficients
$$
\bm{x} := (x_{\bm{i}})_{\bm{i}\in\mathbb{N}_0^d} \in \ell^2(\mathbb{N}_0^d).
$$
where $\ell^p(\mathbb{N}_0^d)$ denotes the space of sequences with finite $\ell^p$ norm. Given a subset $\Lambda \subseteq \mathbb{N}_0^d$ of finite cardinality $|\Lambda| = n$, we obtain a finite approximation of $f$ by considering the truncated basis $\{\phi_{\bm{i}}\}_{\bm{i}\in \Lambda}$. The corresponding finite-dimensional vector of coefficients is denoted by $\bm{x}_\Lambda$. Assuming the set of multi-indices in $\Lambda$ is ordered as $\bm{i}_1, \ldots, \bm{i}_n$, we can also write
\begin{equation}
\label{eq:fLambda}
f_{\Lambda} := \sum_{j = 1}^n x_{\bm{i}_j} \phi_{\bm{i}_j} \quad \text{and} \quad \bm{x}_\Lambda := (x_{\bm{i}_j})_{j \in [n]},
\end{equation}
The vector $\bm{x}_\Lambda$ will be equivalently considered as an element of $\mathbb{C}^n$ or of $\ell^{2}(\mathbb{N}_0^d)$. In the latter case, the components $x_{\bm{i}}$ with $\bm{i} \notin\Lambda$ are implicitly assumed to be zero.

We aim at recovering a sparse approximation of $f$ with respect to the basis $\{\phi_{\bm{i}}\}_{\bm{i}\in\Lambda}$ from pointwise samples. Hence, we define the \emph{vector of samples} $\bm{y}\in\mathbb{C}^m$ and the \emph{design matrix} $\mat{A} \in \mathbb{C}^{m \times n}$ as\footnote{The factor $1/\sqrt{m}$ is needed in order to guarantee the restricted isometry property for the design matrix $\mat{A}$. See \S\ref{sec:theory}.}
\begin{equation}
\label{eq:defA}
y_i := \frac{1}{\sqrt{m}} (f(\bm{t}_{i})),
\quad 
A_{ij} := \frac{1}{\sqrt{m}}(\phi_{\bm{i}_j}(\bm{t}_i)), \quad \forall i \in [m], \; j \in[n],
\end{equation}
where, as in previous works in high-dimensional polynomial approximation \cite{Chkifa2017,Doostan2011,Yang2013}, the samples are distributed independently at random according to the orthogonality measure $\nu$, i.e.
\begin{equation}
\label{eq:rand_samples}
\bm{t}_1,\ldots,\bm{t}_m \stackrel{\text{i.i.d.}}{\sim} \nu(\bm{t}).
\end{equation}
Then, the truncated vector of coefficients $\bm{x}_\Lambda \in \mathbb{C}^n$ satisfies 
\begin{equation}
\label{eq:linsys}
\bm{y} = \mat{A} \bm{x}_{\Lambda} + \bm{e},
\end{equation}
where $\bm{e} \in \mathbb{C}^m$ is an unknown source of error corrupting the samples. The vector $\bm{e}$ contains (at the very least) the truncation error
\begin{equation}
\label{eq:deftrunc}
(e^{\text{trunc}}_\Lambda)_i
:=\frac{1}{\sqrt{m}} (f(\bm{t}_i) - f_\Lambda(\bm{t}_i)) 
= \frac{1}{\sqrt{m}}\sum_{\bm{i} \notin \Lambda} x_{\bm{i}} \phi_{\bm{i}}(\bm{t}_i), \quad \forall i \in [m].
\end{equation}
In general, $\bm{e} = \bm{e}^{\text{trunc}}_\Lambda + \bm{e}^{\text{misc}}$, where $\bm{e}^{\text{misc}}$  contains other sources of error, such as discretization error, numerical error, faults, or other forms of noise, all of which occur commonly in uncertainty quantification applications.  In this paper we shall consider the following two types of noise models:
\begin{enumerate}
\item[(a)] \textit{Bounded noise.}  Here $\| \bm{e} \|_2$ is assumed to be small and the goal is to recover $f$ up to this error.  This is the case for truncation and numerical errors (types (i), (ii), and (iii) in \S \ref{sec:intro}), for instance.
\item[(b)] \textit{Bounded + unbounded sparse noise.}  In this case, $\bm{e} = \bm{e}^{\text{bounded}} + \bm{e}^{\text{sparse}}$, where $\bm{e}^{\text{bounded}} $ is as in (a) and $\bm{e}^{\text{sparse}}$ has only a small number of nonzero entries, but which can be arbitrarily large and whose locations are unknown.  This is the case when corruption errors, due for instance to solver faults, are additionally present (type (iv) in \S \ref{sec:intro}).  The goal in this case is to correct for the corruption error, and to recover $f$ up to $\| \bm{e}^{\text{bounded}}\|_2$.
\end{enumerate}
Note that, since the function to be approximated is high-dimensional, the cardinality $n$ of the truncated index set $\Lambda$ is typically large. Therefore, we focus on the regime $m < n$, where \eqref{eq:linsys} is underdetermined.

Finally, we will consider two particular examples of orthonormal systems of  $L^2_\nu(D)$: the well-known tensorized Legendre and Chebyshev polynomials. 
In $d$ dimensions, these are orthogonal with respect to the tensor uniform and Chebyshev measures
%In the one-dimensional case, they are orthogonal with respect to the uniform measure and the Chebyshev measure
%$$
%\text{d}\nu^{(1)} = \frac{1}{2} \text{d}t \quad\text{(Legendre)}, \qquad 
%\text{d}\nu^{(1)} = \frac{1}{\pi\sqrt{1-t^2}} \text{d}t \quad \text{(Chebyshev)},
%$$
%respectively. 
%In the $d$-dimensional case, the orthogonality  measures are
$$
\text{d}\nu = \frac{1}{2^d} \text{d}\bm{t} \quad\text{(Legendre)}, \qquad 
\text{d}\nu = \prod_{\ell = 1}^d\frac{1}{\pi\sqrt{1-t_\ell^2}} \text{d}\bm{t} \quad \text{(Chebyshev)}.
$$
In each case, the samples $\bm{t}_1,\ldots,\bm{t}_m$ are randomly and independently distributed according to $\nu$.

\subsection{Decoders}
In order to recover an approximate solution $\hat{\bm{x}}_\Lambda = (\hat{x}_{\bm{i}_j})_{j \in[n]} \in \mathbb{C}^n$ from  \eqref{eq:linsys}, we use weighted $\ell^1$ minimization. Given a sequence of \emph{positive} weights $\bm{w}  = (w_{\bm{i}})_{\bm{i} \in \mathbb{N}_0^d}$, the corresponding weighted $\ell^p$ norm, referred to as the $\ell^p_{\bm{w}}$ norm, is defined as 
$$
\|\bm{x}\|_{p,\bm{w}} 
:= \bigg(\sum_{\bm{i} \in \mathbb{N}_0^d} 
w_{\bm{i}}^{2-p} |x_{\bm{i}}|^p\bigg)^{1/p},
$$
and $\ell^p_{\bm{w}}(\mathbb{N}_0^d)$ is the space of sequences having  finite $\ell_{\bm{w}}^p$ norm. 
%\RED{[Condition $\bm{w} \geq 1$ removed.]}

Given weights $\bm{u}$ and errors of the form (i) and (ii), and (iii) (see \S \ref{sec:intro}), we compute $\hat{\bm{x}}_\Lambda$ with one of the following decoders: 
\\
\\
\emph{Weighted Quadratically-Constrained Basis Pursuit (WQCBP)}
\begin{equation}
\label{eq:WQCBP}
\hat{\bm{x}}_{\Lambda} 
= \arg \min_{\bm{z} \in \mathbb{C}^n} \|\bm{z}\|_{1,\bm{u}} \text{ s.t. } \|\mat{A} \bm{z} - \bm{y}\|_2 \leq \eta,
\end{equation}
\emph{Weighted LASSO (WLASSO)}
\begin{equation}
\label{eq:WLASSO}
\hat{\bm{x}}_{\Lambda} 
= \arg\min_{\bm{z} \in \mathbb{C}^n} \|\bm{z}\|_{1,\bm{u}} + \lambda \|\mat{A} \bm{z} - \bm{y}\|_2^2,
\end{equation}
\emph{Weighted Square-Root LASSO (WSR-LASSO)}
\begin{equation}
\label{eq:WSRLASSO}
\hat{\bm{x}}_{\Lambda} 
= \arg\min_{\bm{z} \in \mathbb{C}^n} \|\bm{z}\|_{1,\bm{u}} + \lambda \|\mat{A} \bm{z} - \bm{y}\|_2, 
\end{equation}
%In each case, the resulting approximation $\tilde{f}$ of $f$ is then given by
%\begin{equation}
%\label{eq:ftilde}
%\tilde{f} := \sum_{\bm{i} \in  \Lambda} \hat{x}_{\bm{i}} \phi_{\bm{i}}.
%\end{equation}
The above decoders are designed to deal specifically with bounded noise (model (a) in \S \ref{sec:PB}).  
%Our main results, detailed in \S \ref{sec:rec_err}, give error bounds that depend linearly on $\| \bm{e} \|_2$.  
Conversely, if the measurements additionally contain corruption errors (model (b) in \S \ref{sec:PB}), in which case $\bm{e}$ has a small number of large entries, then we instead consider the following decoder: 
%\RED{[I've removed the weights on the fitting term.]}
\\
\\
\emph{Weighted LAD-LASSO (WLAD-LASSO)}
\begin{equation}
\label{eq:WLADLASSO}
\hat{\bm{x}}_{\Lambda} 
= \arg\min_{\bm{z} \in \mathbb{C}^n} \|\bm{z}\|_{1,\bm{u}} + \lambda \|\mat{A} \bm{z} - \bm{y}\|_{1}.
\end{equation}
The use of the $\ell^1$-norm in the fitting term is to exploit the sparsity of the corruptions in this case.

Notice that the minimizer $\hat{\bm{x}}_\Lambda$ need not be necessarily unique in our framework. Given $\hat{\bm{x}}_\Lambda$, the resulting approximation $\tilde{f}$ of $f$ is then defined by
%\RED{[I have removed the repeated definition of $\tilde{f}$ above]}
\begin{equation}
\label{eq:ftilde}
\tilde{f} := \sum_{\bm{i} \in  \Lambda} \hat{x}_{\bm{i}} \phi_{\bm{i}}.
\end{equation}

We assume throughout that $\eta \geq 0$ and $\lambda >0$. These will be referred to as \emph{tuning parameters}. When $\eta = 0$, the optimization program \eqref{eq:WQCBP} is also referred to as \emph{weighted basis pursuit (WBP)}. It is worth noting that WQCBP is the only optimization program in constrained form.  We will comment more on these decoders while reviewing the literature (\S\ref{sec:lit}).

As for the choice of the weights, $\bm{u}$ is chosen as follows:
\begin{equation}
\label{eq:defintrinsic}
u_{\bm{i}} = \|\phi_{\bm{i}}\|_{L^\infty}, \quad \forall \bm{i} \in \mathbb{N}_0^d.
\end{equation} 
They are referred to as \emph{intrinsic weights} \cite{Adcock2017c}. Notice that $u_{\bm{i}} := \|\phi_{\bm{i}}\|_{L^\infty} \geq \|\phi_{\bm{i}}\|_{L^2_\nu} = 1$.
%\RED{[comment on $\bm{u} \geq 1$ removed]}. 
This particular choice of weights has been proved to be effective for WQCBP, both theoretically and numerically \cite{Adcock2017c,Adcock2017,Chkifa2017}. For the polynomial bases employed here, we have
\begin{equation}
\label{eq:intrinsicChebLeg}
 u_{\bm{i}} = \prod_{\ell = 1}^d \sqrt{2 i_\ell +1} \quad \text{(Legendre)}, \qquad 
 u_{\bm{i}} = 2^{\|\bm{i}\|_0/2} \quad\text{(Chebyshev)}.
\end{equation}

\subsection{Sparsity in lower sets}
\label{sec:lower}

Compressed sensing theory considers the recovery of sparse vectors from limited numbers of measurements.  Many works have sought to apply these principles to high-dimensional approximation (see \S \ref{sec:lit}).  Unfortunately, standard application of these principles leads to sample complexities that depend exponentially on the dimension $d$ (see, for example, \cite{Adcock2017}).  More recently \cite{Adcock2017c,Chkifa2017}, it has been shown that this \textit{curse of dimensionality in the sample complexity} can be overcome by considering certain structured sparsity models based on so-called \emph{lower sets} (also known as \emph{downward closed} {or \emph{monotone sets}):

%We aim at recovering the best approximation of $f$ with respect to a class of structured multi-indices, called \emph{lower sets} (also known as \emph{downward closed sets}. Their usage in high-dimensional approximation is now well-established.
\begin{defn}[Lower set]
The set $S \subseteq \mathbb{N}_0^d$ is said to be a \emph{lower set} if whenever $\bm{i} \in S$ and $\bm{j} \leq \bm{i}$ (where the inequality holds componentwise) then $\bm{j} \in S$.
%
%
%the following holds for every $\bm{i},\bm{j}\in\mathbb{N}_0^d$: 
%$$
%\text{if }\bm{j} \in S \text{ and } \bm{i} \leq \bm{j}, \text{ then } \bm{i} \in S,
%$$ 
\end{defn}
Employing this definition, we consider the space of $s$-sparse vectors whose supports are lower sets
$$
\Sigma_{s,L} := \{\bm{z} \in \ell^2(\mathbb{N}_0^d) : |\supp(\bm{z})| \leq s, \; \supp(\bm{z}) \text{ lower}\}.
$$
Here $\supp(\bm{z}):=\{\bm{i} \in \mathbb{N}_0^d: z_{\bm{i}} \neq 0\}$.  The corresponding best $s$-term approximation in lower sets is
$$
\sigma_{s,L}(\bm{z})_{1,\bm{u}} := 
\inf_{\bm{z}' \in \Sigma_{s,L}} \|\bm{z}-\bm{z}'\|_{1,\bm{u}}.
$$
Lower sets are common tools in high-dimensional approximation \cite{Chkifa2014,deBoor1990,Dyn2014,Lorentz1986,Migliorati2014}. The lower set structure turns out to be crucial when approximating (a quantity of interest of) the solution map of a parametric PDE.  For example, in \cite{Chkifa2015} it has been proved that for a wide class of parametric PDEs the best $s$-term approximation error (in the mean square and in the uniform sense) in lower sets of the solution map with respect to Legendre polynomials decays  algebraically in $s$ and it is independent of the dimension $d$ of the parametric domain (see also \cite{Cohen2010,Cohen2011}). %This class of PDEs includes: (i) elliptic diffusion equations with coefficients that depend on the parameter vector in a not necessarily affine manner, (ii) parabolic diffusion equations with similar dependence of the coefficient on the parameters, (iii) nonlinear, monotone parametric elliptic PDEs, (iv) elliptic equations set on a domain that is parametrized by a vector.

An essential property of lower sets is that the union of all lower sets of cardinality $s$ is the well-known \emph{hyperbolic cross} of order $s$,  namely
\begin{equation}
\label{eq:defHC}
\Lambda^{\text{HC}}_{d,s} 
:= \left\{\bm{i} = (i_1,\ldots,i_d) \in \mathbb{N}_0^d : \prod_{\ell = 1}^d (i_\ell +1) \leq s\right\}
\equiv\bigcup_{\substack{|S|\leq s \\ S \text{ lower}}} S. 
\end{equation}
%This structured multi-index set is ubiquitous in high-dimensional approximation and has been introduced in \cite{Smolyak1963}. 
From now on, this set will be adopted as our truncated multi-index set, i.e.\
$$
\Lambda = \Lambda^{\text{HC}}_{d,s}. 
$$
We note that the cardinality of this index set can be bounded from above as (see, for example, \cite{Adcock2017})
%\footnote{The upper bound \eqref{eq:HCbound} is obtained by combining \cite[Theorem 3.7]{Chernov2016} (with $\delta = 1/2$) and \cite[Theorem 4.9]{Kuhn2015}. }
\begin{equation}
\label{eq:HCbound}
n = |\Lambda^{\text{HC}}_{d,s}| 
\leq \min\left\{ 2 s^3 4^d, \text{e}^2 s^{2 + \log_2(d)}\right\}.
\end{equation}
With this in hand, we also define the \emph{intrinsic lower sparsity} of order $s$ as the quantity
\begin{equation}
\label{eq:defK(s)}
K(s) := \max\left\{|S|_{\bm{u}} : S \subseteq \mathbb{N}_0^d, \; |S|\leq s , \; S \text{ lower}\right\},
\end{equation}
where 
\begin{equation}
\label{eq:def|S|_u}
|S|_{\bm{u}}:=\sum_{\bm{i} \in S} u_{\bm{i}}^2
\end{equation} 
is the \emph{weighted cardinality} of a subset $S$ with respect to the weights $\bm{u}$ \cite{Rauhut2016}. This quantity is of crucial importance to our results, since it will determine a sufficient condition on the sample complexity $m$. In the case of Chebyshev and Legendre polynomials, $K(s)$ is known to scale proportionally to $s^{\gamma}$, with
\begin{equation}
\label{eq:defgamma}
\gamma = 
2 \quad \text{(Legendre)},\qquad
\gamma = \ln(3)/\ln(2) \quad \text{(Chebyshev)}.
\end{equation}
Specifically, (see \cite[Lemma 3.7]{Chkifa2017}) we have the following:
\begin{lem}[Intrinsic lower-sparsity bounds]
\label{lem:K(s)bounds}
Let $2 \leq s \leq 2^{d+1}$. If $\{\phi_{\bm{i}}\}_{\bm{i} \in \mathbb{N}_0^d}$ is the tensor Legendre or Chebyshev basis, then
\begin{equation}
 s^\gamma/4 \leq K(s) \leq s^\gamma,
\end{equation}
where $K(s)$ and $\gamma$ are as in \eqref{eq:defK(s)} and \eqref{eq:defgamma} respectively. Moreover, the upper estimate holds for all $s \geq 2$.
\end{lem}

\subsection{Robust recovery guarantees under unknown error}
\label{sec:rec_err}

We now present the recovery guarantees for the decoders \eqref{eq:WQCBP}-\eqref{eq:WLADLASSO}. These estimates are \emph{robust}, in the sense that they admit the presence of unknown error $\bm{e}$ corrupting the samples, and \emph{uniform} since, given a set of random pointwise samples distributed according to \eqref{eq:rand_samples}, they hold uniformly in $f$ (with high probability). The proofs of these results are given in \S\ref{sec:theory}. 

As usual, the expression $X \lesssim Y$ means that there exists a constant $C$ independent of $X$ and $Y$ such that $X \leq C Y$. The notation $X \gtrsim Y$ is defined analogously and $X \asymp Y$ means that $X \lesssim Y$ and $X \gtrsim Y$ hold simultaneously.

In all cases (apart from WLAD-LASSO, where an additional constraint appears), uniform recovery with probability at least $1-\varepsilon$ holds, provided the sample complexity satisfies
\begin{equation}
\label{eq:samplecomplexity}
m \gtrsim s^{\gamma} \cdot L,
\end{equation}
where $\gamma$ is given by \eqref{eq:defgamma} and $L$ is a polylogarithmic factor defined as
\begin{equation}
\label{eq:polylogsimple}
L = L(s,\varepsilon) =  
\ln^2(s) \min\{\ln(s) + d, \ln(2d)\ln(s)\} + \ln(s)\ln(s/\varepsilon).
\end{equation}
%The expression \eqref{eq:polylogsimple} can be derived from \eqref{eq:defpolylog} choosing a value for $0<\delta < 1/13$, recalling that $K(s) \asymp s^{\gamma}$ thanks to Lemma~\ref{lem:K(s)bounds}, and using the upper bound to the cardinality $n$ of the hyperbolic cross \eqref{eq:HCbound}. Combining \eqref{eq:samplecomplexity} and \eqref{eq:polylogsimple}, we see that the sample complexity depends logarithmically in $d$. 
Note that the particular dependence on $d$ in this estimates stems from bound \eqref{eq:HCbound} on the cardinality $n$ of the hyperbolic cross.

We now present the recovery guarantees decoder by decoder. 

\paragraph{WQCBP} Theorems~\ref{thm:WQCBPerror-blind} \& \ref{thm:WQCBPtailbound} imply that, provided $m \asymp s^{\gamma} L$, the approximate solution $\tilde{f}$ defined in \eqref{eq:ftilde} obtained via WQCBP satisfies
\begin{align*}
\|f-\tilde{f}\|_{L^\infty}
& \lesssim \sigma_{s,L}(\bm{x})_{1,\bm{u}}
+ s^{\gamma/2} [(\eta + \|\bm{e}\|_2) + \mathcal{Q}\sqrt{L}\max\{\|\bm{e}\|_2-\eta,0\}],\\
\|f-\tilde{f}\|_{L^2_\nu}
& \lesssim \frac{\sigma_{s,L}(\bm{x})_{1,\bm{u}}}{s^{\gamma/2}}
+ \eta + \|\bm{e}\|_2  
+ \|f-f_\Lambda\|_{L^2_\nu} 
+ \mathcal{Q}\sqrt{L}\max\{\|\bm{e}\|_2-\eta,0\},
\end{align*}
with probability at least $1-\varepsilon$, where 
$$
\mathcal{Q}
= \mathcal{Q}(\mat{A},\Lambda,\bm{u})
:=  \sqrt{\frac{|\Lambda|_{\bm{u}}}{n}}\frac{1}{\sigma_m(\sqrt{\frac{m}{n}}\mat{A}^*)},
$$
where $\sigma_m(\sqrt{\frac{m}{n}}\mat{A}^*)$ is the $m^{th}$ singular value (in decreasing order) of $\sqrt{\frac{m}{n}}\mat{A}^*$.
Notice that the term depending on $\mathcal{Q} \sqrt{L}$ in the error bound vanishes when $\eta \geq \| \bm{e} \|_2$: that is,
$$
\eta \geq \| \bm{e} \|_2 \quad \Rightarrow \quad \|f-\tilde{f}\|_{L^2_\nu} \lesssim \frac{\sigma_{s,L}(\bm{x})_{1,\bm{u}}}{s^{\gamma/2}}
+ \eta + \|\bm{e}\|_2 + \|f-f_\Lambda\|_{L^2_\nu},
$$
and similarly for the $L^\infty$ error.  Moreover, since $\eta$ and $\| \bm{e} \|_2$ both appear in the error bound, the optimal choice of $\eta$ is 
\begin{equation}
\label{eq:optimaleta}
\eta = \|\bm{e}\|_2.
\end{equation}
Conversely, when $\| \bm{e} \|_2$ is unknown, the price to pay is an additional error term, which depends logarithmically on $d$ and algebraically on $s$. In particular, the algebraic dependence on $s$ is a consequence of Theorem~\ref{thm:WQCBPtailbound}, where we show that $\mathcal{Q} \lesssim s^{\alpha/2}/\sigma_m(\sqrt{\frac{m}{n}\mat{A}^*})$, with $\alpha = 1,2$ for tensor Chebyshev and Legendre polynomials, respectively.   

This bound does, however, suggest that tuning $\eta$ empirically via cross validation, so as to achieve $\eta \approx \| \bm{e} \|_2$ can improve the recovery error.  Our numerical results in \S \ref{sec:num} partially support this conclusion. 

%\PURP{In particular, the benefit of estimating $\eta \approx \|\bm{e}\|_2$ is evident in pure function approximation when the samples are corrupted by random Gaussian noise (\S\ref{sec:par_vs_err}). In the  case of parametric PDEs and ODEs (\S\ref{sec:UQ}), numerical experiments show that choosing $\eta \lesssim \|\bm{e}\|_2$ is sufficient to reach the best accuracy.} \RED{I'm not 100\% sure that this is the best way to comment here... feel free to cut off modify.}

\paragraph{WLASSO} In Theorem~\ref{thm:WLASSOrecovery}, we show that, provided \eqref{eq:samplecomplexity}  holds and
\begin{equation}
\label{eq:lambdaWLASSO}
\lambda 
\asymp \frac{\sqrt{K(s)}}{\|\bm{e}\|_2} 
\asymp \frac{s^{\gamma/2}}{\|\bm{e}\|_2},
\end{equation}
the approximate solution $\tilde{f}$ computed by means of WLASSO satisfies
\begin{align}
\label{eq:WLASSO_Linfest}
\|f-\tilde{f}\|_{L^\infty}
& \lesssim \sigma_{s,L}(\bm{x})_{1,\bm{u}}
+ s^{\gamma/2} \|\bm{e}\|_2,\\
\label{eq:WLASSO_L2est}
\|f-\tilde{f}\|_{L^2_\nu}
& \lesssim \frac{\sigma_{s,L}(\bm{x})_{1,\bm{u}}}{s^{\gamma/2}}
+ \|\bm{e}\|_2 + \|f-f_\Lambda\|_{L^2_\nu},
\end{align}
with probability at least $1-\varepsilon$.

Unfortunately, the choice of tuning parameter \eqref{eq:lambdaWLASSO} requires knowledge of $\| \bm{e} \|_2$, similar to WQCBP.  On the other hand, this requirement is certainly less stringent than the one-sided bound $\eta \geq \| \bm{e} \|_2$: no logarithmic factors are present in the error bounds if $\| \bm{e} \|_2$ is estimated up to a constant.

\paragraph{WSR-LASSO} Theorem~\ref{thm:WSRLASSOrecovery} shows that provided \eqref{eq:samplecomplexity} holds and
\begin{equation}
\label{eq:lambdaWSRLASSO}
\lambda 
\asymp \sqrt{K(s)} 
\asymp s^{\gamma/2},
\end{equation}
the approximate solution $\tilde{f}$ computed using WSR-LASSO satisfies \eqref{eq:WLASSO_Linfest} and \eqref{eq:WLASSO_L2est} with probability at least $1-\varepsilon$. In other words, WSR-LASSO attains the same recovery guarantees as WLASSO and WQCBP \textit{but without any prior knowledge} on the noise $\bm{e}$.   Specifically, optimal WLASSO tuning parameter  \eqref{eq:lambdaWLASSO} depends on $\|\bm{e}\|_2$, whereas the optimal WSR-LASSO tuning parameter \eqref{eq:lambdaWSRLASSO} is independent of this factor.  We note in passing that $\lambda$ depends on the sparsity $s$, but this is always known in our framework since it is the index of the truncated hyperbolic cross set \eqref{eq:defHC}.

%\GR{NOTE TO SELF: Add discussion + context here.}

\paragraph{WLAD-LASSO} From Theorem~\ref{thm:WLADLASSOrecovery}, we see that, given $k \in \mathbb{N}$ and provided that
\begin{equation}
\label{corrSC}
m \gtrsim s^{\gamma} \cdot  \max\{L, k\},
\end{equation}
the approximate solution $\tilde{f}$ obtained via WLAD-LASSO satisfies 
%\RED{[The term $\lambda\|\bm{e}-(\bm{y}-\mat{A}\hat{\bm{x}}_\Lambda)\|_{1,\bm{v}}$ and $\lambda\|\bm{e}-(\bm{y}-\mat{A}\hat{\bm{x}}_\Lambda)\|_{2}$ in the LHSs have been removed.]}
$$
\|f-\tilde{f}\|_{L^\infty}
\lesssim \sigma_{s,L}(\bm{x})_{1,\bm{u}}
+ \lambda \sigma_k(\bm{e})_{1}, 
$$
and also
$$
\|f-\tilde{f}\|_{L^2_\nu}
 \lesssim (1+\sqrt{\Theta}) \left(\frac{\sigma_{s,L}(\bm{x})_{1,\bm{u}}}{s^{\gamma/2}} + \frac{\sigma_k(\bm{e})_{1}}{\sqrt{k}}\right)
+ \|f-f_\Lambda\|_{L^2_\nu},
$$
with probability at least $1-\varepsilon$, where
\begin{equation}
\label{eq:defeta}
\Theta:=\frac{\sqrt{K(s) + \lambda^2 k}}{\min\{\sqrt{K(s)}, \lambda\sqrt{k}\}},
\end{equation}
and
$$
\sigma_{k}(\bm{e})_{1}:= \inf_{\bm{d}:|\supp(\bm{d})| \leq k} \|\bm{e}-\bm{d}\|_{1},
$$
where $\supp(\bm{d}):=\{i \in [m]: d_i \neq 0\}$, is the best $k$-term approximation error. Notice that the quantity $\Theta$ defined in \eqref{eq:defeta} is minimized for 
\begin{equation}
\label{eq:lambdaWLADLASSO}
\lambda = \sqrt{\frac{K(s)}{k}}.
\end{equation} 
However,  $\lambda = 1$ seems to be a better choice in practice (see Fig.~\ref{fig:m_vs_err} and numerical illustrations in \cite{Adcock2017compressed}).

Observe that the recovery guarantees for WQCBP, WLASSO and WSR-LASSO all depend on $\| \bm{e} \|_2$.  Conversely, for WLAD-LASSO this term is replaced by $\sigma_{k}(\bm{e})_{1}$.  This suggests WLAD-LASSO, unlike the other decoders, can effectively correct for errors of type (b).  Specifically, if $\bm{e} = \bm{e}^{\text{bounded}} + \bm{e}^{\text{sparse}}$, where $\bm{e}^{\text{sparse}}$ has at most $k$ nonzero entries, then
$$
\sigma_{k}(\bm{e})_{1} 
\leq 
\sigma_{k}(\bm{e}^{\text{sparse}})_{1} 
+ \sigma_{k}(\bm{e}^{\text{bounded}})_{1} 
\leq \| \bm{e}^{\text{bounded}} \|_{1} 
\leq \sqrt{m} \| \bm{e}^{\text{bounded}} \|_2.
$$
In particular, if $\lambda \asymp s^{\gamma/2} / \sqrt{k}$, as suggested by \eqref{eq:lambdaWLADLASSO}, this yields the error bound
\begin{equation}
\label{WLADoptbound}
\|f-\tilde{f}\|_{L^2_\nu}
 \lesssim \frac{\sigma_{s,L}(\bm{x})_{1,\bm{u}}}{s^{\gamma/2}} + s^{\gamma/2} \| \bm{e}^{\text{bounded}} \|_2 + \| f - f_{\Lambda} \|_{L^2_\nu},
\end{equation}
and similarly for the $L^\infty$-norm error.  Up to the factor $s^{\gamma/2}$ (see Remark \ref{sharpness} below) this is the same as the bounds for WLASSO and WSR-LASSO.  Except, of course, that the WLAD-LASSO corrects for an $k$-sparse corruption error of arbitrary magnitude, unlike the other decoders.

\begin{rmrk}[Sharpness of the sample complexity]
\label{sharpness}
The estimate \eqref{corrSC} suggests the number of corrupted samples $k$ can be roughly $L$ while retaining the same sample complexity as in the uncorrupted case.  We do not believe this is sharp.  Indeed, there is reason to expect that $k = c m$ corruptions can be allowed, for some $0 < c <1$.  See \cite{Adcock2017compressed,Li2013} for further discussion.  If this conjecture were true, then the $s^{\gamma/2}$ dependence in the noise term in \eqref{WLADoptbound} would vanish.
\end{rmrk}

%\RED{[Remark ``Generalization to other sampling schemes'' cut.]}
%\begin{rmrk}[Generalization to other sampling schemes]
%The aforementioned recovery results can be generalized to matrices $\mat{A}$ different from \eqref{eq:defA}. For example, $\mat{A}$ could be a random Gaussian matrix, a random Bernoulli matrix, or a subsampled isometry. The arguments employed in \S\ref{sec:theory} can be easily adapted to these cases using sample complexing bounds already available in the literature \cite{Foucart2013}.  \GR{Do we need this remark?  I'd consider cutting it.}
%\end{rmrk}

\section{Literature review}
\label{sec:lit}

For a general introduction to the optimization programs \eqref{eq:WQCBP}-\eqref{eq:WLADLASSO} in the unweighted case, we refer the reader to, e.g., \cite{Foucart2013,vandeGeer2016,Hastie2015}. We now briefly review the literature regarding each decoder.

\paragraph{WQCBP} This optimization program has a long history \cite{Donoho1992,Logan1965}, but here we are particularly interested in its application in compressed sensing \cite{Candes2006,Donoho2006,Foucart2013}.  The weighted version of QCBP has been introduced in the context of function approximation with bounded orthonormal systems in \cite{Rauhut2016}, although previously studied in \cite{Candes2008b,Friedlander2012,Peng2014,Yu2013}. For applications of WQCBP to high-dimensional function approximation and uncertainty quantification, we refer the reader to \cite{Adcock2017c,Adcock2017b,Adcock2017,Chkifa2017,Peng2014,Rauhut2017,Yang2013} and references therein.
All these works focus on the error-aware scenario, where upper bounds of the form \eqref{knownbound} are assumed to be known \emph{a priori}, even though this assumption is unlikely to be met in practice. It is worth mentioning that robust recovery guarantees for unweighted $\ell^1$ minimization have been recently proved in the error-blind scenario, where upper bounds of the form \eqref{knownbound} are not assumed to be known (see \cite{Brugiapaglia2017,Brugiapaglia2017recovery} and references therein). In this paper, these results are generalized to the weighted case.

\paragraph{WLASSO} The literature regarding LASSO (Least Absolute Shrinkage and Selection Operator) is boundless. The pioneering paper \cite{Tibshirani1996} by R.\ Tibshirani has already reached more than twenty thousand citations according to Google Scholar (roughly one thousand citations per year on average). We refer the reader to \cite[\S 2.10]{Hastie2015} and references therein for an extensive review on the state-of-the-art of LASSO and for historical remarks. A weighted variant of LASSO, called the ``adaptive LASSO'' due to the iterative and adaptive procedure used to update the weights, is considered in \cite{Huang2008,Zou2006}. %\footnote{The name LASSO is employed in the statistics and signal processing literature when referring to
%\begin{equation}
%\label{eq:LASSO}
%\min_{\bm{z} \in \mathbb{C}^n} \|\mat{A}\bm{z}-\bm{y}\|_2, \quad \text{s.t. } \|\bm{z}\|_1 \leq \tau.
%\end{equation}
%or to the program
%\begin{equation}
%\label{eq:LASSOLag}
%\min_{\bm{z} \in \mathbb{C}^n} \|\mat{A}\bm{z}-\bm{y}\|_2^2 + \lambda \|\bm{z}\|_1,
%\end{equation}
%sometimes referred to as the \emph{Lagrangian form} of the LASSO. %These two programs are in fact equivalent. Indeed, given a value of $\tau$, for each solution $\hat{\bm{z}}$ to \eqref{eq:LASSO}, by Lagrangian duality there exists a corresponding value of $\lambda$ such that $\hat{\bm{z}}$ solves  \eqref{eq:LASSOLag} . Conversely, a solution $\hat{\bm{z}}$ to problem \eqref{eq:LASSOLag} solves \eqref{eq:LASSO} with $\tau = \|\hat{\bm{z}}\|_1$. (See \cite[Proposition 3.2]{Foucart2013} and \cite[\S 2.2]{Hastie2015}).}

\paragraph{WSR-LASSO} The unweighted version of SR-LASSO has been introduced in \cite{Belloni2011} and studied under the name ``scaled LASSO'' in \cite{Sun2012}. Further results about SR-LASSO can be found in \cite{Babu2014,Belloni2014,Bunea2014,Pham2015,Tian2015}. In \cite{Stucky2015}, the authors consider a version of the SR-LASSO where the $\ell^1$ norm is replaced by a generic sparsity-inducing norm, admitting the weighted SR-LASSO as a particular case. For the statisticians, an attractive feature of SR-LASSO is that the optimal tuning parameter does not depend on the variance of the random noise corrupting the observations, like in the LASSO case. Our study confirms that this is the case also in high-dimensional function approximation.

\paragraph{WLAD-LASSO}  The LAD (Least Absolute Deviation) regression and the LASSO program have been combined into the so-called LAD-LASSO decoder in order to make the LASSO more robust to the presence of heavy-tailed errors or outliers in the response \cite{Gao2008,Gao2010,Wang2007,Xu2005,Xu2010}. In the context of compressed sensing, this decoder is also known as a \emph{fault-tolerant} version of $\ell^1$ minimization. It has been considered in \cite{Laska2009,Li2013,Nguyen2013,Stankovic2014,Studer2012,Su2016,Su2016b,Wright2010} and in {\cite{Adcock2017compressed,Shin2016} with applications to uncertainty quantification. In \cite{Arslan2012}, the author introduces the weighted variant of LAD-LASSO and proves that is has some statistical properties (asymptotic normality and  consistency). An adaptively reweighted version of LAD-LASSO is considered in \cite{Li2017}.

%%%%%%%%%%%%%%%%%%%%%%%%%%%%%%%%%%%%%%%%%%%%%%%%%%%%%%%%%%%%%%%%%
%%%%%%%%%%%%%%%%%%%%%%%%%%%%%%%%%%%%%%%%%%%%%%%%%%%%%%%%%%%%%%%%%
\section{Numerical experiments}
%%%%%%%%%%%%%%%%%%%%%%%%%%%%%%%%%%%%%%%%%%%%%%%%%%%%%%%%%%%%%%%%%
%%%%%%%%%%%%%%%%%%%%%%%%%%%%%%%%%%%%%%%%%%%%%%%%%%%%%%%%%%%%%%%%%
\label{sec:num}
\newcommand\myxor{\mathbin{\char`\^}}

We compare the decoders \eqref{eq:WQCBP}-\eqref{eq:WLADLASSO} from the numerical viewpoint. The section is structured as follows. First, we formally define the cross-validation procedure employed to choose the tuning parameters in \S\ref{sec:cv}. Then, we compare the four decoders by studying their recovery error as a function of the tuning parameter (\S\ref{sec:par_vs_err}) and of the sample complexity (\S\ref{sec:m_vs_err}). 
%In \S\ref{sec:sigma_min}, we show that  the constants involved in the recovery error estimates for WQCBP are moderate. 
Finally, we compare the performance of the four decoders in an application to the uncertainty quantification of parametric ODEs and PDEs with random inputs in \S\ref{sec:UQ}.

All the numerical experiments have been performed in \textsc{Matlab$^\text{\textregistered}$} using CVX, an optimization toolbox for solving convex problems \cite{Grant2008,cvx}. We always run CVX setting \texttt{cvx\_precision best} and \texttt{cvx\_solver mosek}. We have used \textsc{Matlab$^\text{\textregistered}$} R2016b version 9.1 64-bit on a MacBook Pro equipped with a 3 GHz Intel Core i7 processor and with 8 GB DDR3 RAM. For the sake of convenience, we will sometimes use the \textsc{Matlab$^\text{\textregistered}$} vector notation to denote objects like $10.\text{$\myxor$}(1:0.5:2) = (10^1,10^{1.5},10^2)$.

\subsection{Cross validation} 
\label{sec:cv}
We define the cross-validation procedure that will be employed in the next sections. The pseudo-code is reported in Algorithm~\ref{alg:cv}. It corresponds to the cross-validation procedure described in \cite[\S 2.3]{Hastie2015}, usually referred to as \emph{$K$-fold cross validation} \cite{Arlot2010}. It is worth noticing  that step~\ref{step:valid_err} of Algorithm~\ref{alg:cv} can be replaced with $\varepsilon(t,g,p) = \|\mat{A}_v \hat{\bm{x}}-\bm{y}_v\|_2$, as it is done in \cite{Doostan2011,Yang2013}. However, we prefer to use the squared residual, in accordance with \cite{Hastie2015}. In addition to the design matrix $\mat{A}$ and the vector of samples $\bm{y}$, the cross-validation procedure takes as input:
\begin{itemize}
 \setlength{\topsep}{0pt}
        \setlength{\parskip}{0pt}
        \setlength{\partopsep}{0pt}
        \setlength{\parsep}{0pt}         
        \setlength{\itemsep}{0pt} 
\item a number $G\in \mathbb{N}$ of sample groups;
\item  a decoder $\Delta$ with tuning parameter $p$ such that $\hat{\bm{x}} = \Delta(\mat{A},\bm{y};p)$ approximately solves $\mat{A}\bm{x} = \bm{y}$;
\item  a finite set $\mathcal{P}$ of parameters;
\item  a number $T \in \mathbb{N}$ of repeated random tests to run.
\end{itemize}
The procedure gives as output a tuning parameter $p_{\text{cv}} \in \mathcal{P}$.

\begin{algorithm}
\caption{\label{alg:cv}Cross-validation procedure}

\begin{algorithmic}[1]

\STATE{\textbf{procedure} $p_{\text{cv}}$ = \textsc{CrossValidation}$(\mat{A},\bm{y},G,\Delta,\mathcal{P} ,T)$;}
\FORALL{$t\in[T]$} 
\STATE{Randomly partition $[m] = I_1 \sqcup \cdots \sqcup I_G$ into $G$ sets of cardinality $\lfloor m/G \rfloor$ or $\lfloor m/G \rfloor +1$;}
\FORALL{$g \in[G]$}
\STATE{$\mat{A}_v = \sqrt{\frac{m}{|I_{g}|}} \, (A_{ij})_{i \in I_{g}, j \in [n]}, \quad \bm{y}_v = \sqrt{\frac{m}{|I_{g}|}} \,(y_i)_{i \in I_{g}}$;}
\STATE{$\mat{A}_r = \sqrt{\frac{m}{m-|I_{g}|}} \, (A_{ij})_{i \in [m]\setminus I_{g},j \in [n]}, \quad \bm{y}_r = \sqrt{\frac{m}{m-|I_{g}|}} \,(y_i)_{i \in [m]\setminus I_{g}}$;}  
\FORALL{$p \in \mathcal{P}$} 
\STATE{$\hat{\bm{x}} = \Delta(\mat{A}_r, \bm{y}_r ; p)$; }
\STATE{$\varepsilon(t,g,p) = \|\mat{A}_v \hat{\bm{x}}-\bm{y}_v\|_2^2$; \label{step:valid_err}}
\ENDFOR
\ENDFOR
\ENDFOR
\STATE{$p_{\text{cv}} = \displaystyle \arg\min_{p \in \mathcal{P}} \frac{1}{T \cdot G }\sum_{t \in [T]} \sum_{g \in[G] } \varepsilon(t,g,p)$. }
\end{algorithmic}
\end{algorithm}

\subsection{Tuning parameter \emph{vs.}\ error} 
\label{sec:par_vs_err}

We compare WQCBP, WLASSO, WSR-LASSO, and WLAD-LASSO by studying the behavior of the recovery error as a function of the tuning parameter when approximating the function 
\begin{equation}
\label{eq:par_vs_err_fun}
f (\bm{t}) = \exp\left(-\frac{1}{d}\sum_{\ell=1}^d\cos(t_\ell)\right), \quad \text{with }d = 15.
\end{equation}
The aim of this experiment is to both compare the performance of the four decoders, and to validate the optimal choices of the tuning parameters suggested by the theory (see \S\ref{sec:rec_err}). 

We consider a sparsity level $s = 10$, corresponding to $n = |\Lambda^{\text{HC}}_{15,10}| = 1432$ and  a number of samples $m = \lceil s^\gamma \log(n)\rceil$, with $\gamma$ defined as in \eqref{eq:defgamma}. In particular, this corresponds to a sample complexity  $m=727$ for Legendre and $m=280$ for Chebyshev polynomials. Then, we repeat the following experiment 50 times: we generate $m$ random samples and corrupt them by random noise $\bm{e} = \beta\bm{n}/\|\bm{n}\|_2$ with $\beta \in\{ 0, 10^{-3}, 10^{-2}, 10^{-1}\}$, where $\bm{n}\in\mathbb{R}^m$ is a random vector with independent entries uniformly distributed over $[-1,1]$, and we solve the resulting system \eqref{eq:linsys} by means of WQCBP, WLASSO, WSR-LASSO, and WLAD-LASSO for each value of the tuning parameter, as specified in Table~\ref{tab:par_range}.\footnote{In practice, for WLAD-LASSO we use $\lambda = 1.01$ instead of $\lambda = 1$ since the choice $\lambda = 1$ leads to the presence of spurious outliers in the box plot. We think that this behavior is due to CVX and not to the decoder itself.}

\begin{table}
\centering
\begin{tabular}{c|c|c|c}
WQCBP ($\eta$) & WLASSO ($\lambda$) & WSR-LASSO ($\lambda$) & WLAD-LASSO ($\lambda$)\\\hline
10.$\myxor(-7 : 0.5 : 1)$ & 10.$\myxor(-1:0.5:8)$ & 10.$\myxor(-2:0.25:5)$ & 10.$\myxor(-2:0.25:3)$\\
\end{tabular}
\caption{\label{tab:par_range}Sets of tuning parameters for WQCBP, WLASSO, WSR-LASSO, and WLAD-LASSO used to generate Figs.~\ref{fig:param_vs_err_Leg} \& \ref{fig:param_vs_err_Che}.}%\protect\footnotemark}}
\end{table}

\begin{figure}[t]
\centering
\includegraphics[width = 7cm]{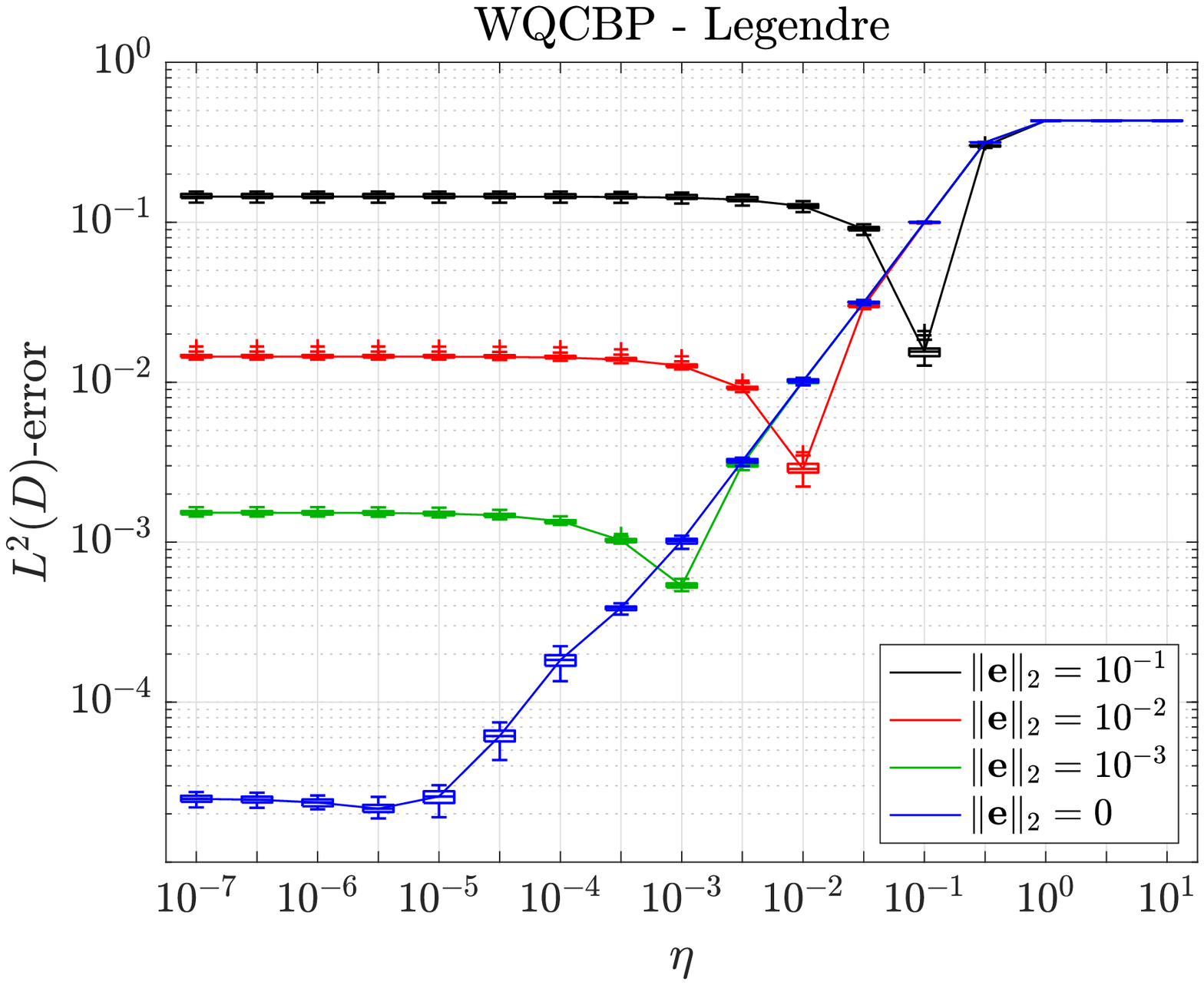}
\includegraphics[width = 7cm]{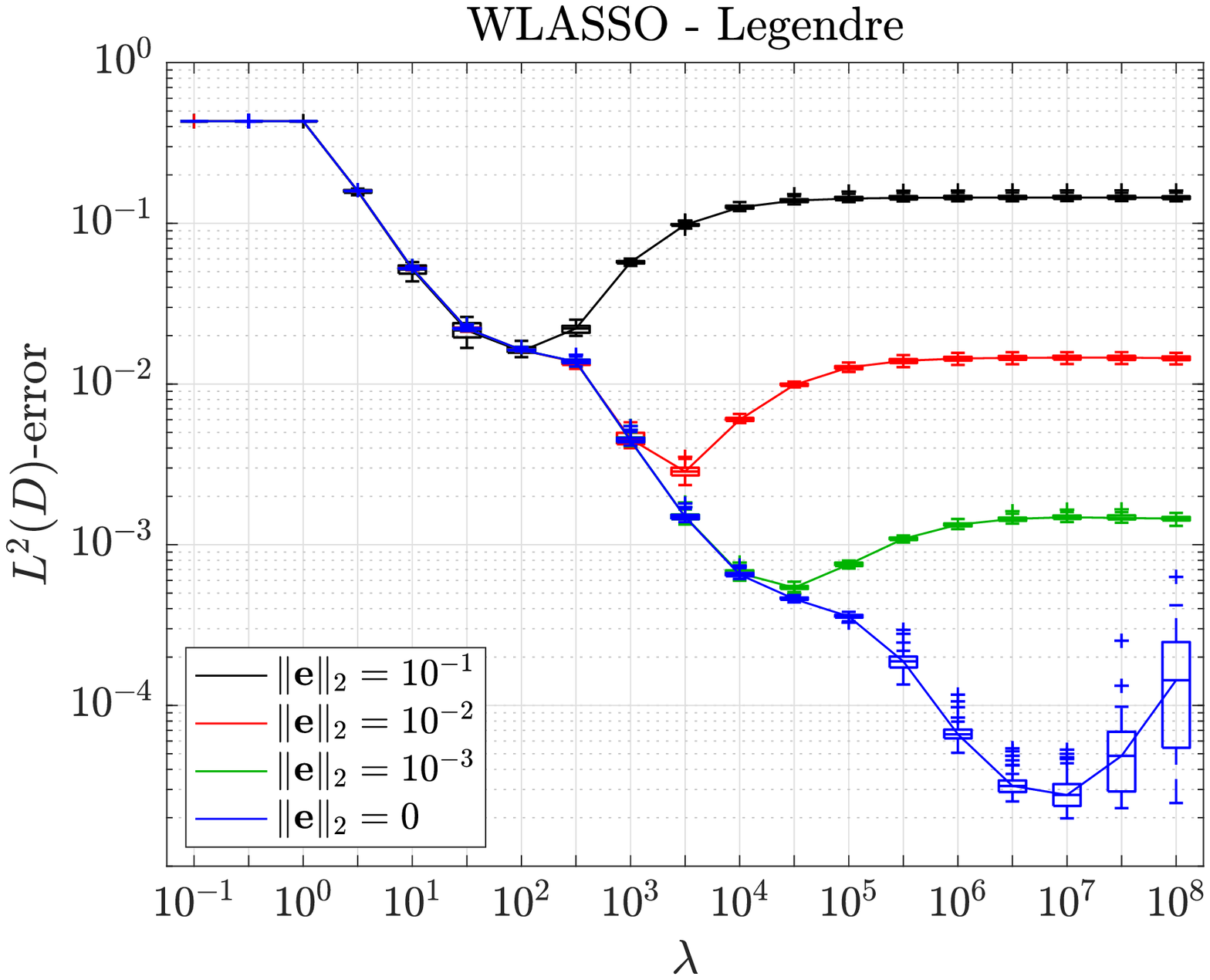}\\
\includegraphics[width = 7cm]{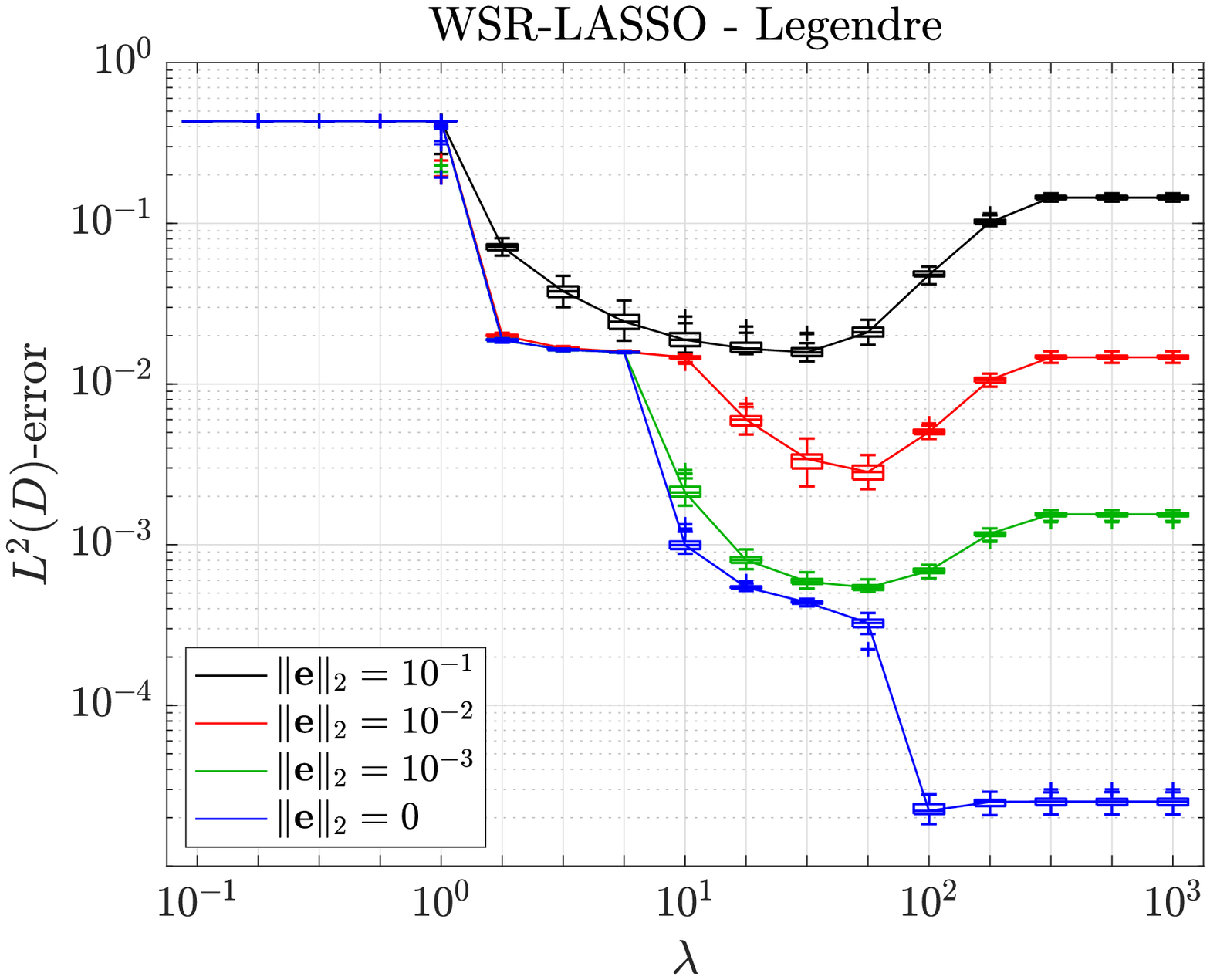}
\includegraphics[width = 7cm]{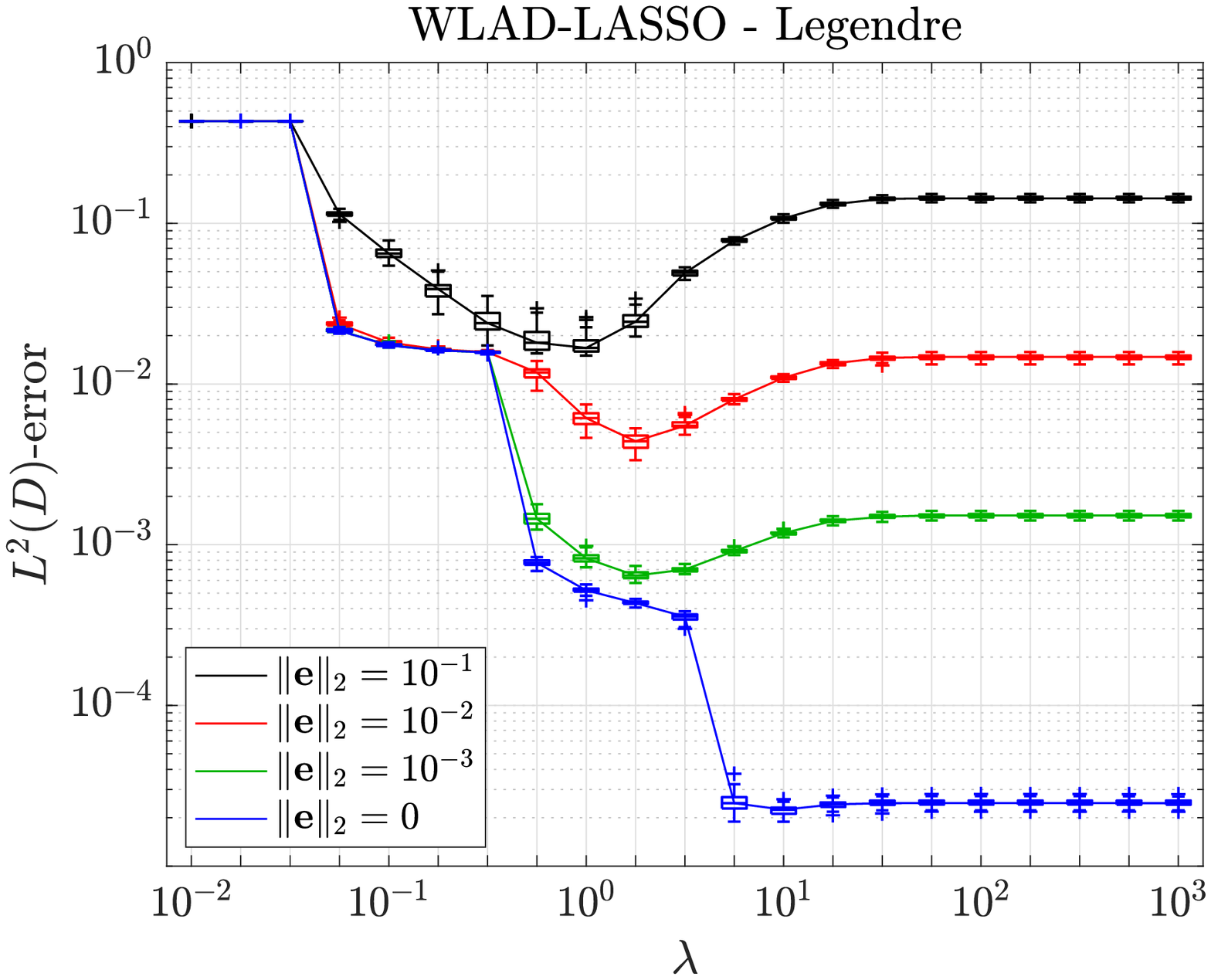}
\caption{\label{fig:param_vs_err_Leg} Box plot of the recover $L^2_\nu$ error as a function of the tuning parameter for WQCBP (top left), WLASSO (top right), WSR-LASSO (bottom left), and WLAD-LASSO (bottom right) when computing the sparse approximation of the function $f$  defined in \eqref{eq:par_vs_err_fun} with respect to tensorized Legendre polynomials.}
\end{figure}

\begin{figure}[t]
\centering
\includegraphics[width = 7cm]{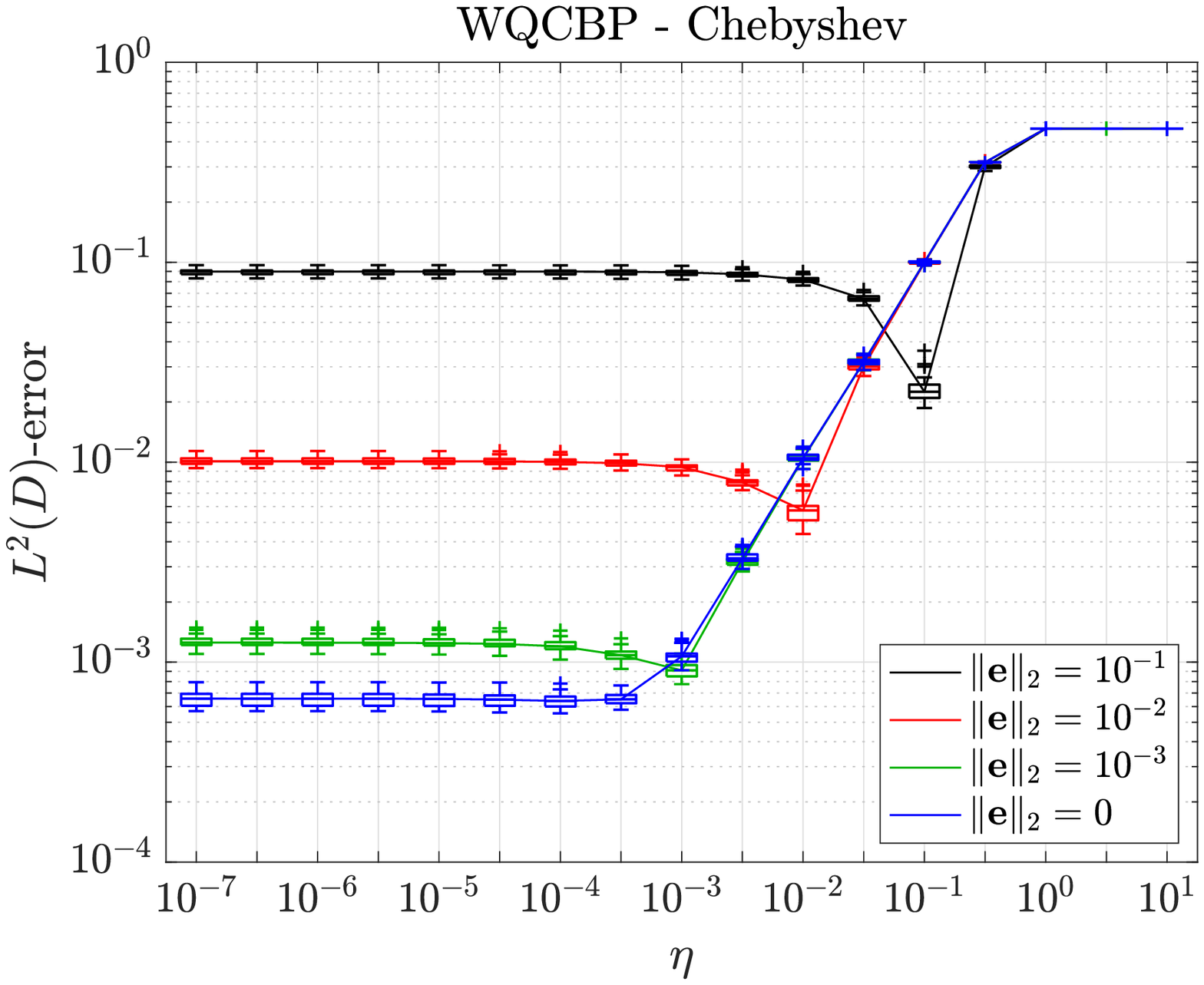}
\includegraphics[width = 7cm]{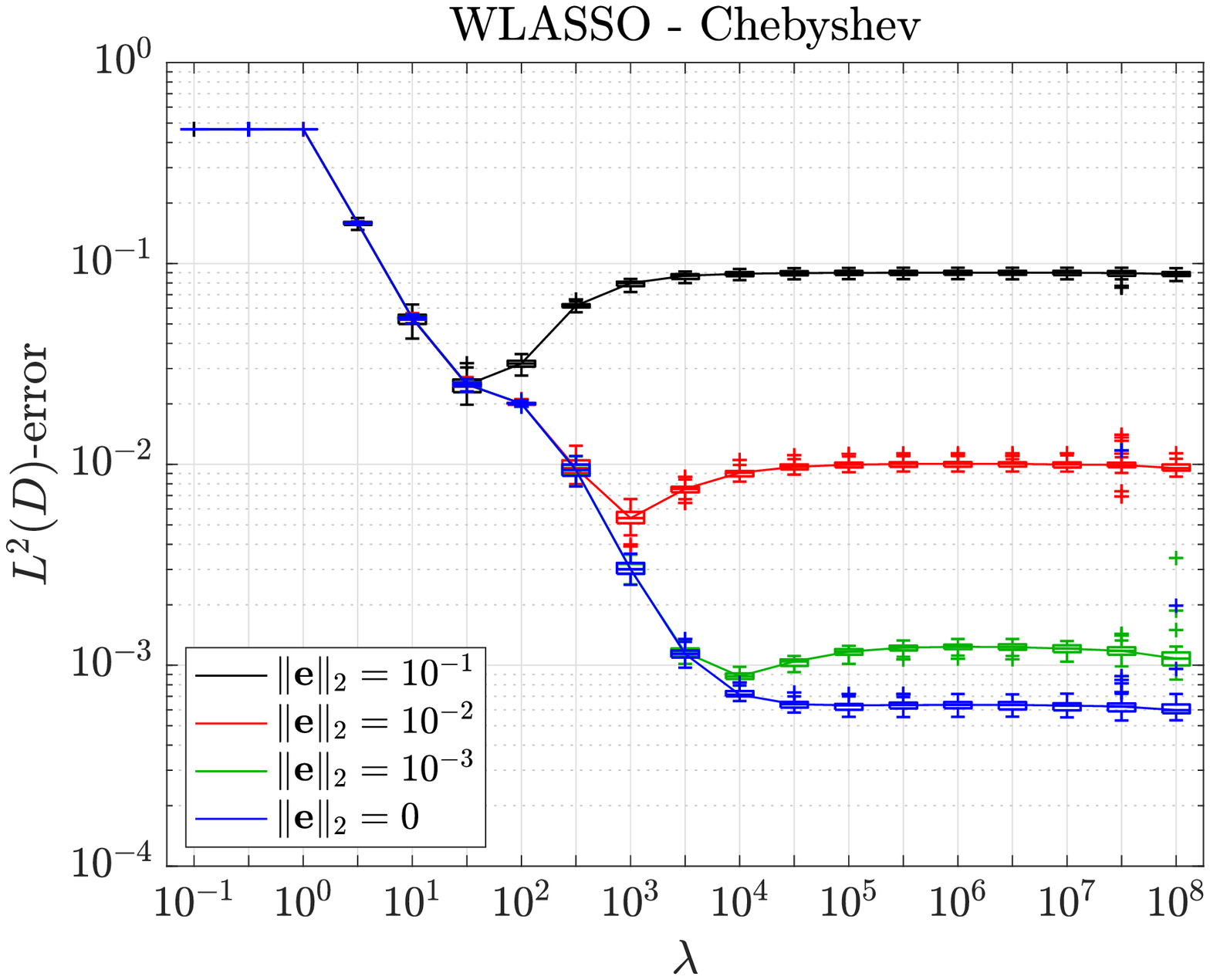}\\
\includegraphics[width = 7cm]{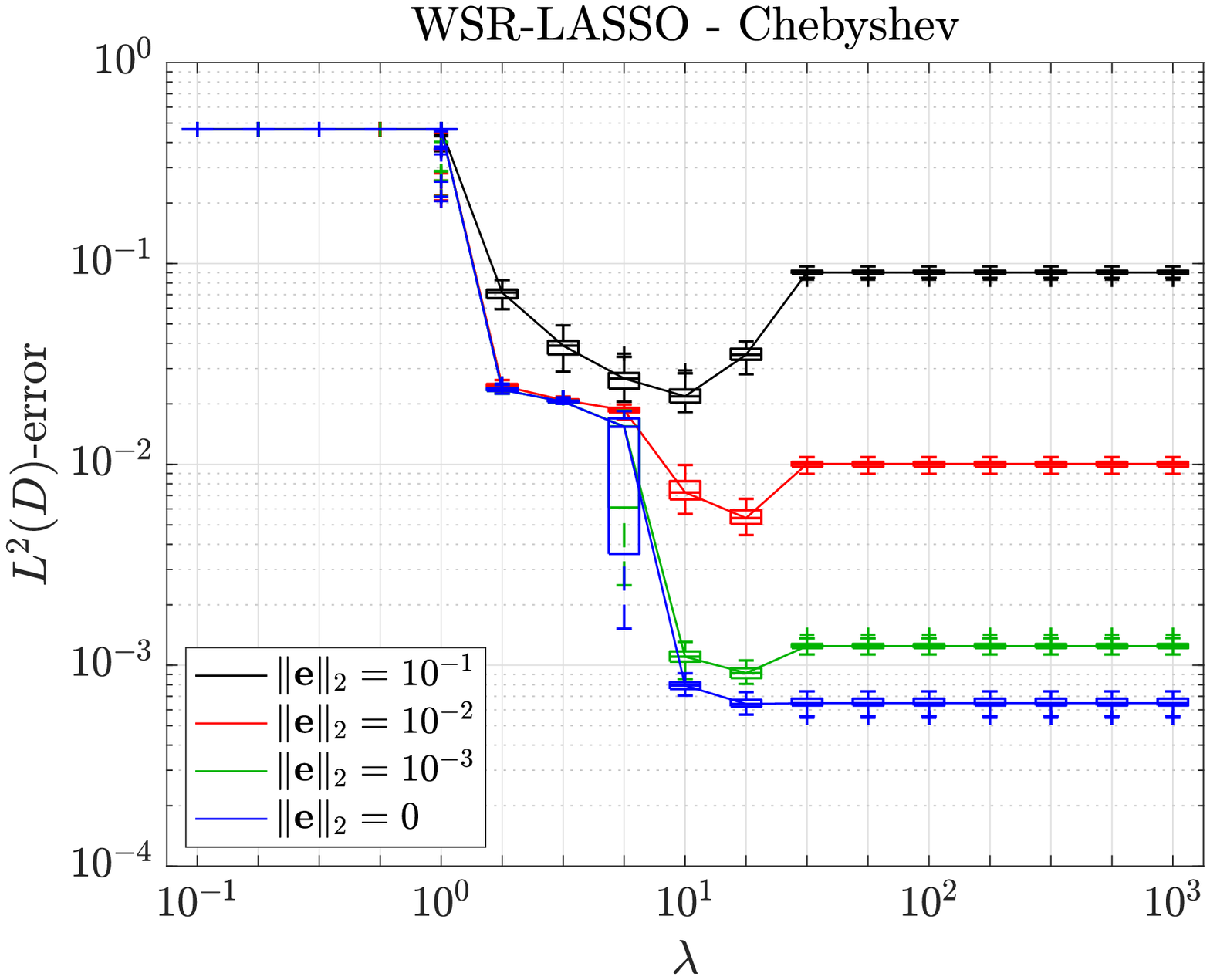}
\includegraphics[width = 7cm]{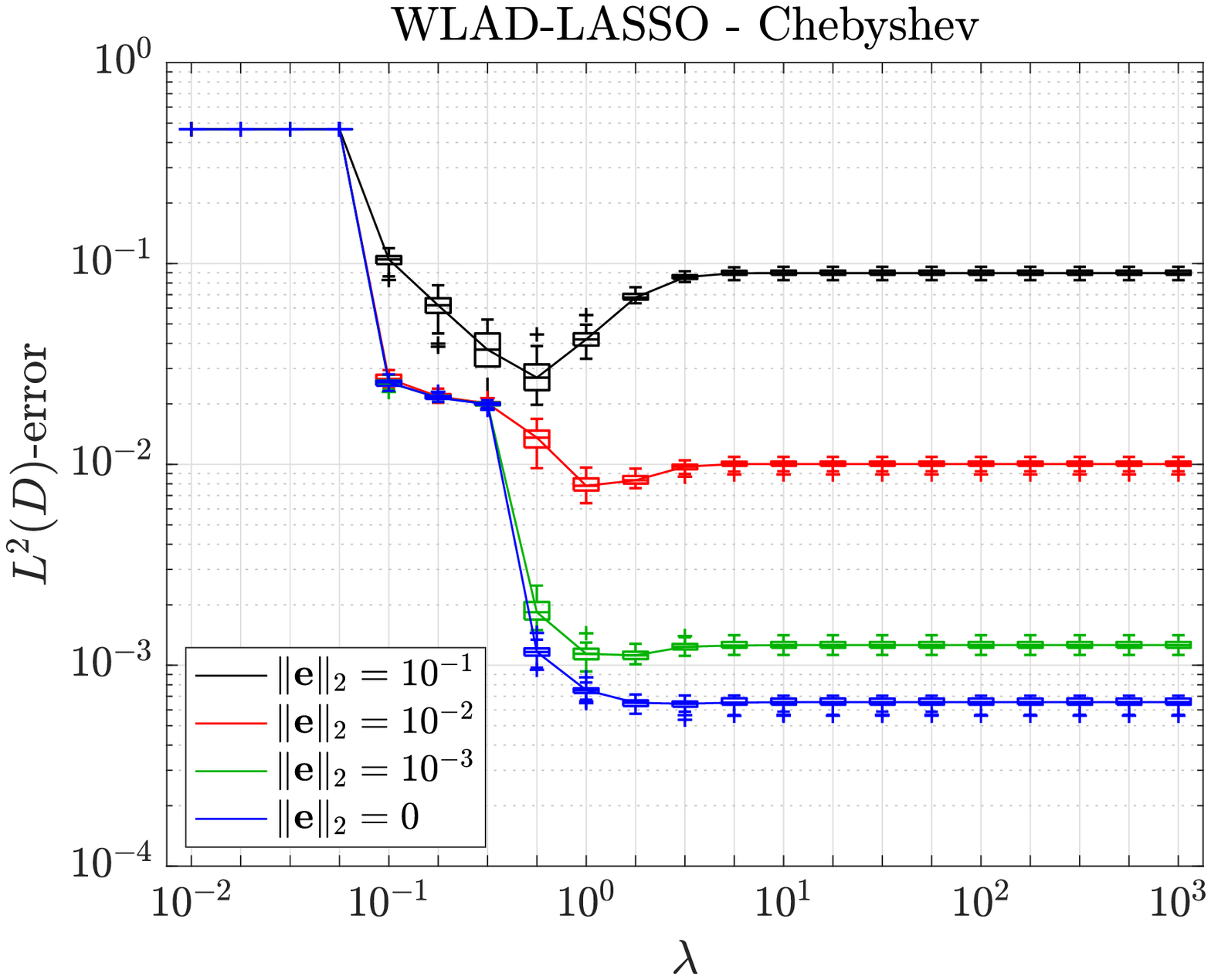}
\caption{\label{fig:param_vs_err_Che}Same experiment as in Fig.~\ref{fig:param_vs_err_Leg}, but for tensorized Chebyshev polynomials.}
\end{figure}

The results are shown in Figs.~\ref{fig:param_vs_err_Leg} and~\ref{fig:param_vs_err_Che}. In order to make the statistics of the randomized experiments transparent, we visualize the results in a box plot. The continuous lines represent the median values (in accordance with the box plot information). We consider the $L^2_\nu$ error with respect to a high-fidelity approximation to $f$, computed using least squares and $20 n = 28620$ random pointwise evaluations.\footnote{Notice that the outliers are sometimes aligned (e.g., in the tail of the blue curve in Fig.~\ref{fig:param_vs_err_Leg} bottom right). This is due to the structure of the proposed numerical experiment: for each randomized choice of samples, all the parameters are tested using the same samples.}

It is remarkable that all the decoders are able to reach an accuracy below the noise level for suitable values of the parameter. Moreover, the optimal choice of the tuning parameters seems to be in accordance with the theory provided in \S\ref{sec:rec_err}. Indeed, we observe a direct proportionality between the optimal value of $\eta$ and the noise level for WQCBP, and an inverse proportionality between the optimal value of $\lambda$ and the noise level for WLASSO. For WSR-LASSO and WLAD-LASSO, the optimal values of $\lambda$ are independent on the noise level. We also notice that the limit value of the $L^2_\nu$ recovery error associated with the unconstrained optimization programs for $\lambda \to \infty$ coincides with the limit value of the $L^2_\nu$ error associated with WQCBP for $\eta \to 0$. This suggests that the solution to the unconstrained programs tends to the solution of WBP as $\lambda \to \infty$. In other words, choosing a very large $\lambda$ forces the data fidelity constraint to be realized exactly, as it is natural to expect.

Let us take a closer look to the performance of WSR-LASSO and WLAD-LASSO, recalling the optimal choices for $\lambda$ given by \eqref{eq:lambdaWSRLASSO} and \eqref{eq:lambdaWLADLASSO}, respectively. For WSR-LASSO, being $s = 10$, we have $\sqrt{K(s)} \approx 10$ for Legendre and $\sqrt{K(s)}\approx \sqrt{10^{\log(3)/(2\log(2))}} \approx 6.2$  for Chebyshev polynomials. From the box plots, the optimal value of $\lambda$ seems to be around $10^{1.5} \approx 31.6$ (Legendre) and $10^{1.25}\approx17.8$ (Chebyshev). Therefore, $\lambda \approx 3 \sqrt{K(s)}$. For the WLAD-LASSO, in principle the theory predicts that $\lambda \approx \sqrt{K(s)/K}$. Since the error $\bm{e}$ is nonsparse in this case, we have $K = m \approx K(s) \log(n)$. We see that $\lambda \approx 3/\sqrt{\log(n)} \approx 1.1129$  seems to match the numerics well. This is the reason why we choose the hidden constant to be $3$ in \eqref{eq:lambdaWSRLASSO} and \eqref{eq:lambdaWLADLASSO} for the subsequent experiments.

\subsection{Sampling complexity \emph{vs.} error}
\label{sec:m_vs_err}

We make another comparison of the four decoders by studying the decay of the $L^2_\nu$ error as a function of the sampling complexity $m$ when approximating the function $f$ defined in \eqref{eq:par_vs_err_fun} with $d = 10$, employing tensorized Chebyshev polynomials. We set $s = 15$, corresponding to $n = |\Lambda^{\text{HC}}_{10,15}| = 1341$ and  consider a number of samples $m = \lceil C\cdot K(s)\rceil$ for $C = 2:0.5:4$, where $K(s) \approx 15^{\log(3)/\log(2)} \approx 73.1$. First, we compute a high-fidelity solution $\hat{\bm{x}}_{\text{LS}}$ by means of least squares, using $20\cdot n = 26820$ random pointwise samples. Then, we consider ten different combinations of solvers and parameters. For the sake of the comparison, we use the high-fidelity solution $\hat{\bm{x}}_{\text{LS}}$ to produce ``oracle'' estimates of the noise level (notice that this is an idealized scenario since $\hat{\bm{x}}_{\text{LS}}$ is not available in practice). Moreover, we employ  the cross-validation procedure described in Algorithm~\ref{alg:cv}, and the values of the parameters suggested by the theoretical analysis for WSR-LASSO ($\lambda = 3\sqrt{K(s)}$) and WLAD-LASSO ($\lambda = 3/\sqrt{\log(n)}$) (see \S\ref{sec:rec_err} and the discussion at the end of \S\ref{sec:par_vs_err}). The cross-validation procedure is always employed as $\textsc{CrossValidation}(\mat{A},\bm{y},5,\Delta,\mathcal{P},3)$, where $\Delta$ and $\mathcal{P}$ are specified on a case-by-case basis. The ten combinations considered are the following:
\begin{enumerate}
 \setlength{\topsep}{0pt}
        \setlength{\parskip}{0pt}
        \setlength{\partopsep}{0pt}
        \setlength{\parsep}{0pt}         
        \setlength{\itemsep}{0pt} 
\item WQCBP with $\eta = 0$;
\item WQCBP with  $\eta = \eta_{\text{oracle}} = \|\mat{A}\hat{\bm{x}}_{\text{LS}}-\bm{y}\|_2$;
\item  WQCBP with cross validation, where $\mathcal{P} = \eta_{\text{oracle}} \cdot  \text{10.$\myxor$(-2:0.5:2)}$;
\item WLASSO with $\lambda = \lambda_{\text{oracle}} = \sqrt{K(s)}/\|\mat{A}\hat{\bm{x}}_{\text{LS}} - \bm{y}\|_2$;
\item WLASSO with cross validation, where  $\mathcal{P} = \lambda_{\text{oracle}} \cdot  \text{10.$\myxor$(-2:0.5:2)}$;
\item WSR-LASSO with $\lambda = 3 \sqrt{K(s)}$;
\item WSR-LASSO with , where $\mathcal{P} = 3\sqrt{K(s)} \cdot  \text{10.$\myxor$(-2:0.5:2)}$;
\item WLAD-LASSO with $\lambda = 3/\sqrt{(k/m)\log(n)}$;
\item WLAD-LASSO with cross validation, where  $\mathcal{P} = \frac{3\sqrt{H}}{\sqrt{m\log(n)}} \cdot  \text{10.$\myxor$(-2:0.5:2)}$;
\item WLAD-LASSO with $\lambda = 1$.
\end{enumerate}
Notice that for WLAD-LASSO we set $\bm{v} = 1$ in \eqref{eq:WLADLASSO}.

This comparison is performed in three different scenarios. First, without corrupting the samples, thus having only truncation error  \eqref{eq:deftrunc} (Fig.~\ref{fig:m_vs_err}, top right). Second, corrupting the measurement by  random noise $\bm{e} = 10^{-2}\bm{n}/\|\bm{n}\|_2$, where $\bm{n}$ is a random vector with independent entries uniformly distributed over $[-1,1]$ (Fig.~\ref{fig:m_vs_err}, bottom left). Third, corrupting $10\%$ of the measurements by a random noise uniformly distributed over $[-10,10]$ (Fig.~\ref{fig:m_vs_err}, bottom right).  The respective curves show the $L^2_\nu$ error with respect to the high-fidelity solution, averaged over 25 trials.
\begin{figure}[t]
\centering
\begin{tabular}{cc}
\raisebox{1cm}{\includegraphics[width = 5cm]{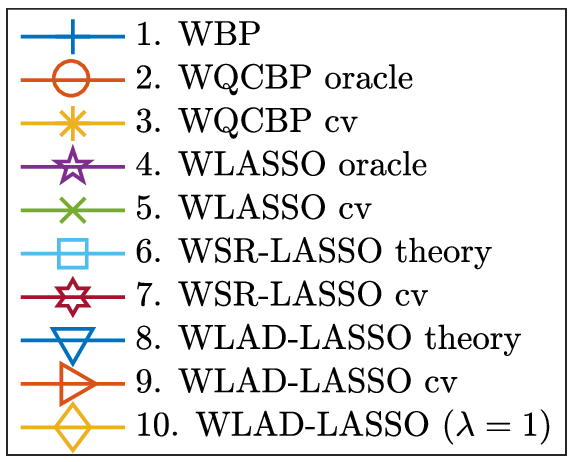}} &
\includegraphics[width = 7cm]{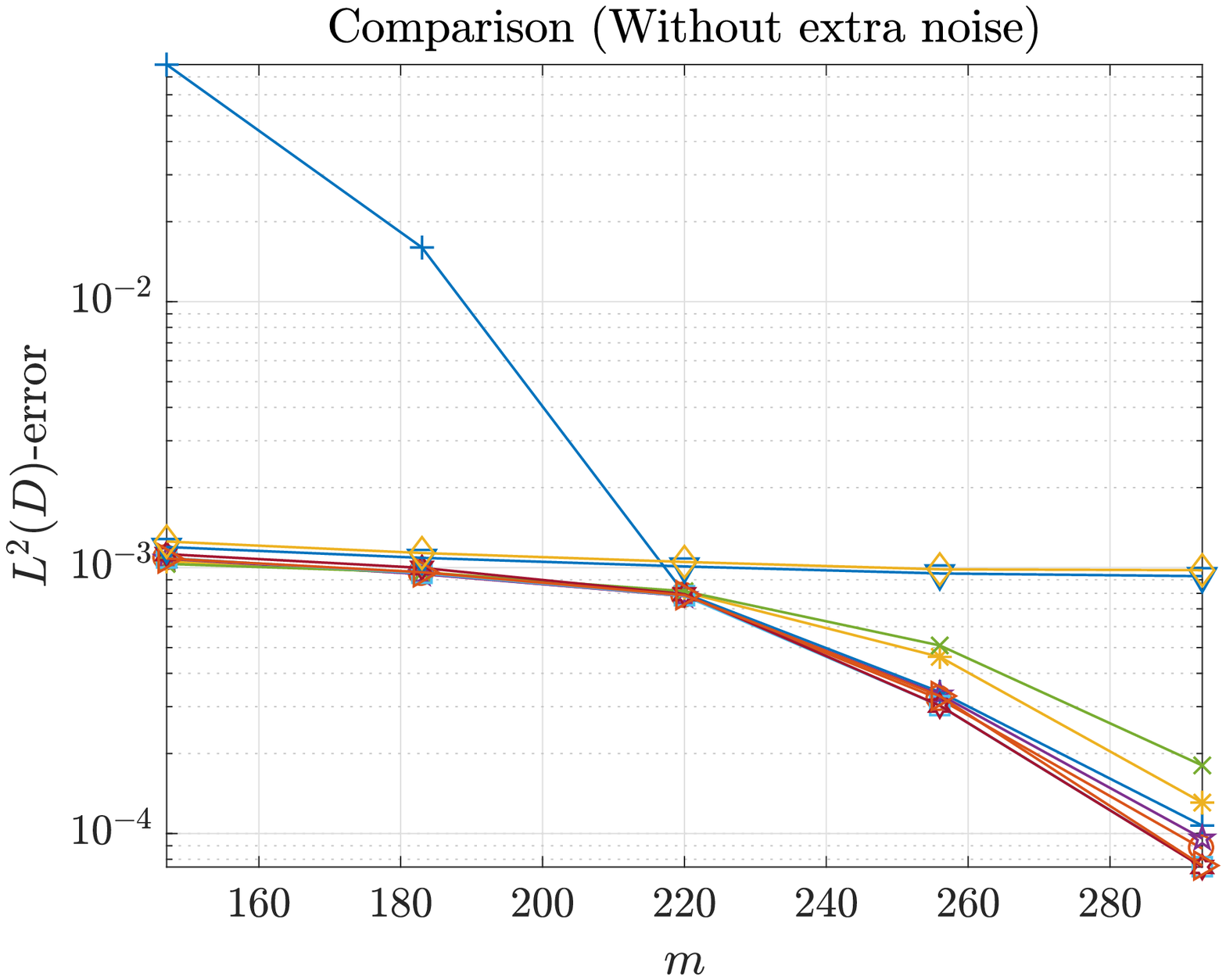} \\
\includegraphics[width = 7cm]{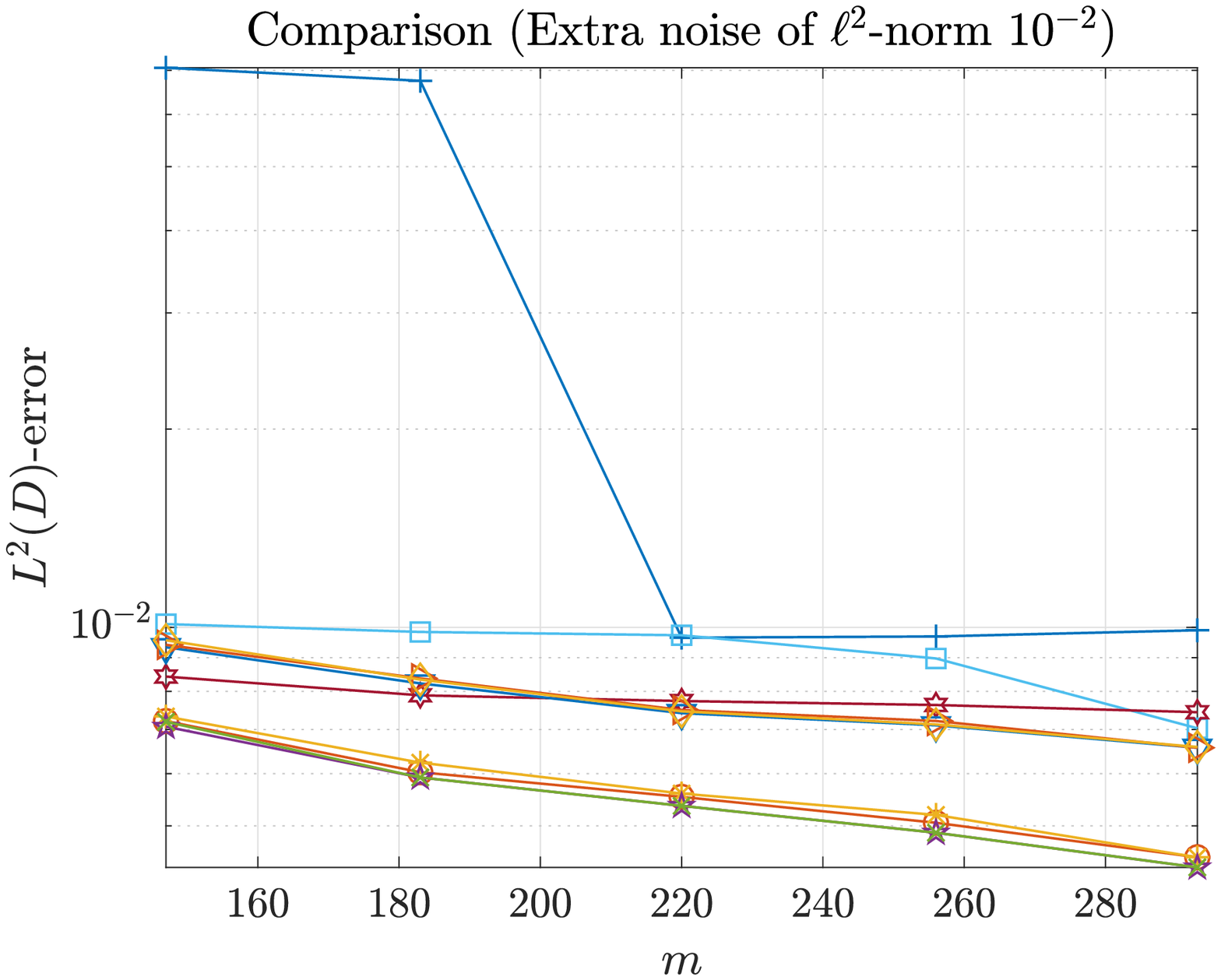} &
\includegraphics[width = 7cm]{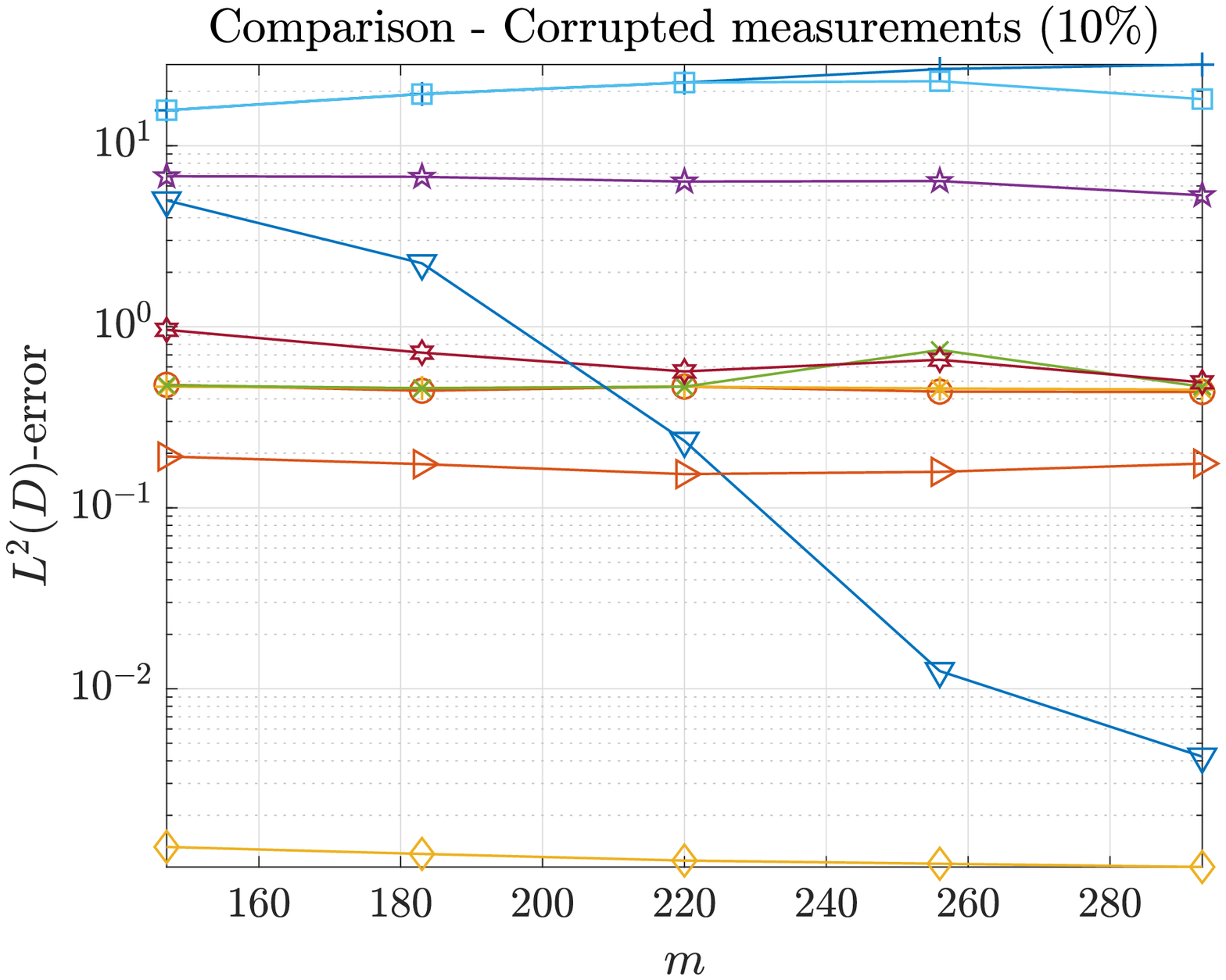}  
\end{tabular}
\caption{\label{fig:m_vs_err}Comparison among WQCBP, WLASSO, WSR-LASSO, and WLAD-LASSO for the approximation of \eqref{eq:par_vs_err_fun} with $d = 10$. The ten combinations of decoder and parameter considered are described in \S\ref{sec:m_vs_err}.}
\end{figure}

Let us comment the results in Fig.~\ref{fig:m_vs_err}. First, the choices of the parameters predicted by the theory give good results in general. Moreover, the performance of each decoder is usually improved by cross validation. WBP slightly underperforms with respect to all the other solvers apart from the case without extra noise and for $m$ large enough (Fig.~\ref{fig:m_vs_err} top right). The performance of WSR-LASSO is comparable with those of WQCBP and WLASSO, but for WSR-LASSO the optimal parameter choice does not depend on the noise level. Finally, in the case of sparsely corrupted measurements, WLAD-LASSO outperforms the other decoders, as expected (Fig.~\ref{fig:m_vs_err} bottom right). The choice $\lambda  = 1$ is particularly good in the sparsely corrupted case. It is also worth noting that cross validation does not seem to help in this case. This is natural, since Algorithm~\ref{alg:cv} chooses the tuning parameter that minimizes the $\ell^2$ norm of the residual, which is not small in this case. A possible remedy could be to replace step~\ref{step:valid_err} in Algorithm~\ref{alg:cv} with $\varepsilon(t,g,p) = \|\mat{A}_v\hat{\bm{x}}-\bm{y}_v\|_1$.\\

We  note in passing that the experiments in Figs.\ref{fig:param_vs_err_Leg}, \ref{fig:param_vs_err_Che} \& \ref{fig:m_vs_err} yield analogous results when the random noise $\bm{n}$ used to corrupt the samples has independent normal Gaussian  entries $\mathcal{N}(0,1)$. We do not add the resulting figures for the sake of brevity.

%\RED{[The section ``WQCBP robustness'' about the numerical study of $\mathcal{Q}$ has been removed.]}

%\subsection{WQCBP robustness}
%\label{sec:sigma_min}

%In this experiment, we show that the constant $\mathcal{Q}_{\bm{u}}(\mat{A})$ defined in \eqref{eq:defQuA}, measuring the quality of the bound to $\mathcal{T}$ given in Theorem~\ref{thm:WQCBPtailbound} (up to log factors), has moderate size. In Fig.~\ref{fig:WQCBP_constant}, we consider  $d = 15$ and plot the average of $\mathcal{Q}_{\bm{u}}(\mat{A})$ over 25 runs for Legendre (left) and Chebyshev (right) polynomials using a logarithmic scale.
%\begin{figure}
%\centering
%\includegraphics[width = 7cm]{fig/Fig4_dim_Leg_15}
%\includegraphics[width = 7cm]{fig/Fig4_dim_Che_15}
%\caption{\label{fig:WQCBP_constant} The constant $\mathcal{Q}_{\bm{u}}(\mat{A})$ defined in \eqref{eq:defQuA} involved in the recovery estimate of WQCBP given in Theorem~\ref{thm:WQCBPerror-blind}, averaged over $25$ trials.}
%\end{figure}
%We let $s = (6:2:14)$ and $m = (200:200:n)$ for each $s$, where $n = |\Lambda^{\text{HC}}_{15,s}|$. We observe that the constant has always moderate size. The largest values are observed when $m \to n$, which is not the regime we are interested in. We also note that $\mathcal{Q}_{\bm{u}}(\mat{A})$ is smaller for Chebyshev polynomials.

\subsection{Application to parametric ODEs and PDEs with random inputs} 
\label{sec:UQ}
Finally, we compare the performance of the decoders under exam when the function begin approximated is a quantity of interest of the solution to a parametric ODE or PDE random inputs. 
%\PURP{For the PDE case, we consider a diffusion equation in \S\ref{sec:PDEs}. As for ODEs, we study the dynamics of a damped harmonic oscillator  in \S\ref{sec:ODEs}. }

\subsubsection{Parametric PDEs: Diffusion equation}
\label{sec:PDEs}

We consider the following diffusion equation:
\begin{equation}
\label{eq:cookies}
\begin{cases}
-\nabla \cdot (a_\bm{t}(x,y) \nabla u_\bm{t}(x,y)) = F, & \forall (x,y) \in \Omega,\\
u_\bm{t}(x,y) = 0,&  \forall (x,y) \in \partial \Omega,
\end{cases}
\end{equation}
where $\Omega = [0,1]^2$ is the physical domain, $a_{\bm{t}}$ is a parametric diffusion coefficient, $u_{\bm{t}}$ is the unknown solution, and $F$ is a nonparametric forcing term. This example and analogous variations have been considered in \cite{Back2011,Ballani2015,Chkifa2015a}. The parametric space is $D = (-1,1)^8$ and the diffusion coefficient affinely depends on the parameters 
$$
a_{\bm{t}}(x,y) = 1 - \sum_{\ell=1}^8 \mathbbm{1}_{\Omega_\ell}(x,y) (0.595 + 0.395 t_\ell), \quad \forall (x,y) \in \Omega, \; \forall \bm{t} \in D,
$$
where $\Omega_1,\ldots,\Omega_8$ are circular subregions of $\Omega$ of radius $0.13$ placed symmetrically with respect to the center of $\Omega$ (see Fig.~\ref{fig:meshes}). If we think about $\bm{t}$ as a  random vector uniformly distributed in $D$, then  $a_{\bm{t}}|_{\Omega_\ell}$ can be thought as a random variable uniformly distributed $[0.01, 0.99]$. The forcing term is 
$$
F(x,y) = 100 \cdot \mathbbm{1}_{\Omega_F}(x,y), \quad \forall (x,y) \in \Omega,
$$
where $\Omega_F = [0.4,0.6]^2$.  We are interested in the function defined by the following quantity of interest:
\begin{equation}
\label{eq:function_UQ}
f(\bm{t}) = \int_{\Omega_F} u_{\bm{t}}(x,y) \text{d}x\text{d}y.
\end{equation}
We sample $f(\bm{t})$ by solving the PDE \eqref{eq:cookies} with FreeFem++ (version 3.42, 64 bits), employing triangular P1 elements  and generating finer and finer meshes \cite{FreeFem++}. The mesh is generated with the FreeFem++ mesher, by controlling its resolution by a parameter $N_{el}$. In particular, we parametrize each edge of $\Omega$ using $5N_{el}$ elements, the circumference of each $\Omega_\ell$ with $4N_{el}$ elements and each edge of $\Omega_R$ using $N_{el}$ elements. Then, we let $N_{el} = 1,2,3,5$.  The corresponding meshes are represented in Fig.~\ref{fig:meshes} and have a number of degrees of freedom equal to 56, 141, 267, and 715, respectively. Each sample $f(\bm{t}_i)$ is then computed by considering the integral of the resulting piecewise linear approximation to $u_{\bm{t}_i}$ over $\Omega_F$. In this case, the unknown error affecting the samples contains the truncation error \eqref{eq:deftrunc}, the discretization error due to the Galerkin projection, and the numerical error associated with the FreeFem++ solver. We only expect this unknown error to decrease with respect to $N_{el}$.
\begin{figure}
\centering
\includegraphics[height=3.8cm]{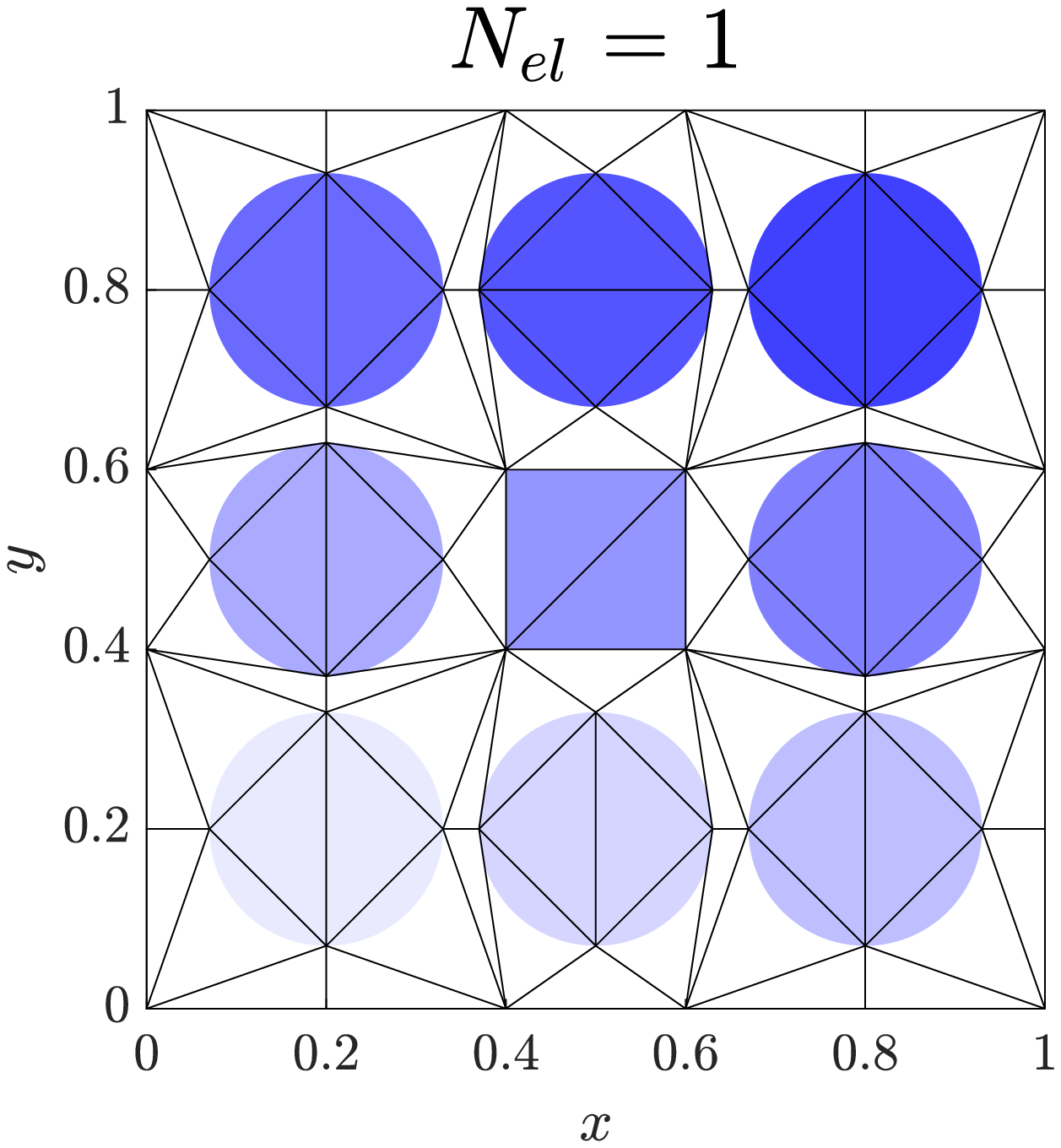}
\includegraphics[height=3.8cm]{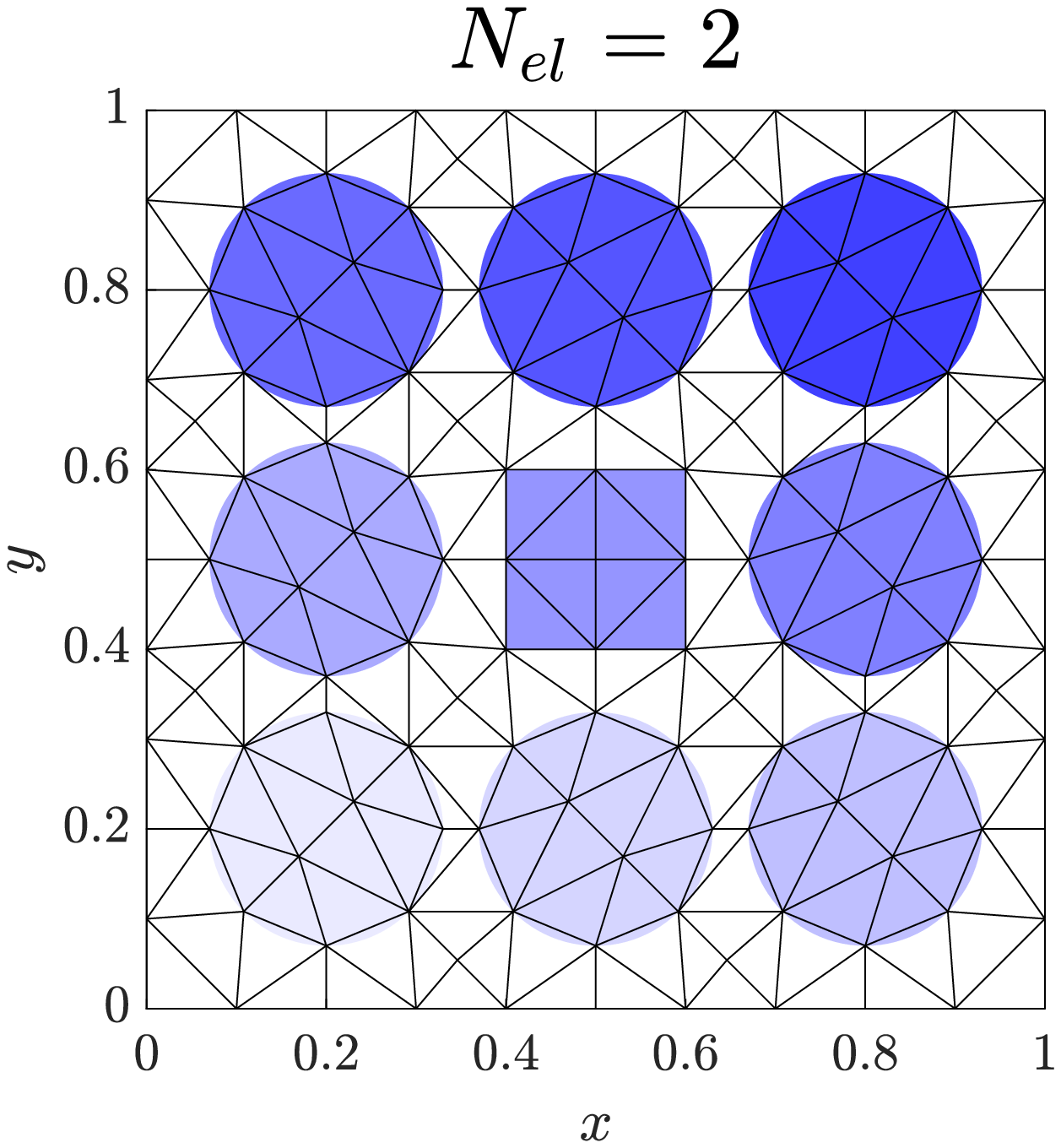}
\includegraphics[height=3.8cm]{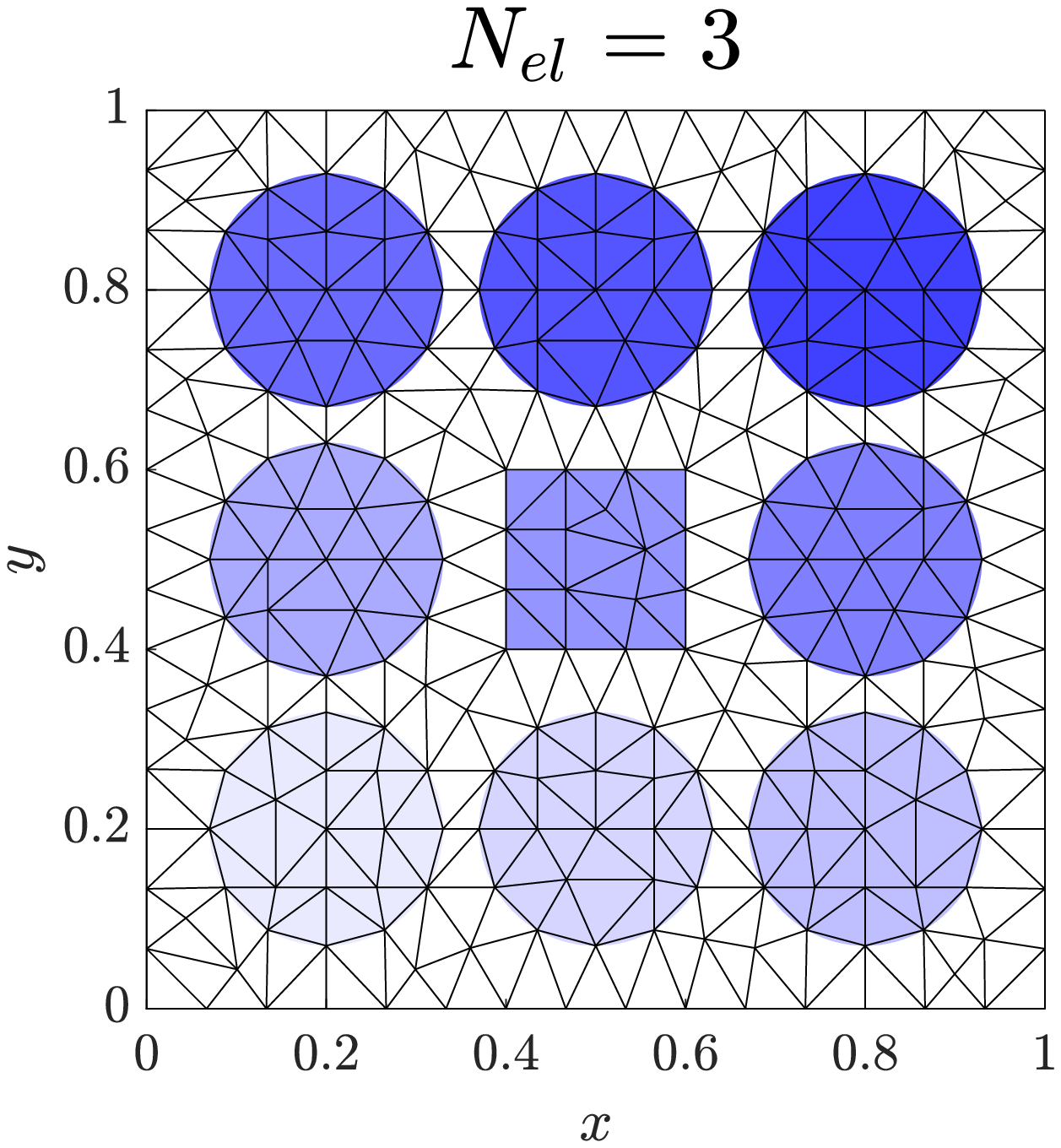}
\includegraphics[height=3.8cm]{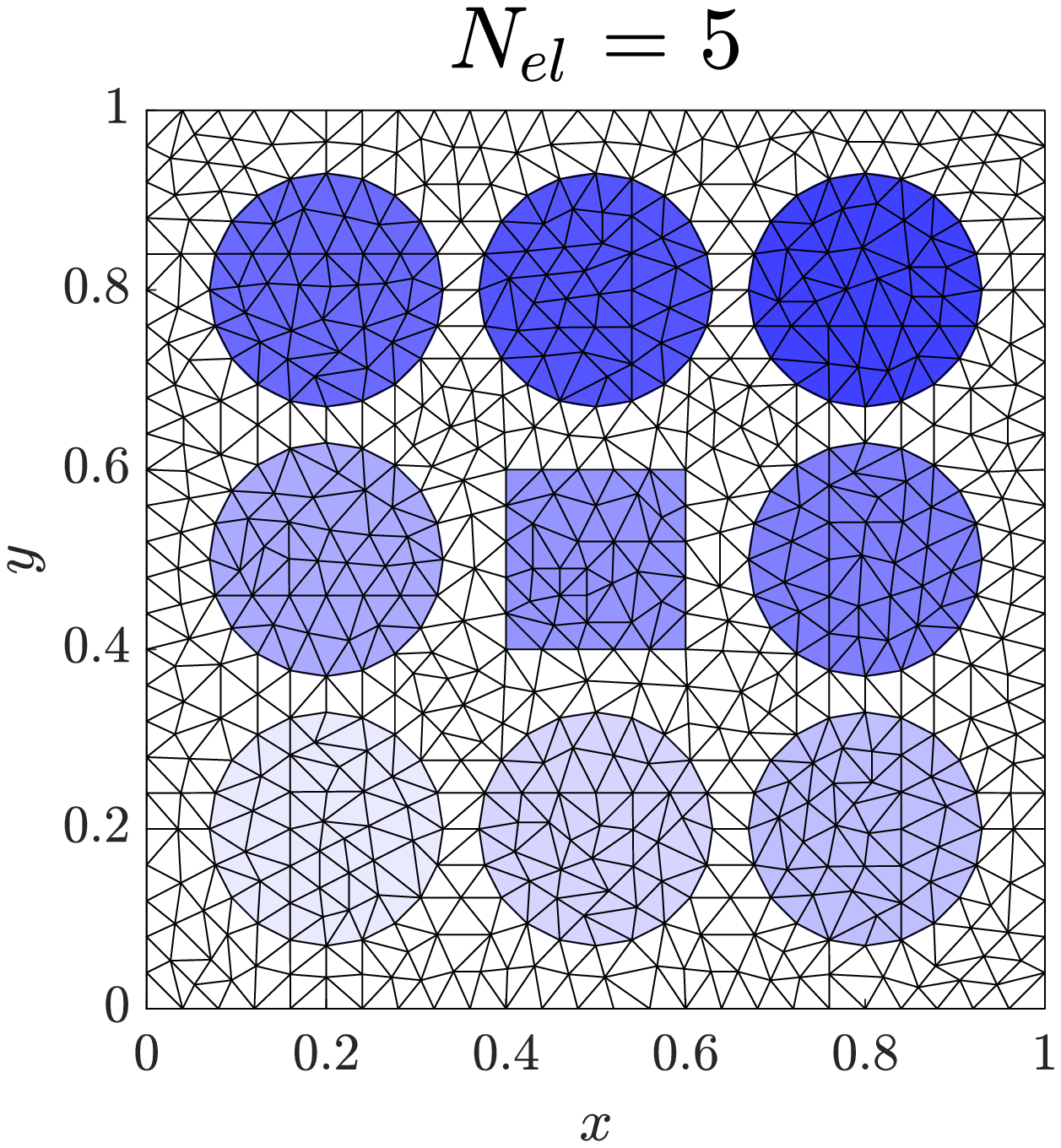}
\caption{\label{fig:meshes}The physical domain $\Omega$ discretized with increasing  mesh resolution, corresponding to $N_{el}=1,2,3,5$ (from left to right). The circular regions $\Omega_1,\ldots,\Omega_8$ and the central square region $\Omega_F$ are highlighted.}
\end{figure}

Similarly to \S\ref{sec:par_vs_err} \& \ref{sec:m_vs_err}, we study the behavior of the recovery error as a function of the tuning parameter and of the sample complexity. Since $\bm{t}$ is uniformly distributed in $D$, we employ tensorized Legendre polynomials.

In Fig.~\ref{fig:UQ_param_vs_err} we consider an experiment analogous  to Figs.~\ref{fig:param_vs_err_Leg} \& \ref{fig:param_vs_err_Che} for the function $f$ defined in \eqref{eq:function_UQ}. Namely, we plot the $L^2_\nu$ error as a function of the tuning parameter (chosen as in Table~\ref{tab:par_range}) for the four decoders. We set $s=10$, resulting in $n = 353$, and $m = 30$. For each value of $N_{el} = 1,2,3,5$, we run 50 random experiments. The $L^2_\nu$ error is computed with respect to a high-fidelity solution computed using a mesh size $N_{el} = 100$ (with 267076 degrees of freedom) and approximated with least squares from 50000 random samples. 
\begin{figure}
\centering
\includegraphics[width = 7cm]{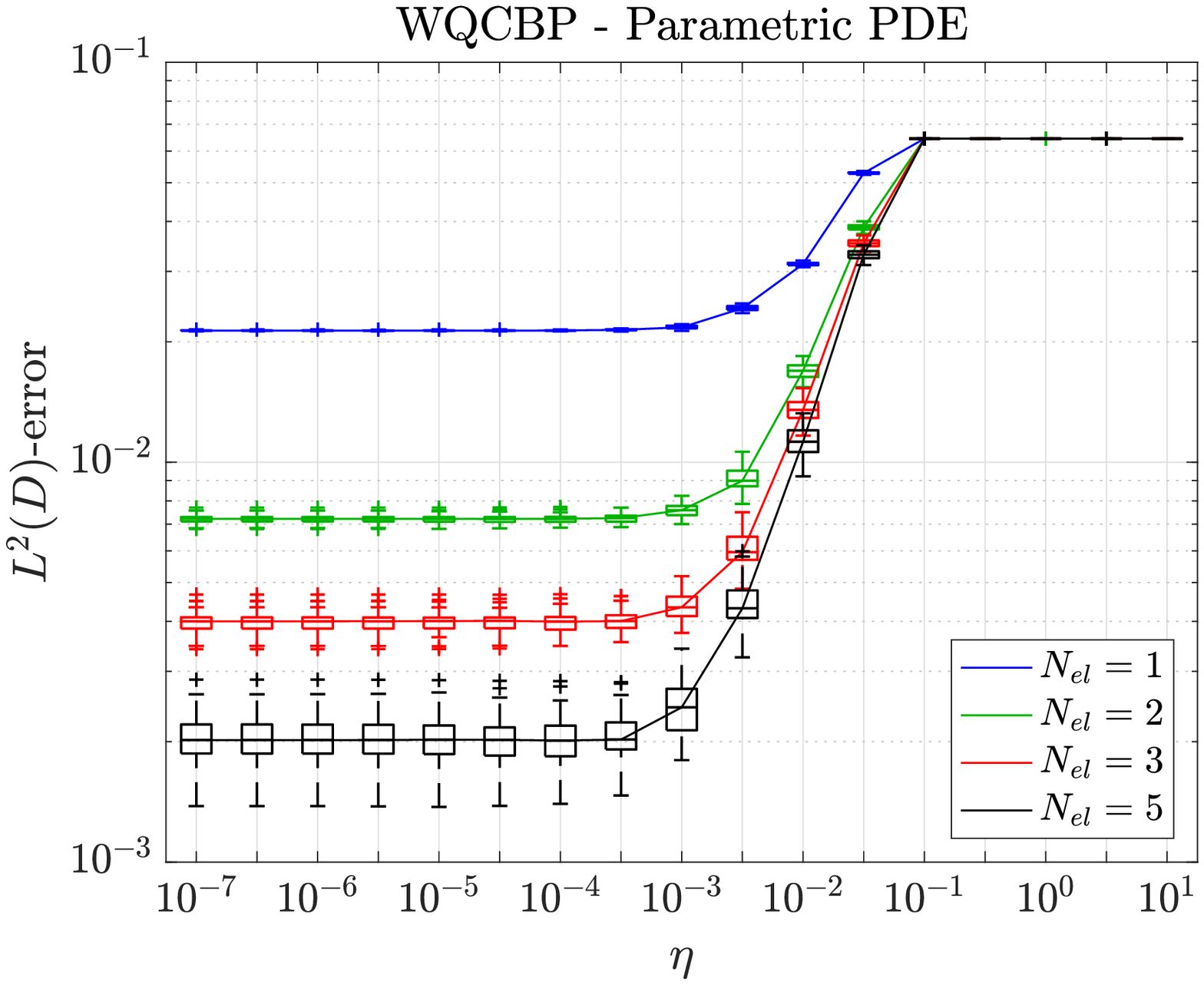}
\includegraphics[width = 7cm]{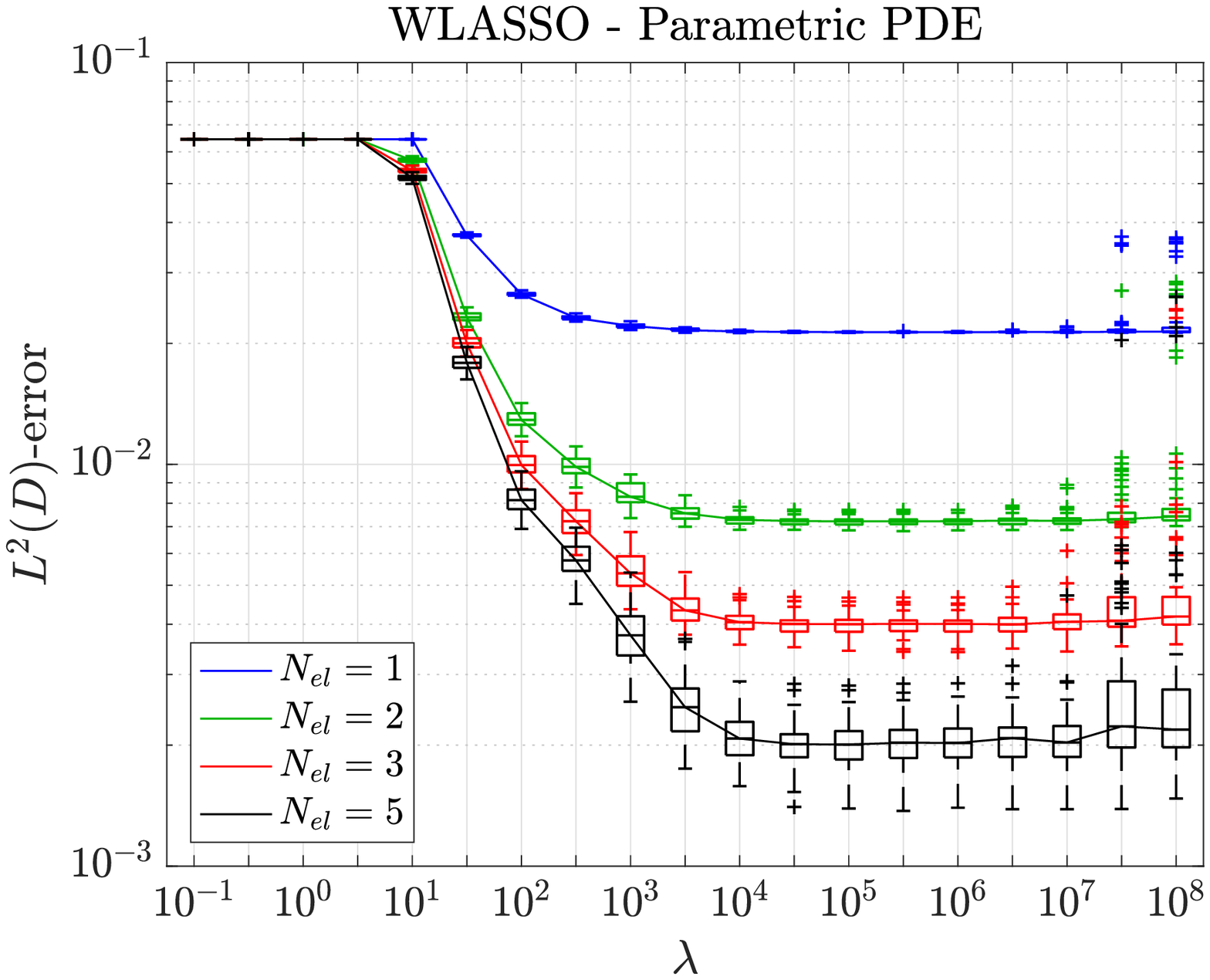}\\
\includegraphics[width = 7cm]{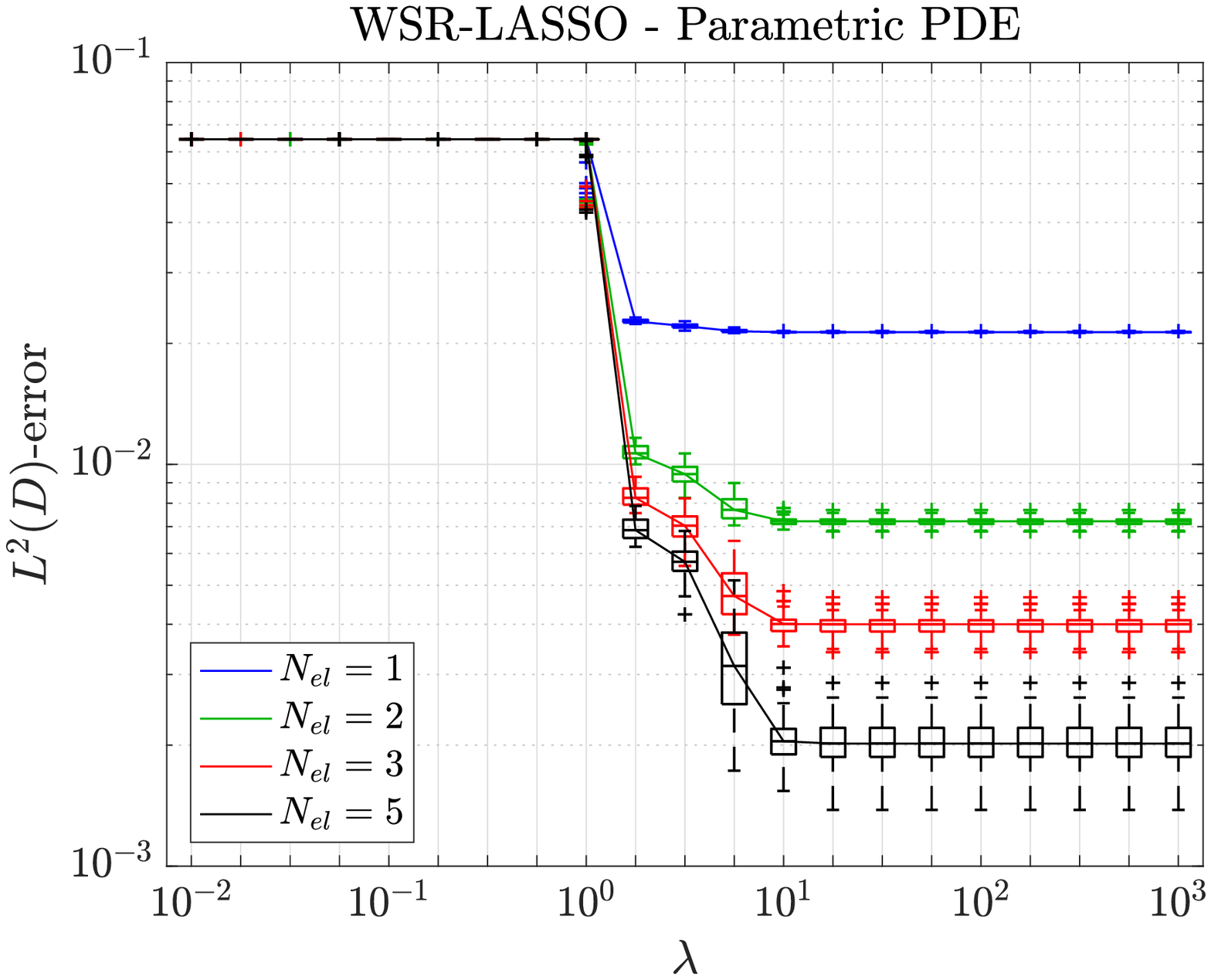}
\includegraphics[width = 7cm]{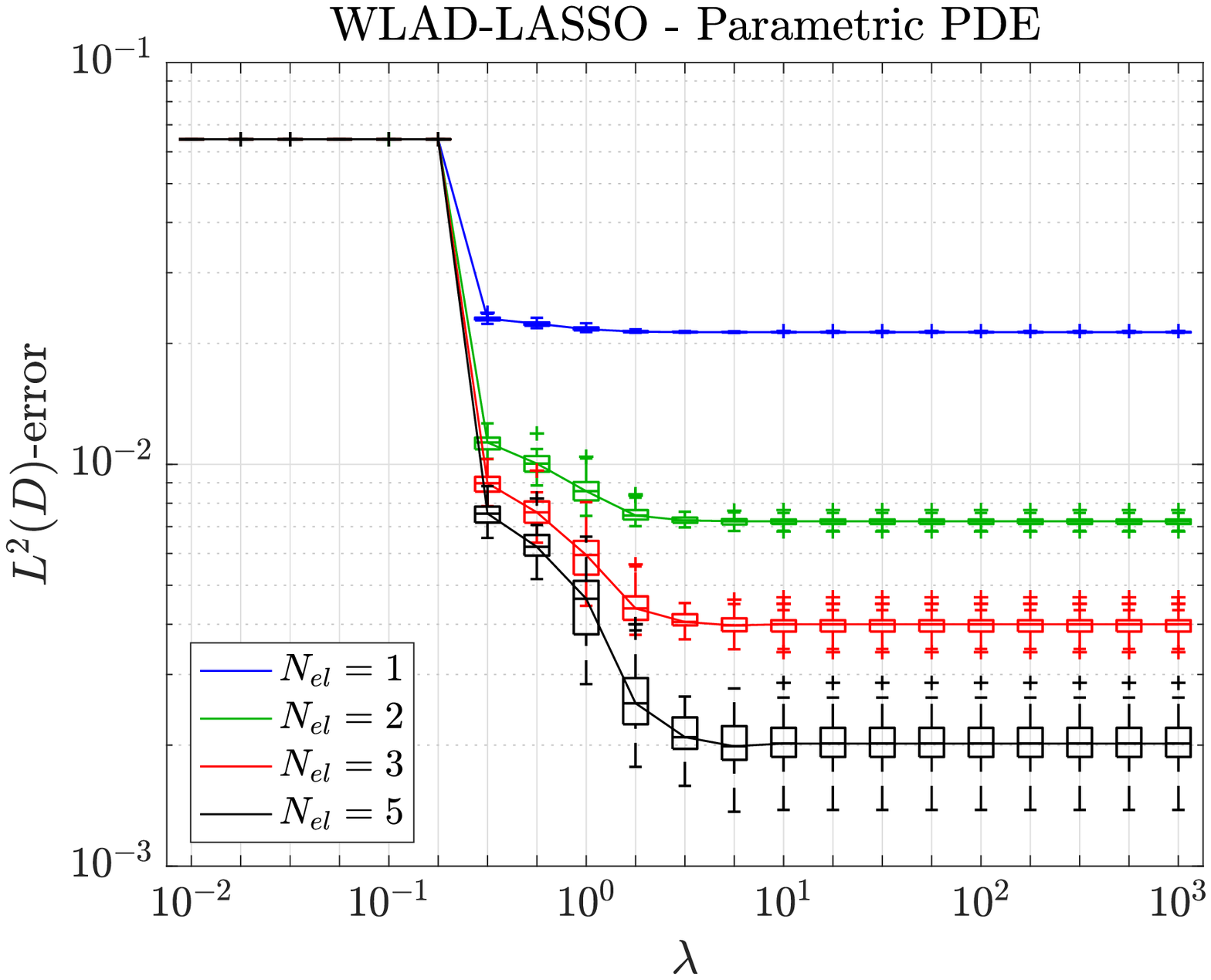}
\caption{\label{fig:UQ_param_vs_err} Box plot of the recovery error as a function of the tuning parameter for WQCBP (top left), WLASSO (top right), WSR-LASSO (bottom left), and WLAD-LASSO (bottom right) for the sparse approximation of the quantity of interest $f$ defined in \eqref{eq:function_UQ} with respect to tensorized Legendre polynomials.}
\end{figure}

As expected, the best accuracy achieved by each solver is a decreasing function of $N_{el}$. We can still see that the best choice of the parameter depends on the noise level for WQCBP and WLASSO, whereas it is independent of the noise for WSR-LASSO and WLAD-LASSO. Surprisingly, the strict global minima corresponding to the optimal choice of the parameter observed in Figs.~\ref{fig:param_vs_err_Leg} \& \ref{fig:param_vs_err_Che} do not appear here. In fact, all solvers exhibit a plateau. We may argue that the optimal choices \eqref{eq:optimaleta}, \eqref{eq:lambdaWLASSO}, \eqref{eq:lambdaWSRLASSO}, and \eqref{eq:lambdaWLADLASSO} hold under the form of inequality, and not of equality. Namely, $\eta \lesssim \|\bm{e}\|_2$ for WCQBP, $\lambda \gtrsim \sqrt{K(s)}/\|\bm{e}\|_2$ for WLASSO, $\lambda \gtrsim \sqrt{K(s)}$ for WSR-LASSO, and $\lambda \gtrsim 1/\sqrt{\log(n)}$ for WLAD-LASSO.\footnote{We have performed the same experiment for  quantities of interest different from \eqref{eq:function_UQ}, such as the integral of $u_{\bm{t}}$ over the regions $\Omega_i$, pointwise evaluations of $u_\bm{t}$, or the integral of $u_{\bm{t}}^2$ over $\Omega_i$ or $\Omega_F$. In all these cases, we do not observe the strict global minima in the parameter-vs-error plot. These experiment are not reported here for the sake of brevity.}  See \S \ref{ss:discussion} for some further discussion.

Next, we compare the performance of the four decoders using the ten combinations of decoders and tuning parameters described in \S\ref{sec:m_vs_err}. We  let $m = (20:20:100)$ and use $N_{el} = 5$. Fig.~\ref{fig:UQ_m_vs_err} shows the $L^2_\nu$ error as a function of $m$, averaged over 25 trials. 
\begin{figure}
\centering
\begin{tabular}{cc}
\raisebox{1cm}{\includegraphics[width = 5cm]{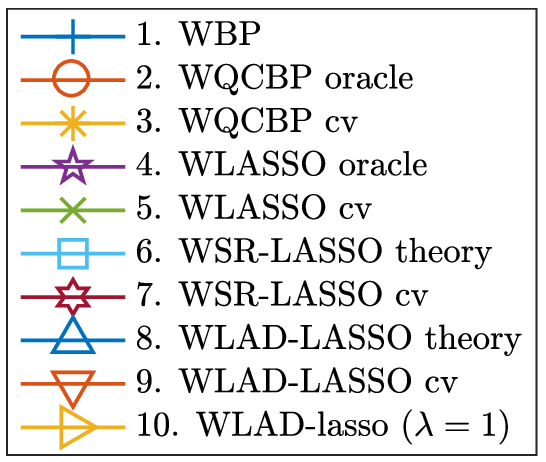}} &
\includegraphics[width = 7cm]{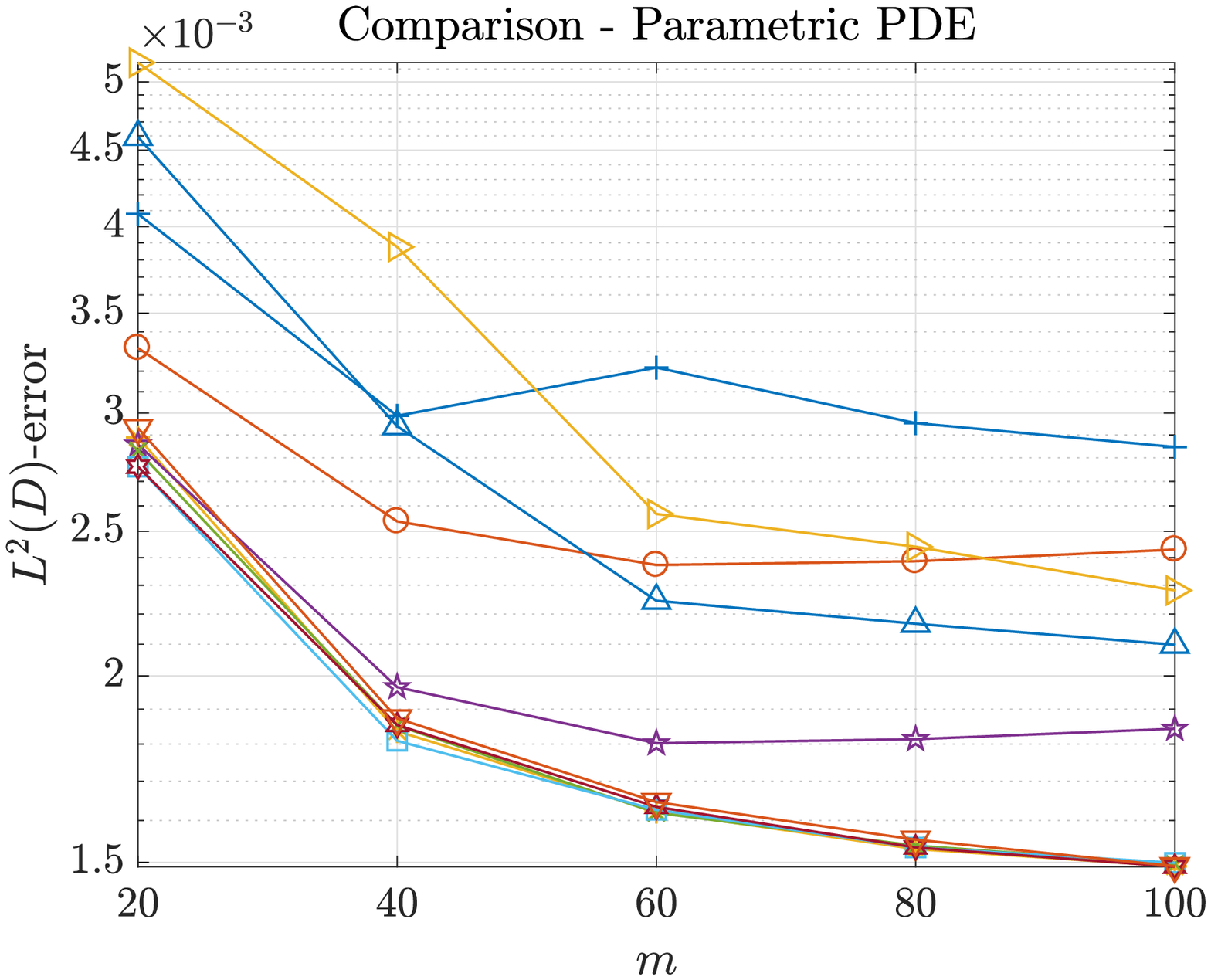} 
\end{tabular}
\caption{\label{fig:UQ_m_vs_err}Comparison among WQCBP, WLASSO, WSR-LASSO, and WLAD-LASSO for the approximation of the quantity of interest $f$ defined in \eqref{eq:function_UQ}.}
\end{figure}
Overall, the decoders have similar performances, in that the recovery error differs by less than one digit. Cross validation is able to achieve the best accuracy level in all cases. Remarkably, WSR-LASSO achieves the best accuracy using the choice \eqref{eq:lambdaWSRLASSO} suggested by the theory. 

\subsubsection{Parametric ODEs: Damped harmonic oscillator}
\label{sec:ODEs}

We study an application to parametric ODEs with random inputs, considering the  equation of a damped harmonic oscillator subject to external forcing
\begin{equation}
\label{eq:DHO}
\begin{cases}
u''_{\bm{t}}(x) + \gamma u'_{\bm{t}}(x) + k u_{\bm{t}}(x) = g \cos(\omega x), \quad \forall x > 0,\\
u_{\bm{t}}(0) = u_0, \quad u_{\bm{t}}'(0) = v_0,
\end{cases}
\end{equation}
where the parameters $\gamma, k , g, \omega, u_0, v_0$ are defined as a function of a random vector $\bm{t}$  uniformly distributed in $D = [-1,1]^6$ as follows:
$$
\begin{array}{lll}
\gamma = 0.1 + 0.02 t_1 \in [0.08, 0.12], 
& k = 0.035 + 0.05 t_2 \in [0.03,0.04], 
& g = 0.1 + 0.02 t_3 \in [0.08,0.12],\\
 \omega = 1 + 0.2 t_4 \in [0.8,1.2], 
 & u_0 = 0.5 + 0.05 t_5 \in [0.45,0.55], 
 & v_0 = 0.05 t_6 \in [-0.05,0.05].
\end{array}
$$
These restrictions make the oscillator always underdamped. We are interested in approximating the following quantity of interest:
\begin{equation}
\label{eq:DHO_QoI}
f(\bm{t}) = u_{\bm{t}}(20).
\end{equation}
An analogous experiment is considered in \cite{Adcock2017compressed}. 

We generate the pointwise samples by solving \eqref{eq:DHO} in \textsc{Matlab$^\text{\textregistered}$} using the command \texttt{ode23} and setting different values for the tolerance  parameter $\texttt{AbsTol} \in\{10^{-1}, 10^{-3}, 10^{-5}\}$ within \texttt{odeset}. We employ the Legendre polynomials as sparsity basis and set $s = 20$ (corresponding to $n = 795$) and $m = 100$. In Fig.~\ref{fig:DHO_par_err} (left), we show the boxplot of the $L^2_\nu$ error as a function of the parameter $\eta$ for WQCBP and WSR-LASSO over 25 runs, where the $L^2_\nu$ error is computed with respect to a high-fidelity reference solution $\bm{x}_{ref}$, approximated via least squares from 10000 random samples with $\texttt{AbsTol}=10^{-14}$.

\begin{figure}[t]
\centering
\includegraphics[width = 7cm]{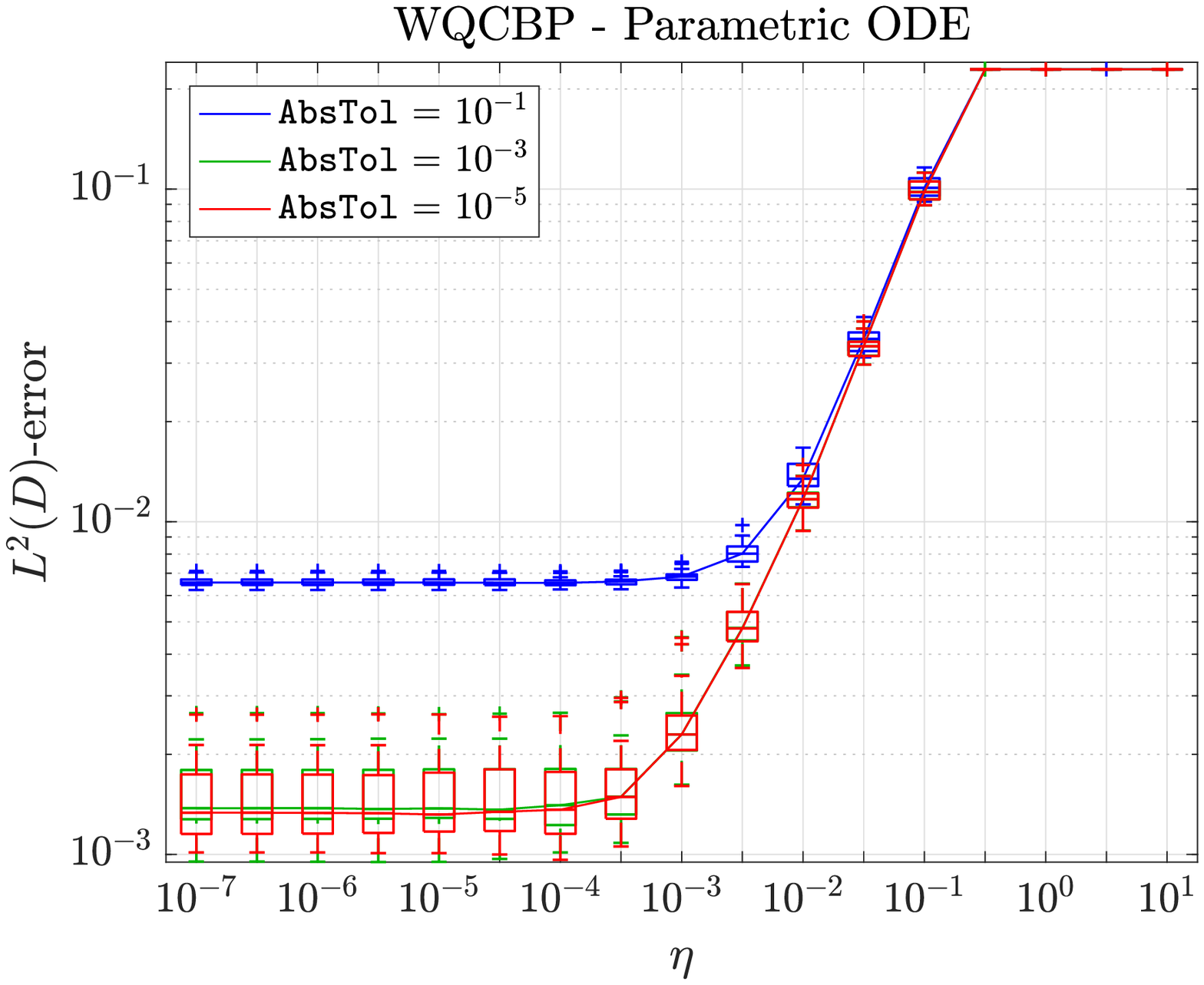}
\includegraphics[width = 7cm]{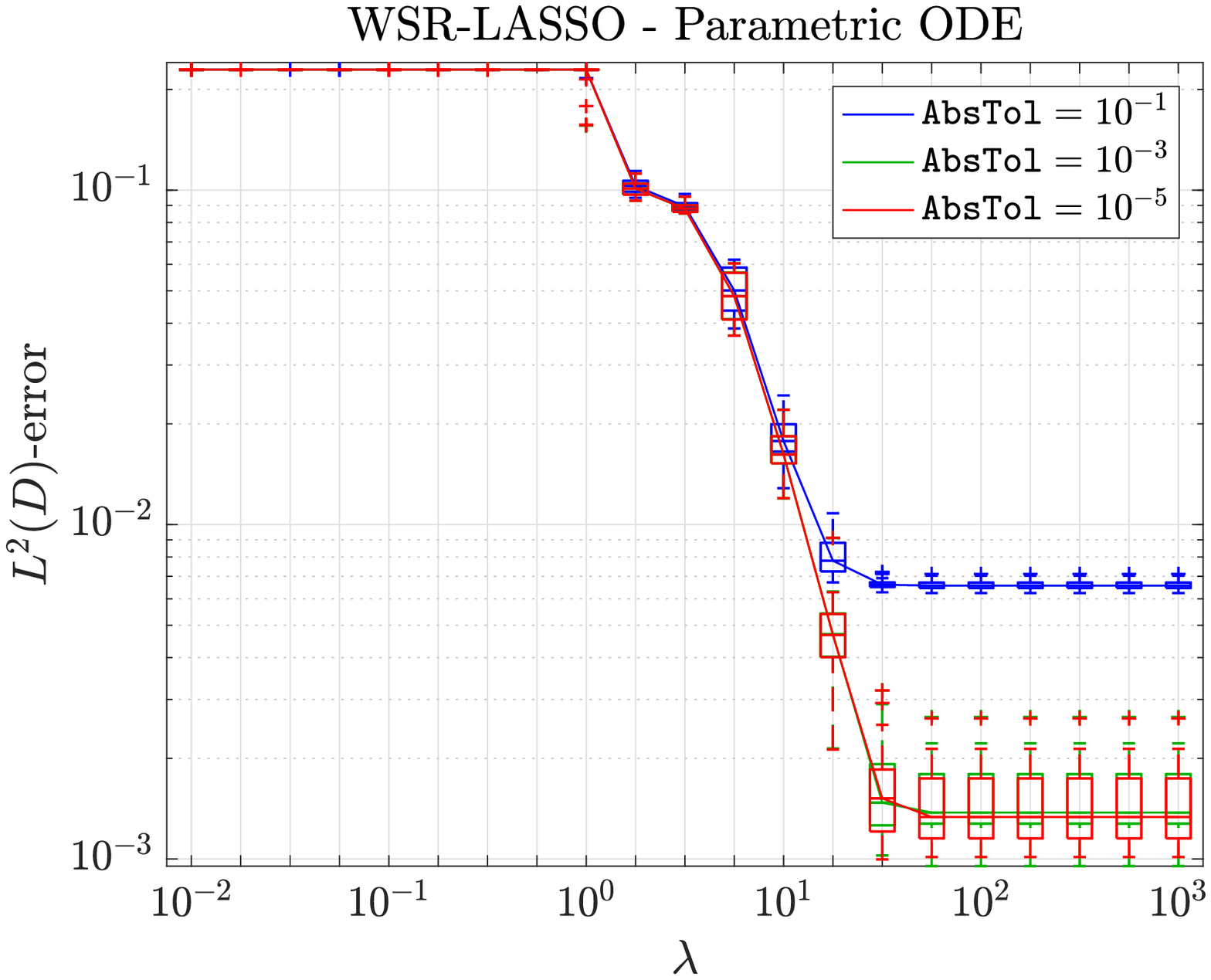}
\caption{\label{fig:DHO_par_err} Boxplot of the $L^2_\nu$ error as a function of the tuning parameter for WQCBP (left) and WSR-LASSO (right) when approximating the quantity of interest \eqref{eq:DHO_QoI} in the damped harmonic oscillator case. }
\end{figure}
Analogously to the case of the parametric PDE (Fig.~\ref{fig:UQ_param_vs_err} top left), we see that there are no strict global minima.  The optimal parameter choices appear to be $\eta \lesssim \| \bm{e} \|_2$ and $\lambda \gtrsim \sqrt{K(s)}$ as in the previous experiment.\footnote{The same phenomenon is observed for WLASSO and WLAD-LASSO, but the plots are not shown here for the sake of brevity.}

\subsubsection{Discussion}\label{ss:discussion}

Both the parametric PDE and ODE problems considered fail to exhibit the strict global minima in the error versus parameter plots as seen in the synthetic examples (Figs.~\ref{fig:param_vs_err_Leg} \& \ref{fig:param_vs_err_Che}).  We conjecture that this stems from the fact that errors produced by the numerical solvers are structured, as opposed to random.  Note that our theoretical recovery guarantees are adversarial, i.e.\ uniform with respect to the $2$-norm $\| \bm{e} \|_{2}$ of the error $\bm{e} \in \mathbb{C}^m$, and therefore cannot explain this behaviour.  Evidently, the structured errors produced by the solvers are quite far from this worst-case scenario.  This may imply that careful parameter tuning is not required in practice, although such a conclusion requires more broad examination.  In particular, understanding this phenomenon from the theoretical viewpoint remains an open problem.

\section{Theoretical analysis}
%%%%%%%%%%%%%%%%%%%%%%%%%%%%%%%%%%%%%%%%
%%%%%%%%%%%%%%%%%%%%%%%%%%%%%%%%%%%%%%%%
\label{sec:theory}

This section contains the proofs of the results of \S\ref{sec:rec_err}. 
%Namely, we prove uniform robust recovery guarantees for WQCBP, WLASSO, WSR-LASSO, and WLAD-LASSO under unknown error.
It is organized in two main parts. First, we prove recovery error estimates for  WQCBP, WLASSO, and WSR-LASSO in \S\ref{sec:l2pen}. Then, in \S\ref{sec:l1pen} we deal with WLAD-LASSO. This section structure depends on the fact that the proof strategy adopted for the WLAD-LASSO contains some additional technicalities  with respect to the other three decoders.

\begin{rmrk}[Proof strategy]
\label{rem:proof_startegy}
In order to prove robust recovery error estimates, we will follow a conceptual roadmap commonly employed in compressed sensing \cite{Foucart2013}. This is composed of the following elements: (LB) lower bound on the sample complexity $m$; (RIP) restricted isometry property for $\mat{A}$; (RNSP) robust null space property for $\mat{A}$; (DB) distance bound; (REE) recovery error estimate. The chain of implications showed is then: 
$$
(\text{LB}) 
\Rightarrow (\text{RIP})^{\text{whp}} 
\Rightarrow (\text{RNSP})^{\text{whp}} 
\Rightarrow (\text{DB})^{\text{whp}} 
\Rightarrow (\text{REE})^{\text{whp}},
$$ 
where the notation $(x)^{\text{whp}}$ indicates that the property $(x)$ holds with high probability.
\end{rmrk}

\subsection{WQCBP, WLASSO, and WSR-LASSO}
\label{sec:l2pen}

Before showing recovery error estimates for WQCBP, WLASSO, and WSR-LASSO, we need to introduce suitable versions of the restricted isometry property and of the robust null space property, adapted to the lower set setting (\S\ref{sec:l2-pen:prelim}). Then, we prove the recovery error estimates for WQCBP, WLASSO, and WSR-LASSO in \S\ref{sec:WQCBP}, \ref{sec:WLASSO}, and \ref{sec:WSRLASSO}, respectively.

\subsubsection{Preliminaries}
\label{sec:l2-pen:prelim}

We first recall the lower restricted isometry property and of lower robust null space property \cite{Chkifa2017}:

\begin{defn}[Lower restricted isometry property] 
\label{def:lowerRIP}
A matrix $\mat{A} \in \mathbb{C}^{m \times n}$ is said to have the \emph{lower restricted isometry property} of order $s$ if there exists a constant $0 < \delta < 1$ such that 
$$
(1-\delta) \|\bm{z}\|_2^2 
\leq \|\mat{A}\bm{z}\|_2^2 
\leq (1+\delta) \|\bm{z}\|_2^2, \quad 
\forall \bm{z} \in \mathbb{C}^n, \; |\supp(\bm{z})|_{\bm{u}} \leq K(s),
$$
where $\supp(\bm{z}) := \{i \in [n]: z_i \neq 0\}$ and $|\supp(\bm{z})|_{\bm{u}}$ is its weighted cardinality defined as in \eqref{eq:def|S|_u}. The smallest constant such that this property holds is called the $s^{\text{th}}$ lower restricted isometry constant of $\mat{A}$ and it is denoted as $\delta_{s,L}$. 
\end{defn}

\begin{defn}[Lower  robust null space  property]
\label{def:lowerNSP}
Given $0 < \rho < 1$ and $\tau > 0$, a matrix $\mat{A} \in \mathbb{C}^{m \times n}$ is said to have the \emph{lower  robust null space property} of order $s$ if
$$
\|\bm{z}_{S}\|_2 \leq \frac{\rho}{\sqrt{K(s)}} \|\bm{z}_{S^c}\|_{1,\bm{u}} + \tau \|\mat{A} \bm{z}\|_2, \quad \forall \bm{z} \in \mathbb{C}^n,
$$
for any $S \subseteq \Lambda$ such that $|S|_{\bm{u}} \leq K(s)$, where $K(s)$ is defined as in \eqref{eq:defK(s)}.
\end{defn}
The lower restricted isometry and the lower robust null space properties of order $s$ are equivalent to the weighted restricted isometry and weighted robust null space properties of order $K(s)$, respectively  (see \cite[Definition 1.3]{Rauhut2016} and \cite[Definition 4.1]{Rauhut2016}).

The following result, corresponding to \cite[Proposition 4.4]{Chkifa2017}, states that the lower restricted isometry property implies the lower robust null space property up to a suitable restriction on the lower restricted isometry constant.

\begin{lem}[Lower restricted isometry property $\Rightarrow$ lower robust null space property] 
\label{lem:LRIP->LRNSP}
Let $s \geq 2$ and $\mat{A} \in \mathbb{C}^{m\times n}$ satisfy the lower restricted isometry property of order $\alpha s$ with constant
$$
\delta = \delta_{\alpha s, L} < \frac{1}{5},
$$
where $\alpha = 2$ if the intrinsic weights $\bm{u}$ arise from the tensor Legendre basis and $\alpha = 3$ if the weights arise from the tensor Chebyshev basis. Then  $\mat{A}$ has the lower robust null space property of order $s$ with
$$
\rho = \frac{4 \delta}{1-\delta} \quad \text{and} \quad \tau = \frac{\sqrt{1+\delta}}{1-\delta}.
$$
\end{lem}

In \cite[Theorem 2.2]{Chkifa2017} it is proved that a sufficient condition for the lower restricted isometry property of order $s$ to hold with high probability is that the sample complexity $m$ must scale linearly with the intrinsic sparsity $K(s)$ up to a polylogarithmic factor. This result is stated below.

\begin{thm}[Sample complexity $\Rightarrow$ lower restricted isometry property]
\label{thm:sampleRIP}
Fix $0<\varepsilon<1$, $0 < \delta < 1$, let $\{\phi_{\bm{i}}\}_{\bm{i}\in\mathbb{N}_0^d}$ be as in \eqref{eq:defbasis} and $\bm{u}$ be the intrinsic weights defined in \eqref{eq:defintrinsic} and suppose that
\begin{equation}
m \gtrsim K(s)  \cdot L(s,n,\delta,\varepsilon),
\end{equation}
 where $K(s)$ is as in \eqref{eq:defK(s)} and where $L = L(s,n,\delta,\varepsilon)$ is a polylogarithmic factor defined as
\begin{equation}
\label{eq:defpolylog}
L
= \frac{1}{\delta^2}\ln\left(\frac{K(s)}{\delta^2}\right)
\max\left\{\frac{1}{\delta^4} \ln\left(\frac{K(s)}{\delta^2}\ln\left(\frac{K(s)}{\delta^2}\right)\right) \ln(n),
\frac{1}{\delta}\ln\left(\frac{1}{\delta\varepsilon} \ln\left(\frac{K(s)}{\delta^2}\right)\right)\right\}.
\end{equation} 
Then with probability at least $1 - \varepsilon$ the design matrix $\mat{A}$ defined in \eqref{eq:defA} satisfies the lower restricted isometry property of order $s$ with constant $\delta_{s,L} \leq \delta$.
\end{thm}
%\RED{[I've removed the condition $\delta < 1/13$ from everywhere. I agree: because we are assuming $m \gtrsim \cdots$, we can avoid to worry about the 13 here... thanks for spotting this.]}

A fundamental feature of the lower (or weighted) robust null space property is that it implies a suitable $\ell^1_{\bm{u}}$ distance bound. This is stated in the following theorem, an immediate consequence of \cite[Theorem 4.2]{Rauhut2016} (notice that $\sigma_{K(s)}(\bm{z})_{1,\bm{u}} \leq \sigma_{s,L}(\bm{z})_{1,\bm{u}}$, for any $\bm{z} \in \mathbb{C}^n$). This can also be viewed as a corollary of the  Theorem~\ref{thm:2-lev_dist_bound}, which deals with the framework of 2-level sparsity. 
\begin{thm}[Lower robust null space property $\Rightarrow$ $\ell^1_{\bm{u}}$ distance bound]
\label{thm:1-lev_dist_bound}
If the matrix $\mat{A} \in \mathbb{C}^{m \times n}$ has the lower null space property of order $s$, then for every $\bm{z}, \hat{\bm{z}} \in \mathbb{C}^{n}$ the following holds:
$$
\|\bm{z}-\hat{\bm{z}}\|_{1,\bm{u}} 
\leq \frac{1+\rho}{1-\rho} 
(2 \sigma_{s,L}(\bm{z})_{1,\bm{u}} + \|\hat{\bm{z}}\|_{1,\bm{u}} - \|\bm{z}\|_{1,\bm{u}})
+ \frac{2\tau \sqrt{K(s)}}{1-\rho} \|\mat{A}(\hat{\bm{z}}-\bm{z})\|_2.
$$
\end{thm}

The next step is to convert an $\ell_\bm{u}^1$ error estimate  into an $L^2_\nu$ error estimate. In order to do this, we need to recall two technical lemmas. Namely, an upper bound on the $\ell^\infty$ norm of the intrinsic weights  and a weighted version of the Stechkin inequality. Their proofs can be found in \cite[Lemma 4.1]{Chkifa2017} and \cite[Theorem 3.2]{Rauhut2016}, respectively.
\begin{lem}[$\ell^\infty$ bound on the intrinsic weights]
\label{lem:norm_u_inf_bound}
For any $s \geq 2 $ and $d \geq 1$, let  $\Lambda=\Lambda^{\textnormal{HC}}_{d,s}$. Then
$$
\|\bm{u}_{\Lambda}\|_\infty^2 \leq \frac34 K(s),
$$
where $(u_{\bm{i}})_{\bm{i} \in \mathbb{N}_0^d}$ are the weights \eqref{eq:defintrinsic} associated with the tensor Chebyshev or Legendre polynomials.
\end{lem}

\begin{lem}[Weighted Stechkin inequality]
\label{lem:stechkin}
For any $\bm{z} \in \ell^1_{\bm{u}}(\mathbb{N}_0^d)$ and $K \in \mathbb{N}$, with $K > \|\bm{u}\|_\infty^2$, the following holds
$$
\inf\{\|\bm{z}-\bm{z}'\|_2 : |\supp(\bm{z}')|_{\bm{u}} \leq K\} \leq \frac{\|\bm{z}\|_{1,\bm{u}}}{\sqrt{K - \|\bm{u}\|_\infty^2}}.
$$
\end{lem}

Thanks to Lemma~\ref{lem:norm_u_inf_bound}, Lemma~\ref{lem:stechkin}, we are able to give a way to convert an $\ell_\bm{u}^1$ error estimate (on the coefficients) into an $L^2_\nu$ error estimate (on the function) under the lower robust null space property. 
\begin{lem}[Lower robust null space property $\Rightarrow$ $\ell_\bm{u}^1\to L^2_\nu$ estimate conversion]
\label{lem:l1tol2}
Let $\hat{\bm{x}}_\Lambda$ be the output of any decoder and assume that $\mat{A}$ has the lower robust null space property of order $s$. Then,
$$
\|f-\tilde{f}\|_{L^2_\nu}
\leq \|f-f_\Lambda\|_{L^2_\nu} 
+ \frac{2+\rho}{\sqrt{K(s)}}\|\bm{x}_\Lambda - \hat{\bm{x}}_\Lambda\|_{1,\bm{u}}
+ \tau \|\mat{A}(\bm{x}_\Lambda - \hat{\bm{x}}_\Lambda)\|_2.
$$
where $\tilde{f}$ is as in \eqref{eq:ftilde}.
\end{lem}
\begin{proof}
First, observe that 
$$
\|f-\tilde{f}\|_{L^2_\nu} 
= \|\bm{x}-\hat{\bm{x}}_\Lambda\|_2 
\leq \|\bm{x}_{\Lambda^c}\|_2 + \|\bm{x}_\Lambda - \hat{\bm{x}}_\Lambda\|_2
= \|f-f_\Lambda\|_{L^2_\nu} + \|\bm{x}_\Lambda - \hat{\bm{x}}_\Lambda\|_2.
$$
Now, choose a (not necessarily lower) set $S \subseteq \Lambda$  such that 
$$
\|(\bm{x}_\Lambda - \hat{\bm{x}}_\Lambda)_{S^c}\|_2 
= \inf_{\bm{z}: |\supp(\bm{z})|_{\bm{u}} \leq K(s)} \|\bm{z} - (\bm{x}_\Lambda - \hat{\bm{x}}_\Lambda)\|_2.
$$
(Notice that the infimum is actually a minimum because it is defined over a finite union of linear subspaces). Employing the weighted Stechkin inequality (Lemma~\ref{lem:stechkin}) and the upper bound  $\|\bm{u}\|_\infty^2 \leq \frac34 K(s)$ (Lemma~\ref{lem:norm_u_inf_bound}), we obtain
$$
\|(\bm{x}_\Lambda - \hat{\bm{x}}_\Lambda)_{S^c}\|_2
\leq \frac{\|\bm{x}_\Lambda - \hat{\bm{x}}_\Lambda\|_{1,\bm{u}}}{\sqrt{K(s) - \|\bm{u}_\Lambda\|_\infty^2}}
\leq \frac{2 \|\bm{x}_\Lambda - \hat{\bm{x}}_\Lambda\|_{1,\bm{u}}}{\sqrt{K(s)}}.
$$
Moreover, the lower robust null space property implies 
$$
\|(\bm{x}_\Lambda - \hat{\bm{x}}_\Lambda)_{S}\|_2
\leq 
\frac{\rho}{\sqrt{K(s)}} \|(\bm{x}_\Lambda - \hat{\bm{x}}_\Lambda)_{S^c}\|_{1,\bm{u}}
+ \tau \|\mat{A}(\bm{x}_\Lambda - \hat{\bm{x}}_\Lambda)\|_2.
$$
Combining the above inequalities with the triangle inequality $\|\bm{x}_\Lambda - \hat{\bm{x}}_\Lambda\|_2 \leq \|(\bm{x}_\Lambda - \hat{\bm{x}}_\Lambda)_{S}\|_2 + \|(\bm{x}_\Lambda - \hat{\bm{x}}_\Lambda)_{S^c}\|_2$ concludes the proof.
\end{proof}

\subsubsection{Robustness of WQCBP}
\label{sec:WQCBP}

In this section, we prove that the WQCBP decoder \eqref{eq:WQCBP} is  robust under unknown error. In particular, we will assume to be in the error-blind scenario, where upper bounds of the form $\|\bm{e}\|_2 \leq \eta$ are not available. The results proved here generalize the robust recovery error guarantees given in \cite{Brugiapaglia2017} for the unweighted case. The recovery error estimate depends on the best $s$-term approximation error in lower sets, on the parameter $\eta$, on the unknown error norm $\|\bm{e}\|_2$, and on a tail term $\mathcal{T}$ whose behavior is studied in Theorem~\ref{thm:WQCBPtailbound}. We follow the proof strategy explained in Remark~\ref{rem:proof_startegy}.

\begin{thm}[Robust recovery for WQCBP]
\label{thm:WQCBPerror-blind}
Let $0 < \varepsilon < 1$, $0 < \delta < 1$, $2 \leq s \leq 2^{d+1}$, $\Lambda = \Lambda^{\text{HC}}_{d,s}$ be the hyperbolic cross index set defined in \eqref{eq:defHC}, and $\{\phi_{\bm{i}}\}_{\bm{i} \in \mathbb{N}_0^d}$ be the tensor Legendre or Chebyshev polynomial basis. Draw $\bm{t}_1,\ldots,\bm{t}_m$ independently according to the corresponding measure $\nu$, with 
\begin{equation}
\label{thm:WQCBPerror-blind:eq:unifsamplecomplex}
m \gtrsim s^\gamma L(s,n,\delta,\varepsilon),
\end{equation}
where $\gamma$ is defined as in \eqref{eq:defgamma} and $L(s,n,\delta,\varepsilon)$ is the polylogatrithmic factor defined in \eqref{eq:defpolylog}. Then, the following holds with probability at least $1-\varepsilon$. For any $\eta \geq 0$ and $f \in L^2_\nu(D) \cap L^\infty(D)$ expanded as in \eqref{eq:fexpansion}, the approximation $\tilde{f}$ defined as in \eqref{eq:ftilde} computed using the WQCBP decoder \eqref{eq:WQCBP} satisfies  
\begin{align}
\label{thm:WQCBPerror-blind:eq:Linferrbound}
\|f - \tilde{f}\|_{L^\infty} 
&\leq
C_1 \sigma_{s,L}(\bm{x})_{1,\bm{u}} + s^{\gamma/2} [C_2(\eta + \|\bm{e}\|_2) + C_3 \mathcal{T}],\\
\label{thm:WQCBPerror-blind:eq:L2errbound}
\|f - \tilde{f}\|_{L^2_\nu} 
&\leq
 C_4\frac{\sigma_{s,L}(\bm{x})_{1,\bm{u}}}{s^{\gamma/2}} 
+ C_5 (\eta + \|\bm{e}\|_2) 
+ C_6 \mathcal{T}
+\|f-f_\Lambda\|_{L^2_\nu},
\end{align}
where $f_\Lambda$ is defined as in \eqref{eq:fLambda}, 
\begin{equation}
\label{eq:defT}
\mathcal{T} 
= \mathcal{T}(\mat{A},\Lambda,\bm{e},\bm{u},\eta) 
:= \min\left\{\frac{\|\bm{z}\|_{1,\bm{u}}}{s^{\gamma/2}} : \bm{z} \in \mathbb{C}^n, \; \|\mat{A} \bm{z} - \bm{e}\|_2 \leq \eta\right\}.
\end{equation}
and with constants
$$
C_1 = \frac{3 +\rho}{1-\rho}, \quad 
C_2 = \frac{2\tau}{1-\rho}, \quad 
C_3 = \frac{1+\rho}{1-\rho},
\quad
C_4 = \frac{2(1+\rho)(2+\rho)}{1-\rho}, \quad
C_5 = \frac{\tau(9+3\rho)}{1-\rho}, \quad
C_6 = C_4, \quad
$$
where $\rho = 4\delta / (1-\delta)$ and $\tau = \sqrt{1-\delta}/(1-\delta)$.
\end{thm}
\begin{proof}
Thanks to Theorem~\ref{thm:sampleRIP} and Lemma~\ref{lem:LRIP->LRNSP}, the sample complexity bound \eqref{thm:WQCBPerror-blind:eq:unifsamplecomplex} guarantees that $\mat{A}$ has the lower robust null space property with probability at least $1-\varepsilon$. As a consequence, Theorem~\ref{thm:1-lev_dist_bound} holds with probability at least $1-\varepsilon$. 

Now, we observe that, on the one hand, 
\begin{equation}
\label{thm:WQCBPerror-blind:eq:datafidelity}
\|\mat{A}(\bm{x}_\Lambda-\hat{\bm{x}}_\Lambda)\|_2
\leq  \|\mat{A}\bm{x}_\Lambda-\bm{y}\|_2 + \|\mat{A}\hat{\bm{x}}_\Lambda - \bm{y}\|_2
\leq \|\bm{e}\|_2 + \eta
\end{equation}
and, on the other hand, 
\begin{align*}
\|\hat{\bm{x}}_{\Lambda}\|_{1,\bm{u}} - \|\bm{x}_{\Lambda}\|_{1,\bm{u}}
& = \min \{\|\bm{z}\|_{1,\bm{u}} : \bm{z} \in \mathbb{C}^n, \|\mat{A}\bm{z}-\bm{y}\|_2 \leq \eta\} - \|\bm{x}_{\Lambda}\|_{1,\bm{u}}\\
& \leq \min \{\|\bm{z} - \bm{x}_\Lambda\|_{1,\bm{u}} : \bm{z} \in \mathbb{C}^n,\|\mat{A}\bm{z}-\bm{y}\|_2 \leq \eta\}\\
& = \min \{\|\bm{z}\|_{1,\bm{u}} : \bm{z} \in \mathbb{C}^n,\|\mat{A}\bm{z}-\bm{e}\|_2 \leq \eta\}\\
& = s^{\gamma / 2} \mathcal{T}(\mat{A},\Lambda,\bm{e},\bm{u},\eta).
\end{align*}
Noting that $\sigma_{s,L}(\bm{x}_\Lambda)_{1,\bm{u}} = \sigma_{s,L}(\bm{x})_{1,\bm{u}}$ (since $\Lambda$ is the union of all  lower sets of cardinality $s$), that $K(s) \leq s^{\gamma}$ (due to Lemma~\ref{lem:K(s)bounds}), and employing Theorem~\ref{thm:1-lev_dist_bound} with $\bm{z} = \bm{x}_\Lambda$ and $\hat{\bm{z}} = \hat{\bm{x}}_\Lambda$, we obtain
\begin{equation}
\label{thm:WQCBPerror-blind:eq:||x_Lambda-x_Lambdahat||Linf}
\|\bm{x}_\Lambda - \hat{\bm{x}}_\Lambda\|_{1,\bm{u}}
 \leq 
\frac{2(1 + \rho)}{1-\rho} \sigma_{s,L}(\bm{x})_{1,\bm{u}} 
+ s^{\gamma/2} \left[\frac{2\tau}{1-\rho} (\|\bm{e}\|_2 + \eta) 
+ \frac{1+\rho}{1-\rho} \mathcal{T} \right].
\end{equation}
Using $\|\bm{x}-\bm{x}_\Lambda\|_{1,\bm{u}} \leq \sigma_{s,L}(\bm{x})_{1,\bm{u}}$ and plugging \eqref{thm:WQCBPerror-blind:eq:||x_Lambda-x_Lambdahat||Linf} into
$$
\|f-\tilde{f}\|_{L^{\infty}}
\leq \|\bm{x}-\hat{\bm{x}}_\Lambda\|_{1,\bm{u}}
\leq \|\bm{x} - \bm{x}_\Lambda\|_{1,\bm{u}}  + \|\bm{x}_\Lambda - \hat{\bm{x}}_\Lambda\|_{1,\bm{u}},
$$
gives  \eqref{thm:WQCBPerror-blind:eq:Linferrbound}.

Finally, we employ the $\ell_\bm{u}^1\to L^2_\nu$ estimate conversion result (Lemma~\ref{lem:l1tol2}) combined with \eqref{thm:WQCBPerror-blind:eq:datafidelity}, \eqref{thm:WQCBPerror-blind:eq:||x_Lambda-x_Lambdahat||Linf}, and with the fact that $K(s) \geq \frac14 s^{\gamma}$ (Lemma~\ref{lem:K(s)bounds}), to obtain \eqref{thm:WQCBPerror-blind:eq:L2errbound}.
\end{proof}

In the following result, we give  an upper bound to the tail term $\mathcal{T}$ defined in \eqref{eq:defT} in terms of the $m^{th}$ singular value (in decreasing order) of the matrix $\sqrt{\frac{m}{n}}\mat{A}^*$, denoted as $\sigma_{m}(\sqrt{\frac{m}{n}}\mat{A}^*)$.
%\RED{[$\sigma_{\min}$ has been replaced with $\sigma_m$ throughout the paper.]}
\begin{thm}[Tail term bound]
\label{thm:WQCBPtailbound}
Consider the setup of Theorem~\ref{thm:WQCBPerror-blind}
with
\begin{equation}
\label{thm:WQCBPtailbound:eq:samplecomplexity}
m \asymp s^{\gamma} L(s,n,\delta,\varepsilon),
\end{equation} 
and let $\mathcal{T}=\mathcal{T}(\mat{A},\Lambda,\bm{e},\bm{u},\eta)$ be as in \eqref{eq:defT}. Then, 
\begin{equation}
\label{eq:Tbound}
\mathcal{T}
\lesssim
\frac{s^{\alpha/2} \sqrt{L}}{\sigma_m(\sqrt{\frac{m}{n}}\mat{A}^*)}
\max\{\|\bm{e}\|_2 - \eta, 0\},
\end{equation}
where $L$ is as in \eqref{eq:defpolylog} and $\alpha = 1, 2$ in the Chebyshev or Legendre case, respectively.
\end{thm}
\begin{proof}
Notice that if the rank of $\mat{A}$ is not full, then $\sigma_m(\sqrt{\frac{m}{n}}\mat{A}^*)=0$ and \eqref{eq:Tbound} is trivially satisfied. Hence, we assume that $\mat{A}$ has full rank. Also if $\eta \geq \|\bm{e}\|_2$  the result holds trivially.  Then, we also suppose that $\eta < \|\bm{e}\|_2$.  

Since $\|\bm{e}\|_2 \neq 0$, we can define the ansatz $\bm{z} := (1 - \eta / \|\bm{e}\|_2 ) \mat{A}^{\dag} \bm{e}$, where $\mat{A}^{\dag} = \mat{A}^*(\mat{A}\mat{A}^*)^{-1}$ is the pseudoinverse.  Then $\bm{z}$ satisfies $\|\mat{A} \bm{z} - \bm{e}\|_{2} = \eta$, and hence, recalling the definition \eqref{eq:defT} of $\mathcal{T}$, we have
$$
s^{\gamma/2} \mathcal{T} 
\leq \|\bm{z}\|_{1,\bm{u}} 
\leq \sqrt{| \Lambda |_{\bm{u}}} \|\bm{z}\|_{2} 
\leq \frac{\sqrt{| \Lambda |_{\bm{u}}}}{\sigma_m(\mat{A}^*)} 
     \left ( \|\bm{e}\|_2- \eta  \right ).
$$
Equation \eqref{thm:WQCBPtailbound:eq:samplecomplexity} implies that $\sqrt{\frac{m}{s^{\gamma}}} \lesssim \sqrt{L}$, and therefore
\begin{equation}
\label{Tbound1}
\mathcal{T} 
\lesssim \sqrt{\frac{| \Lambda |_{\bm{u}}}{n} } \frac{\sqrt{L}}{\sigma_m  \left ( \sqrt{\frac{m}{n}} \mat{A}^* \right )} \left (\|\bm{e}\|_2- \eta  \right ).
\end{equation}
The last step is to estimate $| \Lambda |_{\bm{u}}$ using the explicit formulae \eqref{eq:intrinsicChebLeg} for the intrinsic weights.  For the Chebyshev case, we have 
$$
| \Lambda |_{\bm{u}} 
= \sum_{\bm{i} \in \Lambda} 2^{\|\bm{i}\|_0} 
\leq \sum_{\bm{i} \in \Lambda} \prod^{d}_{\ell=1} \left ( i_\ell + 1 \right ) 
\leq s \sum_{\bm{i} \in \Lambda} 1  
= s n
$$
where in the penultimate step we used the definition of the hyperbolic cross \eqref{eq:defHC}.  For the Legendre case, we have 
$$
| \Lambda |_{\bm{u}} 
= \sum_{\bm{i} \in \Lambda} \prod^{d}_{\ell=1} \left ( 2 i_\ell + 1 \right ) 
\leq \sum_{\bm{i} \in \Lambda} 2^{\|\bm{i}\|_0} \prod^{d}_{\ell=1} \left ( i_\ell + 1 \right ) \leq s^2 n.
$$
This completes the proof. 
\end{proof}

Theorems~\ref{thm:WQCBPerror-blind} \& \ref{thm:WQCBPtailbound} show that the robustness of WQCBP is implied by inequalities of the form $\sigma_m(\sqrt{\frac{m}{n}}\mat{A}^*) \gtrsim 1$. When $d = 1$, this can be achieved by resorting to the spectral theory of random matrices with heavy-tailed rows \cite{Brugiapaglia2017}. Showing this type of inequality when $d > 1$ is still an open problem. However, it is not difficult to show that (see \cite[Lemma 3]{Adcock2017})
$$
\lambda_{\min}\left(\Expe\left[\frac{m}{n}\mat{A}\mat{A}^*\right]\right) = 1-\frac{1}{n}.%\Expe\left[\sigma_m\left(\sqrt{\tfrac{m}{n}}\mat{A}^*\right)\right] = \sqrt{1-\frac1n}.
$$
By inspecting the proof of Theorem~\ref{thm:WQCBPtailbound}, we notice that
$$
\mathcal{T} 
\lesssim 
\mathcal{Q} \sqrt{L} \max\{\|\bm{e}\|_2 -\eta,0\},
$$
where 
\begin{equation}
\label{eq:defQuA}
\mathcal{Q} = \mathcal{Q}(\mat{A},\Lambda,\bm{u}) := \sqrt{\frac{|\Lambda|_{\bm{u}}}{n}} \frac{1}{\sigma_m(\sqrt{\frac{m}{n}}\mat{A}^*)}.
\end{equation}
This constant can be easily estimated numerically and turns out to have moderate size (see \cite{Adcock2017}). 

To conclude, we notice that Theorem~\ref{thm:WQCBPerror-blind} implies an analogous result in the error-aware scenario, i.e.\ when $\|\bm{e}\|_2 \leq \eta$.

\begin{thm}[Robust recovery of WQCBP in the error-aware setting]
\label{thm:WQCBPerror-aware}
Consider the same setup of Theorem~\ref{thm:WQCBPerror-blind} and let $\|\bm{e}\|_2 \leq \eta$. Then, with probability at least $1-\varepsilon$ the following inequalities hold:
\begin{align*}
\|f - \tilde{f}\|_{L^\infty} 
&\lesssim
\sigma_{s,L}(\bm{x})_{1,\bm{u}} + s^{\gamma/2} \eta,\\
\|f - \tilde{f}\|_{L^2_\nu} 
&\lesssim
\frac{\sigma_{s,L}(\bm{x})_{1,\bm{u}}}{s^{\gamma/2}} +  \eta + \|f-f_\Lambda\|_{L^2_\nu}.
\end{align*}

\end{thm}

%%%%%%%%%%%%%%%%%%%%%%%%%%%%%%%%%%%%%%%%%%%%%
\subsubsection{Robustness of WLASSO}
\label{sec:WLASSO}

In this section, we prove that the WLASSO decoder is robust under unknown error, when the tuning parameter is chosen proportionally to the ratio $\sqrt{K(s)} / \|\bm{e}\|_2$. 
\begin{thm}[WLASSO robust recovery]
\label{thm:WLASSOrecovery}
Let $0 < \varepsilon < 1$, $0 < \delta < 1$, $2 \leq s \leq 2^{d+1}$, $\Lambda = \Lambda^{\text{HC}}_{d,s}$ be the hyperbolic cross index set defined in \eqref{eq:defHC}, and $\{\phi_{\bm{i}}\}_{\bm{i} \in \mathbb{N}_0^d}$ be the tensor Legendre or Chebyshev polynomial basis. Draw $\bm{t}_1,\ldots,\bm{t}_m$ independently according to the corresponding measure $\nu$, with 
$$
m \gtrsim s^\gamma L(s,n,\delta,\varepsilon),
$$
where $\gamma$ is defined as in \eqref{eq:defgamma} and $L(s,n,\delta,\varepsilon)$ is the polylogatrithmic factor defined in \eqref{eq:defpolylog}. Then, the following holds with probability at least $1-\varepsilon$. For any $f \in L^2_\nu(D) \cap L^\infty(D)$ expanded as in \eqref{eq:fexpansion}, the approximation $\tilde{f}$ defined as in \eqref{eq:ftilde} computed using the WLASSO decoder \eqref{eq:WLASSO} with tuning parameter  
\begin{equation}
\label{thm:WLASSOrecovery:eq:lambda}
\lambda = \theta \frac{\sqrt{K(s)}}{\|\bm{e}\|_2}, \quad \text{with } \theta > 0,
\end{equation}
satisfies
\begin{align}
\label{thm:WLASSOrecovery:eq:Linfest}
\|f-\tilde{f}\|_{L^\infty}
& \leq C_1 \sigma_{s,L}(\bm{x})_{1,\bm{u}} + C_2 s^{\gamma/2} \|\bm{e}\|_2,\\
\label{thm:WLASSOrecovery:eq:L2est}
\|f-\tilde{f}\|_{L^2_\nu}
& \leq C_3 \frac{\sigma_{s,L}(\bm{x})_{1,\bm{u}}}{s^{\gamma/2}} + C_4 \|\bm{e}\|_2 + \|f-f_\Lambda\|_{L^2_\nu},
\end{align}
where $f_\Lambda$ is defined as in \eqref{eq:fLambda}, and the constants are
$$
C_1 = \frac{3+\rho}{1-\rho}, \quad 
C_2 = \frac{1}{1-\rho}\left[\frac{\tau^2}{(1+\rho)\theta} + (1+\rho)\theta + 2\tau\right],
$$
and
$$
C_3  = \frac{4(1+\rho)(2+\rho)}{1-\rho}\quad \quad
C_4  = \frac{1}{1-\rho}\left[\frac{(5+\rho)^2 \tau^2}{4(1+\rho)(2+\rho)\theta} + (1+\rho)(2+\rho)\theta +  (5+\rho)\tau
\right],
$$
where $\rho = 4\delta / (1-\delta)$, $\tau = \sqrt{1-\delta}/(1-\delta)$.
\end{thm}
\begin{proof}
First, notice that \eqref{eq:WLASSO} can be reformulated in augmented form as
$$
(\hat{\bm{x}}_\Lambda, \hat{\bm{e}})
:= \arg \min_{(\bm{z},\bm{d})\in\mathbb{C}^{n}\times \mathbb{C}^m}
\|\bm{z}\|_{1,\bm{u}} + \lambda \|\bm{d}\|_2^2  \text{ s.t. } \mat{A}\bm{z} + \bm{d} = \bm{y}.
$$
Moreover, Theorem~\ref{thm:1-lev_dist_bound} (that holds with probability at least $1-\varepsilon$) implies
\begin{equation}
\label{thm:WLASSOrecovery:eq:l1uest}
\|\bm{x}_\Lambda - \hat{\bm{x}}_\Lambda\|_{1,\bm{u}}
\leq \frac{2(1+\rho)}{1-\rho} \sigma_{s,L}(\bm{x})_{1,\bm{u}}
 + \frac{\xi}{1-\rho},
\end{equation}
where
\begin{equation}
\label{thm:WLASSOrecovery:eq:defxi}
\xi := (1+\rho) (\|\hat{\bm{x}}_\Lambda\|_{1,\bm{u}} - \|\bm{x}_\Lambda\|_{1,\bm{u}}) + 2\tau \sqrt{K(s)} \|\mat{A}(\hat{\bm{x}}_\Lambda - \bm{x}_\Lambda)\|_2.
\end{equation}
Since  $\|\mat{A}(\hat{\bm{x}}_\Lambda - \bm{x}_\Lambda)\|_2 \leq \|\hat{\bm{e}}\|_2 + \|\bm{e}\|_2$, we  estimate
\begin{align*}
\xi
& \leq (1+\rho) \|\hat{\bm{x}}_\Lambda\|_{1,\bm{u}} 
	+ 2\tau\sqrt{K(s)}\|\hat{\bm{e}}\|_2 
	- (1+\rho)\|\bm{x}_{\Lambda}\|_{1,\bm{u}} 
	+ 2\tau\sqrt{K(s)}\|\bm{e}\|_2.
\end{align*}
Now, we employ Young's inequality $ab \leq \omega a^2 + \frac{b^2}{4\omega}$, with $\omega = (1+\rho)\lambda$, $a = \|\hat{\bm{e}}\|_2$, and $b = 2\tau\sqrt{K(s)}$. This, combined with the optimality of $(\hat{\bm{x}}_\Lambda, \hat{\bm{e}})$ yields
\begin{align*}
\xi 
& \leq \frac{\tau^2 K(s)}{(1+\rho)\lambda} 
	+ (1+\rho) \lambda \|\bm{e}\|_2^2
	+2 \tau \sqrt{K(s)}\|\bm{e}\|_2. 
\end{align*}
Now, plugging \eqref{thm:WLASSOrecovery:eq:lambda} into the above relation, we see that
$$
\xi 
\leq \left[\frac{\tau^2}{(1+\rho)\theta} + (1+\rho)\theta + 2\tau\right]
\sqrt{K(s)} \|\bm{e}\|_2.
$$
Finally, recalling Lemma~\ref{lem:K(s)bounds}, from \eqref{thm:WLASSOrecovery:eq:l1uest} we obtain \eqref{thm:WLASSOrecovery:eq:Linfest} as in Theorem~\ref{thm:WQCBPerror-blind}.

Now, Lemma~\ref{lem:l1tol2} combined with \eqref{thm:WLASSOrecovery:eq:l1uest} yields
\begin{equation}
\label{thm:WLASSOrecovery:eq:L2est_partial}
\|f-\tilde{f}\|_{L^2_\nu}
\leq \|f - f_\Lambda\|_{L^2_\nu}
	+\frac{2(1+\rho)(2+\rho)}{1-\rho} \frac{\sigma_{s,L}(\bm{x})_{1,\bm{u}}}{\sqrt{K(s)}}  	+ \frac{\zeta}{(1-\rho)\sqrt{K(s)}} 
\end{equation}
where 
\begin{equation}
\label{thm:WLASSOrecovery:eq:defzeta}
\zeta := (1+\rho)(2+\rho)(\|\hat{\bm{x}}\|_{1,\bm{u}} - \|\bm{x}\|_{1,\bm{u}}) 
	+ (5+\rho)\tau\sqrt{K(s)}\|\mat{A}(\hat{\bm{x}}_\Lambda-\bm{x}_\Lambda)\|_2
\end{equation}
Employing Young's inequality with $\omega = (1+\rho)(2+\rho)\lambda$, $a= \|\hat{\bm{e}}\|_2$, and $(5+\rho)\tau\sqrt{K(s)}$ and using analogous manipulations as before, we have
$$
\zeta \leq 
\left[\frac{(5+\rho)^2 \tau^2}{4(1+\rho)(2+\rho)\theta} + (1+\rho)(2+\rho)\theta +  (5+\rho)\tau
\right]
\sqrt{K(s)} \|\bm{e}\|_2.
$$
Combining the above inequality with \eqref{thm:WLASSOrecovery:eq:L2est_partial} and recalling Lemma~\ref{lem:K(s)bounds} gives \eqref{thm:WLASSOrecovery:eq:L2est}, as required.
\end{proof}

%%%%%%%%%%%%%%%%%%%%%%%%%%%%%%%%%%%%%%%%%%%%%%%%%%%%%%
\subsubsection{Robustness of WSR-LASSO}
\label{sec:WSRLASSO}

We prove a robustness under unknown error for the WSR-LASSO decoder \eqref{eq:WSRLASSO}. Our theory suggests that the tuning parameter should be chosen directly proportional to the quantity $\sqrt{K(s)}$. This choice is advantageous, since it is independent of the unknown error $\bm{e}$.

\begin{thm}[WSR-LASSO robust recovery]
\label{thm:WSRLASSOrecovery}
Let $0 < \varepsilon < 1$, $0 < \delta < 1$, $2 \leq s \leq 2^{d+1}$, $\Lambda = \Lambda^{\text{HC}}_{d,s}$ be the hyperbolic cross index set defined in \eqref{eq:defHC}, and $\{\phi_{\bm{i}}\}_{\bm{i} \in \mathbb{N}_0^d}$ be the tensor Legendre or Chebyshev polynomial basis. Draw $\bm{t}_1,\ldots,\bm{t}_m$ independently according to the corresponding measure $\nu$, with 
$$
m \gtrsim s^\gamma L(s,n,\delta,\varepsilon),
$$
where $\gamma$ is defined as in \eqref{eq:defgamma}  and $L(s,n,\delta,\varepsilon)$ is the polylogatrithmic factor defined in \eqref{eq:defpolylog}. Then, the following holds with probability at least $1-\varepsilon$. For any $f \in L^2_\nu(D) \cap L^\infty(D)$ expanded as in \eqref{eq:fexpansion}, the approximation $\tilde{f}$ defined as in \eqref{eq:ftilde} computed using the WSR-LASSO decoder \eqref{eq:WSRLASSO} with tuning parameter 
\begin{equation}
\label{thm:WSRLASSOrecovery:eq:lambda}
\lambda = \theta \sqrt{K(s)}, 
\quad \text{with } \theta \geq \frac{(5+\rho)\tau}{(1+\rho)(2+\rho)}
\end{equation}
where $\rho = 4\delta / (1-\delta)$ and $\tau = \sqrt{1-\delta}/(1-\delta)$, satisfies
\begin{align}
\label{thm:WSRLASSOrecovery:eq:Linfest}
\|f-\tilde{f}\|_{L^\infty} 
&\leq C_1 \sigma_{s,L}(\bm{x})_{1,\bm{u}} + C_2 s^{\gamma/2} \|\bm{e}\|_2,\\
\label{thm:WSRLASSOrecovery:eq:L2est}
\|f-\tilde{f}\|_{L^2_\nu}
& \leq C_3 \frac{\sigma_{s,L}(\bm{x})_{1,\bm{u}}}{s^{\gamma/2}} + C_4 \|\bm{e}\|_2 + \|f-f_\Lambda\|_{L^2_\nu},
\end{align}
where $f_\Lambda$ is defined as in \eqref{eq:fLambda} and the constants are
\begin{align*}
C_1 = \frac{3+\rho}{1-\rho}, \quad 
C_2 =  \frac{(1+\rho)\theta + 2\tau}{1-\rho}, \quad
C_3 = \frac{4(1+\rho)(2+\rho)}{1-\rho}, \quad
C_4 = \frac{(1+\rho)(2+\rho)\theta +(5+\rho) \tau}{1-\rho}.
\end{align*}
\end{thm}
\begin{proof}
The proof is similar to that of Theorem~\ref{thm:WLASSOrecovery}. First, observe that \eqref{eq:WSRLASSO} admits the equivalent augmented formulation
$$
(\hat{\bm{x}}_\Lambda, \hat{\bm{e}})
= \arg \min_{(\bm{z},\bm{d})\in\mathbb{C}^{n}\times \mathbb{C}^m}
\|\bm{z}\|_{1,\bm{u}} + \lambda \|\bm{d}\|_2  \text{ s.t. } \mat{A}\bm{z} + \bm{d} = \bm{y}.
$$
Moreover, Theorem~\ref{thm:1-lev_dist_bound} (that holds with probability at least $1-\varepsilon$) implies the $\ell^1_{\bm{u}}$ estimate \eqref{thm:WLASSOrecovery:eq:l1uest}.
Now, we observe that $\rho < 1$ and \eqref{thm:WSRLASSOrecovery:eq:lambda} yield
$$
\lambda \geq \frac{(5+\rho)\tau\sqrt{K(s)}}{(1+\rho)(2+\rho)}\geq\frac{2\tau\sqrt{K(s)}}{1+\rho}.
$$ 
Therefore, using the above inequality and the optimality of $(\hat{\bm{x}}_\Lambda,\hat{\bm{e}})$, we obtain
\begin{align*}
\xi
& \leq (1+\rho) \left(\|\hat{\bm{x}}_\Lambda\|_{1,\bm{u}} 
	+ \lambda \|\hat{\bm{e}}\|_2\right)
	+ 2\tau \sqrt{K(s)} \|\bm{e}\|_2
	- (1+\rho) \|\bm{x}_\Lambda\|_{1,\bm{u}}\\
& \leq [(1+\rho) \lambda + 2\tau \sqrt{K(s)}] \|\bm{e}\|_2
 = [(1+\rho)\theta +2 \tau] \sqrt{K(s)}  \|\bm{e}\|_2, 
\end{align*}
where $\xi$ is defined as in \eqref{thm:WLASSOrecovery:eq:defxi}. Plugging the above inequality into \eqref{thm:WLASSOrecovery:eq:l1uest} gives \eqref{thm:WSRLASSOrecovery:eq:Linfest}. 

Again, analogously to the proof of Theorem~\ref{thm:WLASSOrecovery}, we have that the $L^2_\nu$ estimate \eqref{thm:WLASSOrecovery:eq:L2est_partial} holds. Then, using \eqref{thm:WSRLASSOrecovery:eq:lambda} and an argument analogous to that employed to bound $\xi$,  we obtain
$$
\zeta
\leq [(1+\rho)(2+\rho)\theta +(5+\rho) \tau] \sqrt{K(s)}  \|\bm{e}\|_2,
$$
where $\zeta$ is defined as in \eqref{thm:WLASSOrecovery:eq:defzeta}. Recalling Lemma~\ref{lem:K(s)bounds}, we have \eqref{thm:WSRLASSOrecovery:eq:L2est}. The proof is thus complete.
\end{proof}

%%%%%%%%%%%%%%%%%%%%%%%%%%%%%%%%%%%%%%%%%%%%%%%%%%%%%%
\subsection{WLAD-LASSO}
\label{sec:l1pen}

The proof strategy to show the robustness of WLAD-LASSO under unknown error still follows the guidelines of Remark~\ref{rem:proof_startegy}. Yet, in this case we need to introduce a further level of technical difficulty, dealing with the framework of sparsity in levels, introduced in \cite{Adcock2017breaking}. In particular, we will make use of a weighted version of the 2-level sparsity. We  introduce this notion and some related technical lemmas in \S\ref{sec:WLADLASSO_prelim}. Then, we  show the robustness of WLAD-LASSO in \S\ref{sec:WLADLASSO}.

For the sake of generality, the results in this section are proved for a variant of \eqref{eq:WLADLASSO} with weighted $\ell^1$ penalization of the data-fidelity term\begin{equation}
\label{eq:WLADLASSOweighted}
\hat{\bm{x}}_\Lambda 
:= \arg \min_{\bm{z} \in \mathbb{C}^n} 
\|\bm{z}\|_{1,\bm{u}} + \lambda \|\mat{A}\bm{z} - \bm{y}\|_{1,\bm{v}},
\end{equation}
where $\bm{v} \in \mathbb{R}^m$ is such that $v_i \geq 1$ for every $i \in [m]$.

\subsubsection{Preliminaries}
\label{sec:WLADLASSO_prelim}

%\RED{[As observed in \S\ref{sec:rec_err}, I'm using the following rationale: capital letters for weighted sparsity (in accordance with the notation $K(s)$) and small letters for standard sparsity. I think this makes the notation in this section cleaner. Let me know if you have a better suggestion.]}

We first introduce some notation regarding the weighted 2-level sparsity framework \cite{Adcock2017breaking}.  

\begin{defn}[Weighted 2-level sparsity notation]
\label{def:2-lev_notation}
Let $n,p \in \mathbb{N}$. Consider a vector $\bm{z} \in \mathbb{C}^{n+p}$ and weights $\bm{w}\in\mathbb{R}^{n+p}$, partitioned as
$$
\bm{z} = \bmat{\bm{z}_1\\ \bm{z}_2}, \quad
\bm{w}=\bmat{\bm{w}_1\\\bm{w}_2},
$$
such that $\bm{z}_1 \in \mathbb{C}^n$, $\bm{z}_2 \in \mathbb{C}^p$, $\bm{w}_1 \in \mathbb{R}^n$, and $\bm{w}_2 \in \mathbb{R}^p$. Then, given $\bm{J}:=(J_1, J_2)\in\mathbb{R}^2$, $\bm{z}$ is said to be \emph{weighted $\bm{J}$-sparse} if
$$
|\supp(\bm{z}_i)|_{\bm{w}_i} \leq J_i, \quad \forall i = 1,2.
$$
The \emph{weighted best $\bm{J}$-term approximation error} to $\bm{z}$ is
$$
\sigma_{\bm{J}}(\bm{z})_{1,\bm{w}} 
:= \inf_{\bm{z}':|\supp(\bm{z}'_i)|_{\bm{w}_i} \leq J_i} \|\bm{z}-\bm{z}'\|_{1,\bm{w}}.
$$
\end{defn}
We also introduce a \emph{scale} $\lambda > 0$ that will allow us to switch from the two- to the one-level formalism.
\begin{defn}[Scaled weights and scaled sparsity]
Let $\lambda >0$ be a \emph{scale}, then for every $\bm{w} = \bmat{\bm{w}_1\\\bm{w}_2} \in \mathbb{C}^{n+p}$ and $\bm{J} = (J_1,J_2) \in \mathbb{R}^{2}$, we define the \emph{scaled weights} $\bm{w}_\lambda$ and the \emph{scaled sparsity} $J_\lambda$ as
\begin{equation}
\label{eq:defwJlambda}
\bm{w}_\lambda := \bmat{\bm{w}_1\\ \lambda \bm{w}_2}, \quad 
J_\lambda = J_1 + \lambda^2 J_2.
\end{equation}
\end{defn}
Notice that if $\bm{z}$ is $\bm{J}$-sparse, then
\begin{equation}
\label{eq:l1wlambda-l2_ineq}
\|\bm{z}\|_{1,\bm{w}_\lambda} \leq \sqrt{J_\lambda} \|\bm{z}\|_2.
\end{equation}
Using the notation above, the WLAD-LASSO optimization program \eqref{eq:WLADLASSOweighted} admits an equivalent augmented formulation
as a weighted basis pursuit program, i.e.\
\begin{equation}
\label{eq:WLADLASSO_aug}
\bmat{\hat{\bm{x}}_\Lambda\\ \hat{\bm{e}}}
= \arg \min_{\bm{z} \in \mathbb{C}^{n+m}}
\|\bm{z}\|_{1,\bm{w}_\lambda}  \text{ s.t. } \mat{M}\bm{z} = \bm{y},
\end{equation}
where $\bm{w}:=\bmat{\bm{u}\\\bm{v}}$, $\mat{M} :=\bmat{\mat{A} & \mat{I}}$ and $\bm{w}_\lambda$ is defined as in \eqref{eq:defwJlambda}. %\RED{[Comment about weights not being $\geq 1$ removed.]}

We are now in a position to define the 2-level weighted robust null space property.
\begin{defn}[2-level weighted robust null space property]
\label{def:2levNSP}
A matrix $\mat{M} \in \mathbb{C}^{m \times (N+p)}$ is said to satisfy the \emph{2-level weighted robust null space property} of scale $\lambda > 0$ and order $\bm{J} = (J_1, J_2)$ if
$$
\|\bm{z}_{S_1 \cup S_2}\|_2 
\leq 
\frac{\rho}{\sqrt{J_\lambda}} \|\bm{z}_{{(S_1 \cup S_2)}^c}\|_{1,\bm{w}_\lambda}
+ \tau \|\mat{M} \bm{z}\|_2,
$$
for every $S_1 \subseteq [n], S_2 \subseteq n + [p]$, such that $|S_i|_{\bm{w}_i} \leq J_i$, for $i = 1,2$.

\end{defn}
Notice that in the one-level case (i.e., when $p = 0$) and with $J_1 = K(s)$ Definition~\ref{def:2levNSP} and Definition~\ref{def:lowerNSP} are equivalent.

We now prove a result analogous to Theorem~\ref{thm:1-lev_dist_bound}. It is an upper bound to the $\ell^1_{\bm{w}_\lambda}$ distance between two arbitrary finite-dimensional vectors under the weighted robust null space property. This result is the 2-level generalization of \cite[Theorem 4.2]{Rauhut2016} and is based on the same proof strategy. 
\begin{thm}[$\ell^1_{\bm{w}_\lambda}$ distance bound]
\label{thm:2-lev_dist_bound}
If the matrix $\mat{M} \in \mathbb{C}^{m \times (N+p)}$ has the 2-level weighted robust null space property, then for every $\bm{z}, \hat{\bm{z}} \in \mathbb{C}^{n+p}$  the following holds:
\begin{equation}
\|\bm{z}-\hat{\bm{z}}\|_{1,\bm{w}_\lambda} 
\leq 
\frac{1+\rho}{1-\rho} 
(2 \sigma_{\bm{J}}(\bm{z})_{1,\bm{w}_\lambda} 
+ \|\hat{\bm{z}}\|_{1,\bm{w}_\lambda} 
- \|\bm{z}\|_{1,\bm{w}_\lambda})
+ \frac{2\tau \sqrt{J_\lambda}}{1-\rho} 
\|\mat{M}(\bm{z}-\hat{\bm{z}})\|_2.
\end{equation}
\end{thm}
\begin{proof}
Choose $S_1 \subseteq [n]$ and $S_2 \subseteq [p]$ such that $\sigma_{\bm{J}}(\bm{z})_{1,\bm{w}_\lambda} = \|\bm{z}-\bm{z}_{T}\|_{1,\bm{w}_\lambda} = \|\bm{z}_{T^c}\|_{1,\bm{w}_\lambda}$ and define $T:=S_1 \cup S_2$. Using the 2-level weighted robust null space property and recalling \eqref{eq:l1wlambda-l2_ineq}, we obtain
\begin{equation}
\label{thm:distbound:eq:(x-z)_S}
\|(\bm{z}-\hat{\bm{z}})_{T}\|_{1,\bm{w}_\lambda} 
\leq 
\rho \|(\bm{z}-\hat{\bm{z}})_{T^c}\|_{1,\bm{w}_\lambda}
+
\tau \sqrt{J_\lambda} \|\mat{M}(\bm{z}-\hat{\bm{z}})\|_2.
\end{equation}
Then, we estimate
\begin{align*}
\|(\bm{z}-\hat{\bm{z}})_{T^c}\|_{1,\bm{w}_\lambda}
& \leq
\|\bm{z}_{T^c}\|_{1,\bm{w}_\lambda} 
+ \|\hat{\bm{z}}_{T^c}\|_{1,\bm{w}_\lambda}\\
& =
2 \|\bm{z}_{T^c}\|_{1,\bm{w}_\lambda}
+ \|\bm{z}_{T}\|_{1,\bm{w}_\lambda}
- \|\hat{\bm{z}}_{T}\|_{1,\bm{w}_\lambda}
+ \|\hat{\bm{z}}\|_{1,\bm{w}_\lambda}
- \|\bm{z}\|_{1,\bm{w}_\lambda}
\\
& \leq 
2 \sigma_{\bm{J}}(\bm{z})_{1,\bm{w}_\lambda}
+ \|(\bm{z}-\hat{\bm{z}})_{T}\|_{1,\bm{w}_\lambda}
+ \|\hat{\bm{z}}\|_{1,\bm{w}_\lambda}
- \|\bm{z}\|_{1,\bm{w}_\lambda}.
\end{align*}
Plugging \eqref{thm:distbound:eq:(x-z)_S} into the inequality above and solving for $\|(\bm{z}-\hat{\bm{z}})_{T^c}\|_{1,\bm{w}_\lambda}$, we see that
\begin{equation}
\label{thm:distbound:eq:(x-z)_Sc}
\|(\bm{z}-\hat{\bm{z}})_{T^c}\|_{1,\bm{w}_\lambda}
\leq \frac{2 \sigma_{\bm{J}}(\bm{z})_{1,\bm{w}_\lambda}
+ \tau \sqrt{J_\lambda} \|\mat{M}(\bm{z}-\hat{\bm{z}})\|
+ \|\hat{\bm{z}}\|_{1,\bm{w}_\lambda}
- \|\bm{z}\|_{1,\bm{w}_\lambda}}{1-\rho}.
\end{equation}
Combining \eqref{thm:distbound:eq:(x-z)_S} and \eqref{thm:distbound:eq:(x-z)_Sc} with the triangle inequality $
\|\bm{z}-\hat{\bm{z}}\|_{1,\bm{w}_\lambda} 
\leq
\|(\bm{z}-\hat{\bm{z}})_{T}\|_{1,\bm{w}_\lambda}
+
\|(\bm{z}-\hat{\bm{z}})_{T^c}\|_{1,\bm{w}_\lambda}
$
concludes the proof.
\end{proof}

%%%%%%%%%% Technical stuff for the LAD-LASSO

Before proving Theorem~\ref{thm:WLADLASSOrecovery}, we need to introduce some more technical elements to the picture. First, we consider the 2-level weighted restricted isometry property (a weighted version of the restricted isometry property in levels, introduced in \cite{Bastounis2014}, also studied in \cite{Adcock2017compressed} for sparse corruptions). 
\begin{defn}[2-level weighted restricted isometry property]
Let $\bm{J}\in \mathbb{R}^2$, and $\bm{w} \in \mathbb{R}^{n+p}$ partitioned as in Definition~\ref{def:2-lev_notation}. Then, a matrix $\mat{M} \in \mathbb{C}^{m \times (n+p)}$ is said to have the \emph{2-level weighted restricted isometry property} of order $\bm{J}$ if there exists a constant $0 < \delta < 1$ such that 
$$
(1-\delta)\|\bm{z}\|_2^2 \leq \|\mat{M}\bm{z}\|_2^2 \leq (1+\delta)\|\bm{z}\|_2^2,
\quad \forall \bm{z} \in \mathbb{C}^{n+p}: |\supp(\bm{z}_i)|_{\bm{w}_i}\leq J_i, \; \forall i= 1,2.
$$ 
The smallest constant such that this property holds is called the $\bm{J}^{th}$ \emph{2-level weighted restricted isometry constant} of $\mat{M}$ and it is denoted as $\delta_{\bm{J}}$.
\end{defn}

As a consequence of the definition above, we have the following technical result. 
\begin{lem}
\label{lem:disj_bound}
Consider $\bm{z}, \hat{\bm{z}} \in \mathbb{C}^{n+p}$, weights $\bm{w}  \in \mathbb{R}^{n+p}$, and $\bm{J},\hat{\bm{J}} \in \mathbb{R}^2$ partitioned as in Definition~\ref{def:2-lev_notation} and such that $|\supp(\bm{z}_i)|_{\bm{w}_i}\leq J_i$ and $|\supp(\hat{\bm{z}}_i)|_{\bm{w}_i}\leq \hat{J}_i$, for $i = 1,2$. Moreover, assume that $\supp(\bm{z}_1) \cap \supp(\hat{\bm{z}}_1) = \supp(\bm{z}_2) \cap \supp(\hat{\bm{z}}_2) =\emptyset$. Then, it holds
$$
|\left<\mat{M}\bm{z},\mat{M}\hat{\bm{z}}\right>|
\leq \delta_{\bm{J}+\hat{\bm{J}}} \|\bm{z}\|_2 \|\hat{\bm{z}}\|_2, \quad \text{with }\mat{M} = \bmat{\mat{A} & \mat{I}}.
$$
\end{lem}
\begin{proof}
Let $S_1:=\supp(\bm{z}_1)\cup \supp(\hat{\bm{z}}_1)$ and $S_2:=\supp(\hat{\bm{z}}_2)\cup \supp(\hat{\bm{z}}_2)$. Then $|S_i|_{\bm{w}_i} \leq J_i + \hat{J}_i$, for $i = 1,2$. Using the disjointedness of the supports of $\bm{z}_1$ and $\hat{\bm{z}}_1$ and of those of $\bm{z}_2$ and $\hat{\bm{z}}_2$, we have
\begin{align*}
|\left<\mat{M}\bm{z},\mat{M}\hat{\bm{z}}\right>|
&= \left|\left<\mat{M}\bm{z},\mat{M} \hat{\bm{z}}\right>
-\left<\bm{z},\hat{\bm{z}}\right>\right|
\\
&=\left|\left<\bmat{\mat{A}_{S_1}\ \mat{I}_{S_2}}\bm{z}_{S_1 \cup S_2}\bmat{\mat{A}_{S_1}\ \mat{I}_{S_2}} \hat{\bm{z}}_{S_1 \cup S_2}\right>
-\left<\bm{z}_{S_1 \cup S_2},\hat{\bm{z}}_{S_1 \cup S_2}\right>\right|
\\
& = \left|\left<
\left(\bmat{\mat{A}_{S_1}\ \mat{I}_{S_2}}^*
\bmat{\mat{A}_{S_1}\ \mat{I}_{S_2}}-\mat{I}\right)
\bm{z}_{S_1 \cup S_2},
\hat{\bm{z}}_{S_1 \cup S_2}
\right>\right|
\\
& \leq 
\|\bmat{\mat{A}_{S_1}\ \mat{I}_{S_2}}^*
\bmat{\mat{A}_{S_1}\ \mat{I}_{S_2}}-\mat{I}\|_{2\rightarrow2}
\left\|\bm{z}_{S_1 \cup S_2}\right\|_2
\left\|\hat{\bm{z}}_{S_1 \cup S_2}\right\|_2
\\
& \leq \delta_{\bm{J}+\hat{\bm{J}}}\|\bm{z}\|_2 \|\hat{\bm{z}}\|_2.
\end{align*}
This completes the proof.
\end{proof}

In view of Remark~\ref{rem:proof_startegy}, we now prove that the restricted isometry property implies the robust null space property in the 2-level weighted setting.
\begin{thm}[2-level weighted restricted isometry property  $\Rightarrow$ 2-level weighted robust null space property]
\label{thm:2-levRIP->2-levNSP}
Given $\bm{J} = (K,H)\in \mathbb{R}^2$, any matrix $\mat{M} \in \mathbb{C}^{m \times (n+p)}$ of the form $\mat{M} = \bmat{\mat{A} & \mat{I}}$ with $(3\bm{J})^{th}$ weighted restricted isometry constant
\begin{equation}
\label{thm:2-levRIP->2-levNSP:eq:cond_delta}
\delta_{3\bm{J}}<\frac{1}{1+4\Theta},
\quad \text{with } \Theta := \frac{\sqrt{K + \lambda^2 H}}{\min\{\sqrt{K},\lambda \sqrt{H}\}},
\end{equation}
with respect to weights $\bm{w} = \bmat{\bm{u}\\\bm{v}}\in\mathbb{R}^{n+p}$ such that $\bm{J}\geq \frac{4}{3}(\|\bm{u}\|_{\infty}^2,\|\bm{v}\|_\infty^2)$ satisfies the 2-level weighted robust null space property of order $\bm{J}$ with constants 
\begin{equation}
\label{thm:2-levRIP->2-levNSP:eq:rho_tau}
\rho=\frac{4\delta_{3\bm{J}}\Theta}{(1-\delta_{3\bm{J}})},\quad
\tau=\frac{\sqrt{1+\delta_{3\bm{J}}}}{1-\delta_{3\bm{J}}}.
\end{equation}
\end{thm}
\begin{proof}
Consider  $\bm{x}\in \mathbb{C}^{n}$ and $S \subset[n]$ with $|S|_{\bm{u}}\leq K$. We partition $S^c = [n] \setminus S$ into disjoint blocks $S_1, S_2, \ldots$ with $K-\|\bm{u}\|^2_{\infty} \leq |S_\ell|_{\bm{u}} \leq K$ according to the non-increasing rearrangement of $(|x_j| u_j^{-1})_{j \in S^c}$. As a result, we have $|x_{j}|u_{j}^{-1} \leq |x_{k}|u_{k}^{-1}$ for all $j \in S_\ell, k\in S_{\ell-1}$, and $\ell\geq 2$. Similarly, we let $\bm{e} \in \mathbb{C}^{p}$ and $T \subset [p]$ with $|T|_{\bm{w}_2} \leq H$. We partition $T^c$ into blocks $T_1,T_2,\ldots$ with $H-\|\bm{v}\|^2_{\infty} \leq |S_\ell|_{\bm{v}} \leq H$ according to the non-increasing rearrangement of $(|e_j| v_j^{-1})_{j \in T^c}$. We have $|e_{j}|v_{j}^{-1} \leq |e_{k}|v_{k}^{-1}$ for all $j \in T_\ell, k\in T_{\ell-1}$, and $\ell\geq 2$.  Using  and the 2-level weighted restricted isometry property, we have
\begin{align*}
\left\|\bmat{\bm{x}_{S \cup S_1}\\\bm{e}_{T \cup T_1}}\right\|_2^2 
&\leq \frac{1}{1-\delta_{2\bm{J}}} 
\left\|\mat{M}\bmat{\bm{x}_{S \cup S_1}\\\bm{e}_{T \cup T_1}}\right\|_2^2 \\
&  = \frac{1}{1-\delta_{2\bm{J}}} 
\left<\mat{M}\bmat{\bm{x}_{S \cup S_1}\\\bm{e}_{T \cup T_1}},\mat{M} \bmat{\bm{x} \\\bm{e}}-\sum_{\ell \geq 2}\mat{M} \bmat{\bm{x}_{S_\ell}\\\bm{e}_{T_\ell}}\right>\\
& = \frac{1}{1-\delta_{2\bm{J}}}
\left( 
\left<\mat{M}\bmat{\bm{x}_{S \cup S_1}\\\bm{e}_{T \cup T_1}},\mat{M} \bmat{\bm{x} \\\bm{e}}\right>
 - \sum_{\ell \geq 2}\left<\mat{M}\bmat{\bm{x}_{S \cup S_1}\\\bm{e}_{T \cup T_1}},\mat{M}\bmat{\bm{x}_{S_\ell}\\\bm{e}_{T_\ell}}\right>\right)
\end{align*}
Then, employing agin the restricted isometry property, Lemma~\ref{lem:disj_bound}, and the Cauchy-Schwarz inequality, we see that
\begin{align*}
\left\|\bmat{\bm{x}_{S \cup S_1}\\\bm{e}_{T \cup T_1}}\right\|_2^2 
& \leq \frac{1}{1-\delta_{2\bm{J}}}	\left(
\left\|\mat{M}\bmat{\bm{x}_{S \cup S_1}\\\bm{e}_{T \cup T_1}}\right\|_2
\|\mat{A}\bm{x}+\bm{e}\|_2
 +\delta_{3\bm{J}}
\left\|\bmat{\bm{x}_{S  \cup S_1}\\\bm{e}_{T  \cup T_1}}\right\|_2
\sum_{\ell \geq 2}\left\|\bmat{\bm{x}_{S_\ell}\\\bm{e}_{T_\ell}}\right\|_2\right)
\\
& \leq \frac{\sqrt{1+\delta_{2\bm{J}}}}{1-\delta_{2\bm{J}}}
\left\|\bmat{\bm{x}_{S \cup S_1}\\ \bm{e}_{T \cup T_1}}\right\|_2
\|\mat{A}\bm{x}+\bm{e}\|_2
+\frac{\delta_{3\bm{J}}}{1-\delta_{2\bm{J}}}
\left\|\bmat{\bm{x}_{S \cup S_1} \\ \bm{e}_{T \cup T_1}}\right\|_2
\sum_{\ell \geq 2}\left\|\bmat{\bm{x}_{S_\ell}\\\bm{e}_{T_\ell}}\right\|_2.
\end{align*}
Dividing both sides by $\left\|\bmat{\bm{x}_{S \cup S_1}\\\bm{e}_{T \cup T_1}}\right\|_2$ and using the fact that $\delta_{2\bm{J}} \leq \delta_{3\bm{J}}$ yields
$$
\left\|\bmat{\bm{x}_{S}\\\bm{e}_T}\right\|_2 
\leq \left\|\bmat{\bm{x}_{S \cup S_1}\\\bm{e}_{T \cup T_1}}\right\|_2
\leq  
\frac{\sqrt{1+\delta_{3\bm{J}}}}{1-\delta_{3\bm{J}}}
\|\mat{A}\bm{x}+\bm{e}\|_2
+
\frac{\delta_{3\bm{J}}}{1-\delta_{3\bm{J}}}
\sum_{\ell \geq 2}\left\|\bmat{\bm{x}_{S_\ell}\\\bm{e}_{T_\ell}}\right\|_2.
$$
Thanks to the definition of the $S_\ell$'s, for every $j \in S_\ell$ and $k \in S_{\ell-1}$  with $\ell \geq 2$, we have  $|x_j|u_j^{-1} \leq|x_k|u_k^{-1}$. Moreover, let  $\alpha_k := u_k^2 / |S_{\ell-1}|_{\bm{u}}$. Then, $\sum_{k \in S_{\ell-1}}\alpha_k = 1$ and, for every $j \in S_\ell$ and $\ell \geq 2$, we see that
$$
\frac{|x_j|}{u_j} 
\leq\sum_{k \in S_{\ell-1}}\alpha_k\frac{|x_k|}{u_k}
= \frac{\|\bm{x}_{S_{\ell-1}}\|_{1,\bm{u}}}{|S_{\ell-1}|_{\bm{u}}} 
\leq \frac{\|\bm{x}_{S_{\ell-1}}\|_{1,\bm{u}}}{K-\|\bm{u}\|_{\infty}^2}.
$$
Therefore, for any $\ell \geq 2$, using that $K \geq \frac43 \|\bm{u}\|_\infty^2$, we have
$$
\|\bm{x}_{S_\ell}\|_2
\leq \sqrt{K}\frac{\|\bm{x}_{S_{\ell-1}}\|_{1,\bm{u}}}{K-\|\bm{u}\|_{\infty}^2}
\leq 
\frac{4 \|\bm{x}_{S_{\ell-1}}\|_{1,\bm{u}}}{\sqrt{K}}. 
$$
Using an analogous argument, we obtain
$\|\bm{e}_{T_\ell}\|_2 
\leq \frac{4 \|\bm{e}_{T_{\ell-1}}\|_{1,\bm{v}}}{\sqrt{H}}$ for every $\ell \geq 2$.

Therefore, we estimate
\begin{align*}
\left\|\bmat{\bm{x}_{S}\\\bm{e}_{T}}\right\|_2 
&\leq  \frac{\delta_{3\bm{J}}}{1-\delta_{3\bm{J}}}
\sum_{\ell \geq 2}\left\|\bmat{\bm{x}_{S_\ell}\\\bm{e}_{T_\ell}}\right\|_2
+\frac{\sqrt{1+\delta_{3\bm{J}}}}{1-\delta_{3\bm{J}}}
\|\mat{A}\bm{x}+\bm{e}\|_2
\\
&\leq 
\frac{\delta_{3\bm{J}}}{1-\delta_{3\bm{J}}}
\sum_{\ell \geq 2}\left(\|\bm{x}_{S_\ell}\|_2+\|\bm{e}_{T_\ell}\|_2\right)
+
\frac{\sqrt{1+\delta_{3\bm{J}}}}{1-\delta_{3\bm{J}}}
\|\mat{A}\bm{x}+\bm{e}\|_2\\
&\leq \frac{ 4\delta_{3\bm{J}}}{1-\delta_{3\bm{J}}}
\sum_{\ell \geq 2}\left(\frac{\|\bm{x}_{S_{\ell-1}}\|_{1,\bm{u}}}{\sqrt{K}}+\lambda\frac{\|\bm{e}_{T_{\ell-1}}\|_{1,\bm{v}}}{\lambda\sqrt{H}}\right)
+\frac{\sqrt{1+\delta_{3\bm{J}}}}{1-\delta_{3\bm{J}}}
\|\mat{A}\bm{x}+\bm{e}\|_2\\
& = \frac{4\delta_{3\bm{J}}}{1-\delta_{3\bm{J}}}\left(\frac{\|\bm{x}_{S^c}\|_{1,\bm{u}}}{\sqrt{K}}
+\frac{\lambda\|\bm{e}_{T^c}\|_{1,\bm{v}}}{\lambda\sqrt{H}}\right)
+\frac{\sqrt{1+\delta_{3\bm{J}}}}{1-\delta_{3\bm{J}}}
\|\mat{A}\bm{x}+\bm{e}\|_2\\
& \leq \frac{4\delta_{3\bm{J}}}{1-\delta_{3\bm{J}}}
\max{\left\{\frac{1}{\sqrt{K}},\frac{1}{\lambda\sqrt{H}}\right\}}
\left\|\bmat{\bm{x}_{S^c}\\ \bm{e}_{T^c}}\right\|_{1,\bm{w}_\lambda}
+\frac{\sqrt{1+\delta_{3\bm{J}}}}{1-\delta_{3\bm{J}}}
\left\|\mat{M} \bmat{\bm{x} \\ \bm{e}}\right\|_2.
\end{align*}
Now, observing that
$$
\max\left\{\frac{1}{\sqrt{K}}, \frac{1}{\lambda\sqrt{H}}\right\}
= \frac{\Theta}{\sqrt{H + \lambda^2 K}},
$$
we have shown the 2-level null-space property with constants $\rho$ and $\tau$ defined by \eqref{thm:2-levRIP->2-levNSP:eq:rho_tau}. For $\rho$ to be strictly smaller than $1$, we need condition \eqref{thm:2-levRIP->2-levNSP:eq:cond_delta}. This concludes the proof.
\end{proof}

\begin{rmrk}
The constant 4/3 in the hypothesis $\bm{J} \geq \frac43(\|\bm{u}\|_\infty^2,\|\bm{v}\|_\infty^2)$ of Theorem~\ref{thm:2-levRIP->2-levNSP} can be made arbitrarily close to 1. This particular choice has been made in view of Lemma~\ref{lem:norm_u_inf_bound}.
\end{rmrk}

Finally, we establish the 2-level weighted restricted isometry property for the augmented matrix $\bmat{\mat{A} & \mat{I}}$ with high probability provided a suitable lower bound for the sample complexity $m$.
\begin{thm}[Sample complexity $\Rightarrow$ 2-level weighted restricted isometry property]
\label{thm:sample->2-levRIP}
Let $\bm{J} = (K(s), H)$, $0 < \delta < 1$, and suppose that
$$
m \gtrsim K(s) \max\{L(s,n,\delta,\varepsilon), \delta^{-2} H\}
$$
Then, with probability at least $1-\varepsilon$, the matrix $\bmat{\mat{A} & \mat{I}}$ has the  2-level weighted restricted isometry property of order $\bm{J}$ with weights $(\bm{u}_\Lambda, \bm{v})$ and constant $\delta_{\bm{J}} \leq \delta$.
\end{thm}
\begin{proof}
First, notice that \cite[Lemma 3.14]{Adcock2017compressed} can  be easily adapted to the weighted case by noticing that $|\left<\mat{A}\bm{x}, \bm{e}\right>| \leq \sqrt{\frac{K(s) H}{m}} \|\bm{x}\|_2 \|\bm{e}\|_2$,  due to the assumption $\bm{v}\geq 1$. As a consequence, \cite[Theorem 3.15]{Adcock2017compressed} admits an analogous generalization to the weighted case. The thesis now follows from Theorem~\ref{thm:sampleRIP}. 
\end{proof}

Finally, we introduce the concept of  quasi-best weighted $K$-term approximation, introduced and studied in \cite{Rauhut2016}.
\begin{defn}[Quasi-best weighted $K$-term approximation]
\label{def:quasi-best}
Given $\bm{z} \in \mathbb{C}^n$, the corresponding \emph{weighted quasi-best $K$-term approximation error} is defined as
$$
\tilde{\sigma}_{K}(\bm{z})_{p,\bm{w}}=\|\bm{z}-\bm{z}_S\|_{p,\bm{w}},
$$
where $S=\{\pi(1),\ldots,\pi(s)\}$ where $s$ is the maximal number such that $\sum_{j=1}^s w^2_{\pi(j)}\leq K$, and $\pi$ is some permutation such that $|z_{\pi(j)}|^p w_j^{-p} \geq |z_{\pi(j+1)}|^p w_{j+1}^{-p}$ for every $j \in [n]$.

\end{defn}

Equipped with the technical elements discussed above, we are now ready to prove the robustness of WLAD-LASSO under unknown error.

%%%%%%%%%%%%%%%%%%%%%%%%%%%%%%%%%%%%%%%%%%%%%%%%%%%
\subsubsection{Robustness of WLAD-LASSO}
\label{sec:WLADLASSO}

We prove robust recovery error estimates for the WLAD-LASSO decoder \eqref{eq:WLADLASSOweighted} under unknown error.

\begin{thm}[WLAD-LASSO recovery estimate]
\label{thm:WLADLASSOrecovery}
Let $0 < \varepsilon < 1$, $2 \leq s \leq 2^{d+1}$, $H >0$, $\Lambda = \Lambda^{\text{HC}}_{d,s}$ be the hyperbolic cross index set defined in \eqref{eq:defHC}, and $\{\phi_{\bm{i}}\}_{\bm{i} \in \mathbb{N}_0^d}$ be the tensor Legendre or Chebyshev polynomial basis. Let $0 < \delta < 1$ be such that
$$
\delta \leq \frac{1}{1+4 \Theta}, \quad \text{with }\Theta=\frac{\sqrt{K(s)+\lambda^2 H}}{\min{\{\sqrt{K(s)},\lambda \sqrt{H}\}}}.
$$
Draw $\bm{t}_1,\ldots,\bm{t}_m$ independently according to the corresponding measure $\nu$, with 
\begin{equation}\label{thm:WLADLASSOrecovery:eq:samplecomplex}
m \gtrsim s^{\gamma/2} \max\{L(s,n,\delta,\varepsilon), \delta^{-2} H\}
\end{equation}
where $\gamma$ is defined as in \eqref{eq:defgamma} and $L(s,n,\delta,\varepsilon)$ is the polylogarithmic factor defined in \eqref{eq:defpolylog}. Then, the following holds with probability at least $1-\varepsilon$. For any $f \in L^2_\nu(D) \cap L^\infty(D)$ expanded as in \eqref{eq:fexpansion}, the approximation $\tilde{f}$ defined as in \eqref{eq:ftilde} computed using the WLAD-LASSO decoder \eqref{eq:WLADLASSOweighted} satisfies
\begin{equation}
\label{thm:WLADLASSOrecovery:eq:Linfest}
\|f-\tilde{f}\|_{L^\infty} + \lambda \|\bm{e}-(\bm{y}-\mat{A}\hat{\bm{x}}_\Lambda)\|_{1,\bm{v}}
\leq C_1 \left(\sigma_{s,L}(\bm{x})_{1,\bm{u}} + \lambda \sigma_H(\bm{e})_{1,\bm{v}}\right).
\end{equation}
and, additionally assuming $H\geq 2\|\bm{v}\|_{\infty}^2$,
\begin{equation}
\label{thm:WLADLASSOrecovery:eq:L2est}
\|f-\tilde{f}\|_{L^2_\nu} + \|\bm{e}-(\bm{y}-\mat{A}\hat{\bm{x}}_\Lambda)\|_{2}
\leq C_2 (1+\sqrt{\Theta})\left ( \frac{ \sigma_{s,L}(\bm{x})_{1,\bm{u}}}{s^{\gamma/2}}+ \frac{\sigma_H(\bm{e})_{1,\bm{v}}}{\sqrt{H}} \right )+ \|f-f_{\Lambda}\|_{L^2_\nu},
\end{equation}
where $f_\Lambda$ is defined as in \eqref{eq:fLambda} and the constants are 
$$
C_1 = \frac{3+\rho}{1-\rho}, \quad 
C_2 = \frac{4\sqrt{2}(1+\rho)\max\{\rho,2\sqrt{\rho}\}}{1-\rho},
$$
where $\rho = 4\delta / (1-\delta)$, $\tau = \sqrt{1-\delta}/(1-\delta)$.
\end{thm}
\begin{proof}
First, we observe that the sample complexity lower bound \eqref{thm:WLADLASSOrecovery:eq:samplecomplex} is sufficient to guarantee the  2-level weighted restricted isometry property and  the 2-level weighted robust null space property  for the matrix $\bmat{\mat{A} & \mat{I}}$ with probability at least $1-\varepsilon$, due to Theorems~\ref{thm:sample->2-levRIP} \& \ref{thm:2-levRIP->2-levNSP}. 

Now, considering the augmented formulation \eqref{eq:WLADLASSO_aug} of WLAD-LASSO, we apply Theorem~\ref{thm:2-lev_dist_bound} with $\mat{M} = \bmat{\mat{A}& \mat{I}}$, $\bm{z} = \bmat{\bm{x}_\Lambda\\ \bm{e}}$, and $\hat{\bm{z}} = \bmat{\hat{\bm{x}}_\Lambda \\ \hat{\bm{e}}}$. We obtain
\begin{equation}
\label{thm:WLADLASSOrecovery:eq:l1_err_bound}
\|\bm{x}_\Lambda - \hat{\bm{x}}_\Lambda\|_{1,\bm{u}}
+ \lambda \|\bm{e} - \hat{\bm{e}}\|_{1,\bm{v}}
\leq \frac{2(1+\rho)}{1-\rho} \left(\sigma_{s,L}(\bm{x})_{1,\bm{u}} + \lambda \sigma_H(\bm{e})_{1,\bm{v}}\right),
\end{equation}
which, for reasons analogous to  those of the previous robust recovery theorems, implies \eqref{thm:WLADLASSOrecovery:eq:Linfest}.\footnote{Notice that since $\mat{M}(\bm{z}-\hat{\bm{z}}) = \bm{0}$ the constant $\tau$ does not appear in \eqref{thm:WLADLASSOrecovery:eq:l1_err_bound}. Indeed, we are just using a 2-level weighted version of the so-called \emph{stable} null space property (see \cite[\S 4.2]{Foucart2013}).}

As a second step, we now convert the weighted $\ell^1$ error estimate \eqref{thm:WLADLASSOrecovery:eq:l1_err_bound} into an $L^2_\nu$ error estimate as usual. This will require some extra work in this case. 

Recalling Definition~\ref{def:quasi-best}, let $\bm{d} := \bm{x}_\Lambda - \hat{\bm{x}}_\Lambda$ and $\bm{f} := \bm{e} - \hat{\bm{e}}$. Let $S$ and $T$ be such that $\|\bm{d}-\bm{d}_S\|_{1,\bm{u}}=\|\bm{d}_{S^c}\|_{1,\bm{u}}=\tilde{\sigma}_{K(s)}(\bm{d})_{1,\bm{u}}$ and  $\|\bm{f}-\bm{f}_S\|_{1,\bm{v}}=\|\bm{f}_{T^c}\|_{1,\bm{v}}=\tilde{\sigma}_H(\bm{f})_{1,\bm{v}}$. Moreover, using the 2-level notation of Definition~\ref{def:2-lev_notation}, define $\bm{g} := \bmat{\bm{d}\\\bm{f}}$ and $\bm{J}:=(K(s),H)$. Recalling \eqref{eq:defwJlambda}, let $\bm{w}_\lambda := \bmat{\bm{u}\\ \lambda \bm{v}}$, $J_\lambda := K(s) + \lambda^2 H$. Moreover, define
$$
\theta_{\bm{d}}:=\min_{i \in S}{\frac{|d_i|}{u_i}}, \quad
\theta_{\bm{f}}:=\min_{j \in T}{\frac{|f_j|}{v_j}}, \quad
\theta:=\max\left\{\theta_{\bm{d}},\frac{\theta_{\bm{f}}}{\lambda}\right\}.
$$
Using the fact that $S$ and $T$ realize the quasi-best approximation error, we estimate
\begin{align}
\|\bm{g}_{(S \cup T)^c}\|^2_{2} &= \sum_{i \notin S} |d_i|^2 + \sum_{j \notin T} |f_j|^2 
\leq \theta_{\bm{d}}\sum_{i \notin S}u_i|d_i|+\theta_{\bm{f}} \sum_{j \notin T} v_j|f_j| \\
&\leq \theta\left(\sum_{i \notin S}u_i|d_i|+\lambda\sum_{j \notin T}v_j|f_j|\right)
= \theta \|\bm{g}_{(S \cup T)^c}\|_{1,\bm{w}_\lambda}
\label{thm:WLADLASSOrecovery:eq:norm_g_SUT}
\end{align}
Similarly to the proof of Theorem~\ref{thm:2-levRIP->2-levNSP}, let $\alpha_k :=u_k^2 / |S|_{\bm{u}}$. Notice that $\alpha_k \leq (K(s)-\|\bm{u}_\Lambda\|_{\infty}^2)^{-1}u_k^2$ (due to the maximality of $S$, realizing the quasi-best term approximation) and that $\sum_{k \in S}\alpha_k=1$. Now we observe that 
\begin{align*}
\theta_{\bm{d}}^2 
= \sum_{i \in S}{\alpha_i\min_{k \in S}{\left(\frac{|d_k|}{u_k}\right)^{2}}} 
\leq \sum_{i \in S}{\alpha_i\left(\frac{|d_i|}{u_i}\right)^{2}} 
\leq \frac{1}{K(s)-\|\bm{u}_\Lambda\|_{\infty}^2}
\sum_{i \in S}u_i^2\left(\frac{|d_i|}{u_i}\right)^{2}
= \frac{1}{K(s)-\|\bm{u}_\Lambda\|_{\infty}^2}
\|\bm{d}_S\|_{2}^2.
\end{align*}
Recalling that $\frac{3}{4}K(s) \geq \|\bm{u}_\Lambda\|_{\infty}^2$ (due to Lemma~\ref{lem:norm_u_inf_bound}), we get
$\theta_{\bm{d}}\leq \frac{2}{\sqrt{K(s)}}\|\bm{d}_S\|_{2}$. Similarly, we can deduce that $\theta_{\bm{f}} \leq \frac{\sqrt{2}}{\sqrt{H}}\|\bm{f}_T\|_{2}$ thanks to the hypothesis $H \geq 2 \|\bm{v}\|_\infty^2 $. Therefore, using the 2-level weighted robust null space property, we obtain
\begin{align*}
\theta 
=\max{\left\{\theta_{\bm{d}},\frac{\theta_{\bm{f}}}{\lambda}\right\}} 
\leq \frac{2\left(\|\bm{d}_S\|_{2}+\|\bm{f}_T\|_{2}\right) }{\min{\{\sqrt{K(s)},\lambda\sqrt{H}\}}} \leq \frac{2\sqrt{2}\|\bm{g}_{S \cup T}\|_2}{\min{\{\sqrt{K(s)},\lambda\sqrt{H}\}}} 
\leq \frac{2\sqrt{2}\rho\| \bm{g}_{(S \cup T)^c}\|_{1,\bm{w}_{\lambda}}}{\min{\{\sqrt{K(s)},\lambda\sqrt{H}\}} \sqrt{J_\lambda}}.
\end{align*}
Plugging the above inequality into \eqref{thm:WLADLASSOrecovery:eq:norm_g_SUT} and defining
$
\beta:=\frac{ \| \bm{g}_{(S \cup T)^c} \|_{1,\bm{w}_\lambda}}{\sqrt{J_\lambda}},
$ 
we deduce that
$$
\|\bm{g}_{(S \cup T)^c}\|^2_{2} 
\leq \theta \|\bm{g}_{(S \cup T)^c}\|_{1,\bm{w}_\lambda}
\leq  2\sqrt{2}\rho \Theta  \beta^2 \leq 4 \rho \Theta \beta^2.
$$
Using the above estimate and the weighted 2-level robust null space property again, we see that 
\begin{align*}
\tfrac{1}{\sqrt{2}}\left(\| \bm{d} \|_{2} + \| \bm{f}\|_{2}\right) 
&\leq \| \bm{g}\|_2 \leq \|\bm{g}_{S \cup T}\|_2 + \|\bm{g}_{(S \cup T)^c}\|_2\\ 
&\leq \frac{\rho}{\sqrt{J_\lambda}} \|\bm{g}_{(S \cup T)^c}\|_{1,\bm{w}_\lambda} + \|\bm{g}_{(S \cup T)^c}\|_2
\\
& \leq \rho \beta +2\sqrt{\rho\Theta} \beta
\leq \max\{\rho,2\sqrt{\rho}\}\left(1+\sqrt{\Theta}\right)\beta.
\end{align*}
Now, taking advantage of \eqref{thm:WLADLASSOrecovery:eq:l1_err_bound}, we have  
\begin{align*}
\beta  
\leq \frac{ \| \bm{g}\|_{1,\bm{w}}}{\sqrt{J_\lambda}} 
 \leq \frac{2(1+\rho)}{1-\rho} \frac{\sigma_{s,L}(\bm{x})_{1,\bm{u}} + \lambda \sigma_H(\bm{e})_{1,\bm{v}}}{\sqrt{K(s) + \lambda^2 H}} \leq \frac{2(1+\rho)}{1-\rho} \left(\frac{\sigma_{s,L}(\bm{x})_{1,\bm{u}}}{\sqrt{K(s)}}+ \lambda \frac{\sigma_{H}(\bm{e})_{1,\bm{v}}}{\sqrt{H}}\right),
\end{align*}
The above inequalities, combined with Lemma~\ref{lem:K(s)bounds} and with the inequality $\|f-\tilde{f}\|_{L^2_\nu} \leq \|f-f_{\Lambda}\|_{L^2_\nu} + \|\bm{d}\|_2$ yield  \eqref{thm:WLADLASSOrecovery:eq:L2est}. This concludes the proof.
\end{proof}

\section{Conclusions}
\label{sec:concl}

In the context of sparse high-dimensional approximation from pointwise samples, we have considered four decoders for weighted $\ell^1$ minimization: weighted quadratically-constrained basis pursuit, weighted LASSO, weighted square-root LASSO, and weighted LAD-LASSO. We have compared these decoders from the  theoretical and the numerical perspectives, focusing on the case where the samples are corrupted by unknown error (such as truncation, discretization, and numerical error). 

On the theoretical side, we have proved uniform robust recovery guarantees for all the decoders considered, showing that they achieve the best $s$-term approximation error in lower sets, provided a suitable lower bound on the sample complexity and up to the error level (see \S\ref{sec:rec_err} and Theorems~\ref{thm:WQCBPerror-blind}, \ref{thm:WLASSOrecovery}, \ref{thm:WSRLASSOrecovery}, and \ref{thm:WLADLASSOrecovery}). Our analysis suggests optimal strategies for the choice of the respective tuning parameters (see equations \eqref{eq:optimaleta}, \eqref{eq:lambdaWLASSO}, \eqref{eq:lambdaWSRLASSO}, and \eqref{eq:lambdaWLADLASSO}). 
%An open theoretical question is to characterize the decay of the best $s$-term approximation in lower sets with respect to $s$ when $f$ belongs to a suitable approximation class.

From the numerical viewpoint, we have compared the decoders' performance on a synthetic example, were $f$ is explicitly available (\S\ref{sec:par_vs_err} \& \ref{sec:m_vs_err}), and on a more applicable case, where the function is defined as the quantity of interest of a parametric ODE or PDE with random inputs (\S\ref{sec:UQ}). The first set of experiments corroborates the robustness and the reliability of the decoders under examination and confirms the optimality of the choices of the tuning parameters suggested by the theory.  Conversely, for the latter examples, our numerical results suggest that careful parameter tuning may not always be needed, possibly due to the structured nature of the error arising in these types of problems (see \S \ref{ss:discussion}).  A proper theoretical understanding of this phenomenon is still an open issue. Moreover, devising suitable strategies to incorporate some \emph{a priori} information on $f$ to the proposed approximation scheme is another open problem. For example, when $f$ is known to have more variability along certain axial directions than others, one might replace the hyperbolic cross with some anisotropic set $\Lambda$ or modify the weights $\bm{u}$ used in the decoders. Analyzing these issues is beyond the scope of this paper and it is left to future work.

%
%The experiments corroborate the robustness and the reliability of the decoders under examination and confirm the optimality of the choices of the tuning parameters suggested by the theory. However, when the error has not zero mean (like in the parametric PDE application) the optimal choices of the tuning parameters \eqref{eq:optimaleta},\eqref{eq:lambdaWLASSO}, \eqref{eq:lambdaWSRLASSO}, and \eqref{eq:lambdaWLADLASSO} seem to hold as inequalities reather than as equalities (see Fig.~\ref{fig:neg_error}). A proper theoretical understanding of this phenomenon is still an open issue.

Finally, we remark the following.  In light of the comparison made in this paper, we suggest that the square-root LASSO should be the decoder of choice out of those studied (apart from the case of sparse corruptions, where the best choice is of course the weighted LAD-LASSO). Indeed, this decoder has the same theoretical guarantees as the other ones and exhibits comparable (and sometimes better) numerical performance. Yet, the optimal choice of its tuning parameter does not depend on the unknown error. Although this fact is well-known in the statistics community, we think that it has not been fully exploited in high-dimensional function approximation and in compressed sensing. We hope that this paper could convince the reader to consider the square-root LASSO as a valid alternative to quadratically-constrained basis pursuit and LASSO in their future investigations.

\section*{Acknowledgements}
BA, AB and SB acknowledge the Natural Sciences and Engineering Research Council of Canada through grant 611675 and the Alfred P. Sloan Foundation and the Pacific Institute for the Mathematical Sciences (PIMS) Collaborative Research Group ``High-Dimensional Data Analysis''. SB acknowledges the support of the PIMS Post-doctoral Training Center in Stochastics. The authors are grateful to Claire Boyer, John Jakeman, Richard Lockhart, Akil Narayan, and Clayton G. Webster for interesting and fruitful discussions.

%\GR{BEN: note to self -- add grants}


\small
\bibliographystyle{plain}
\bibliography{library}

\end{document}